\documentclass[oneside,english,a4wide]{amsart}
\usepackage[latin9]{inputenc}
\usepackage[a4paper]{geometry}
\usepackage{babel}
\usepackage{mathtools}
\usepackage{amstext}
\usepackage{amsthm}
\usepackage{amssymb}
\usepackage{enumerate}
\usepackage{esint}
\usepackage{graphicx}
\usepackage{bbm}
\usepackage{dsfont}
\usepackage[all]{xy}
\usepackage[unicode=true,
 bookmarks=true,bookmarksnumbered=true,bookmarksopen=true,bookmarksopenlevel=2,
 breaklinks=false,pdfborder={0 0 1},backref=false,colorlinks=false]
 {hyperref}
\hypersetup{
 colorlinks,linkcolor=blue,anchorcolor=blue,citecolor=blue}
\usepackage{breakurl}
\makeatletter
\numberwithin{equation}{section}
\numberwithin{figure}{section}
\theoremstyle{plain}
\newtheorem{thm}{\protect\theoremname}[section]
  \theoremstyle{plain}
  \newtheorem{prop}[thm]{\protect\propositionname}
  
  \newtheorem{corollary}[thm]{Corollary}
  \theoremstyle{definition}
  
  \theoremstyle{plain}
  \newtheorem{lem}[thm]{\protect\lemmaname}
  \newtheorem{lemma}[thm]{Lemma}

  \newtheorem{remark}[thm]{Remark}
  \newtheorem{que}[thm]{Question}

\theoremstyle{definition}
\newtheorem{definition}[thm]{Definition}
\newtheorem*{claim*}{Claim}
\newtheorem{example}[thm]{Example}

\theoremstyle{definition}
\newtheorem*{defn*}{\protect\definitionname}

\usepackage{babel}
\usepackage{amscd}
\usepackage{comment}
\makeatother

  \providecommand{\definitionname}{Definition}
  \providecommand{\lemmaname}{Lemma}
  \providecommand{\propositionname}{Proposition}
  
\providecommand{\theoremname}{Theorem}

\newcommand{\Rmnum}[1]{\expandafter\@slowromancap\romannumeral #1@}
\makeatother

\newcommand{\real}{{\mathbb R}}

\newcommand{\8}{\infty}

\newcommand{\be}{\begin{eqnarray*}}
	\newcommand{\ee}{\end{eqnarray*}}
\newcommand{\beq}{\begin{equation}}
\newcommand{\eeq}{\end{equation}}
\newcommand{\beqn}{\begin{equation*}}
\newcommand{\eeqn}{\end{equation*}}
\newcommand{\bs}{\begin{split}}
\newcommand{\es}{\end{split}}

\renewcommand{\a}{\alpha}
\renewcommand{\b}{\beta}

\begin{document}

\title[Quantitative ergodic theorems for actions of groups...]{Quantitative ergodic theorems for actions of groups of polynomial growth}
%\thanks{{\it 2000 Mathematics Subject Classification:} Primary: 42B20, 42B25. Secondary: 46E30}

\author{Guixiang Hong}
\address{School of Mathematics and Statistics, Wuhan University, Wuhan 430072 and Hubei Key Laboratory of Computational Science, Wuhan University, Wuhan 430072, China}
\email{guixiang.hong@whu.edu.cn}

\author{Wei Liu}
\address{School of Mathematics and Statistics, Wuhan University, Wuhan 430072, China}
\email{Wl.math@whu.edu.cn}

\date{}
\begin{abstract}
We strengthen the maximal ergodic theorem for actions of groups of polynomial growth to a form involving jump quantity, which is the sharpest result among the family of variational or maximal ergodic theorems. As a consequence, we deduce in this setting the quantitative ergodic theorem, in particular, the upcrossing inequalities with exponential decay. The ideas or techniques involve probability theory, non-doubling Calder\'on-Zygmund theory, almost orthogonality argument and some delicate geometric argument involving the balls and the cubes on the group equipped with a not necessarily doubling measure.
\end{abstract}
\maketitle

\bigskip
%%%%%%%%%%%%%%%%%%%%%%%%%%%%%%%%%%%%%%%%%%%%%%%%%%%%%%%%%%%%%%%%%%%%%%%%%%%%%%%%%%%%%%%%%%%%%%%%%%%%%%%%%%
%%%%%%%%%%%%%%%%%%%%%%%%%%%%%%%%%%%%%%%%%%%%%%%%%%%%%%%%%%%%%%%%%%%%%%%%%%%%%%%%%%%%%%%%%%%%%%%%%%%%%%%%%%
\section{Introduction}\label{ST1}
%%%%%%%%%%%%%%%%%%%%%%%%%%%%%%%%%%%%%%%%%%%%%%%%%%%%%%%%%%%%%%%%%%%%%%%%%%%%%%%%%%%%%%%%%%%%%%%%%%%%%%%%%%
%%%%%%%%%%%%%%%%%%%%%%%%%%%%%%%%%%%%%%%%%%%%%%%%%%%%%%%%%%%%%%%%%%%%%%%%%%%%%%%%%%%%%%%%%%%%%%%%%%%%%%%%%%
%It is obvious that this result was established along some sequence of set $U_t$ choice.We are concerned in this paper with the ergodic theorems and variational inequalities for group actions on $L^p$-spaces.On the other hand, a number of authors have established pointwise ergodic theorems for group actions.
\subsection{Background and main results}
In the past few decades, a great deal of significant results related to the pointwise ergodic theorems for group actions have been established. The earliest pointwise ergodic theorems, to our knowledge, was obtained by Birkhoff~\cite{Birkhoff31}, where he established pointwise ergodic theorems for one-parameter flow (such as a translation group on $\mathbb{R}$ or $\mathbb{Z}$).  Wiener~\cite{Wiener39} extended Birkhoff's result to the case of several commuting flows. Even further, these pointwise ergodic results were generalized by Calder\'{o}n~\cite{Cal53} for an increasing family of compact symmetric neighborhoods of the identity satisfying doubling condition, which is abundant on non-commutative groups with polynomial volume growth. Calder\'{o}n's works~\cite{Cal53, Cal69} motivate further research on pointwise ergodic theorems, such as \cite{Bewley71,Chatard70, Coifman-Weiss76, Emerson74, Herz71, Hong-Liao-Wang17, Tempelman67, Tempelman72}. In particular,
 Breuillard \cite{Breuillard14} (see also Tessera  \cite{Tessera07} ) showed that the sequence of balls with respect to any fixed word metric on groups of polynomial growth satisfy the doubling condition and are asymptotically invariant, and thus established the corresponding pointwise ergodic theorem; actually these results apply to more general metrics such as the periodic pseudodistances as defined in \cite{Breuillard14} (and recalled in Section~\ref{ST7}).  This settled a long-standing problem in ergodic theory since Calder\'on's classical paper \cite{Cal53} in 1953.
Lindenstrauss~\cite{Lindenstrauss01} established the pointwise ergodic theorem for tempered F{\o}lner sequences; this result resolves the problem of the existence of a F{\o}lner sequence which satisfies the pointwise ergodic theorem on an arbitrary amenable group. For more details we refer the reader to the survey works~\cite{AAB+, Nevo06}.

In this paper, we aim at establishing the quantitative pointwise ergodic theorems for actions of polynomial growth groups in terms of the following jump quantity. Given a sequence of measurable functions $\{\a_r(x):r>0\}$ and $\lambda>0$,  the $\lambda$-jump function of the sequence $\a=\{\a_r(x):r\in \mathcal I\}$ is defined by
\begin{equation*}%\label{jump-function}
   \mathcal{N}_{\lambda}(\a)(x)=\sup\{N|\exists~r_0<r_1<\cdots<r_N,r_i\in\mathcal{I}:\min_{0<i\le N}|\a_{r_{i}}(x)-\a_{r_{i-1}}(x)|>\lambda\}.
\end{equation*}
where $\mathcal I$ is a subset of $(0,\8)$ and the supremum is taken over all finite increasing sequences in $\mathcal I$.

\bigskip

Let $G$ be a locally compact group equipped with a {measure $m$, and let $d$ be a metric on $G$. For $r>0$ and $h\in G$, we define the ball $B(h,r)=\{g\in G : d(g, h) \leq r\}$, and we will write it simply $B_r$ when $x=e$ (the identity in $G$).} Let $r_0>0$. Let $\epsilon\in(0,1]$, we say that $(G,d,m)$ satisfies the $(\epsilon,r_0)$-annular decay property if {there exists a constant $K>0$ such that for all $h\in G$, $r\in (r_0,\8)$ and $s\in (0,r]$,
\begin{equation}\label{decay property}
m(B(h,r+s))-m(B(h,r))\le K\bigg(\frac{s}{r}\bigg)^{\epsilon}m(B(h,r)).
\end{equation}}
{Let $D_0>0$, we call that $(G,d)$ satisfies the $(D_0,4r_0)$}-geometrically doubling property if for every $r\in(0,4r_0]$ and every ball $B(h,r)$, there are at most $D_0$ balls $B(h_i,r/2)$ such that
\begin{equation}\label{geo-doubling}
  B(h,r)\subseteq \bigcup_{1\le i\le{D_0}}B(h_i,r/2).
\end{equation}
Let $p\in[1,\8)$ and $f\in L^p(G,m)$, {we consider the following averages
\begin{equation}\label{averaging operator1}
  A^\prime_rf(h)=\frac{1}{m(B(h,r))}\int_{B(h,r)}f(g)dm(g).
\end{equation}}

One of the main result of this paper is the following theorem.

\begin{thm}\label{main-thm1}
Assume that $(G,d,m)$ satisfies~\eqref{decay property} and~\eqref{geo-doubling}.  Let $\mathbf{A}^\prime=\{A^\prime_r:r\ge r_0\}$ be the sequence of averaging operators given  by~\eqref{averaging operator1}. Then the following assertions hold true.
\begin{enumerate}[\noindent]
  \item \emph{(i)}~For any $p\in(1,\8)$, $\lambda\sqrt{\mathcal{N}_{\lambda}(\mathbf{A}^\prime)}$ is of strong type $(p,p)$ uniformly in $\lambda>0$, that is, there exists a constant $c_{p}>0$ such that
  \begin{equation*}%\label{}
    \sup_{\lambda>0}\|\lambda\sqrt{\mathcal{N}_{\lambda}(\mathbf{A}^\prime f)(g)}\|_{L^p(G,m)}\le c_{p}\|f\|_{L^p(G,m)},\;\forall f\in L^p(G,m).
  \end{equation*}
  \item \emph{(ii)}~For $p=1$, $\lambda\sqrt{\mathcal{N}_{\lambda}(\mathbf{A}^\prime)}$ is of weak type $(1,1)$ uniformly in $\lambda>0$, that is, there exists a constant $c>0$ such that for any $\gamma>0$,
  \begin{equation*}%\label{}
    \sup_{\lambda>0}m\big(\{g\in G:\lambda\sqrt{\mathcal{N}_{\lambda}(\mathbf{A}^\prime f)(g)}>\gamma\}\big)\le\frac{c}{\gamma}\|f\|_{L^1(G,m)},\;\forall f\in L^1(G,m).
  \end{equation*}
  %\item \emph{(iii)}~For $q\in(2,\8)$, $V_q(\mathbf{A})$ is of weak type $(1,1)$ and strong type $(p,p)$ with $1<p<\infty$.
\end{enumerate}
\end{thm}

Conditions~\eqref{decay property} and~\eqref{geo-doubling} are related to the geometric structure of group $G$; there exist lots of examples such as the ones introduced in~\cite{Nevo06,Tessera07}, see Section~\ref{ST7} for more details. {Moreover, as a consequence of Theorem~\ref{main-thm1}, we obtain the quantitative ergodic theorems for actions of groups of polynomial growth.}

Let $(X,\Sigma,\mu)$ be a $\sigma$-finite measure space and $T$ an action of $G$ on the associated $L^p$-spaces $L^p(X,\mu)$, under some additional assumptions recalled in later sections. In particular, if $T$ is induced by a $\mu$-preserving measurable transformation $\tau$ on $X$, then $T$ extends to an isometric action on $L^p(X,\mu)$ for all $1\le p\leq \8$, given by $T_gf(x)=f(\tau_{g^{-1}}x)$. {Given an action $T$ and a ball $B_r$, the associated averaging operator is given by
 \begin{equation}\label{averaging operator}
  A_rf(x)=\frac{1}{m(B_r)}\int_{B_r}T_gf(x)dm(g).
\end{equation}
}

\begin{thm}\label{main-thm2}
%{\color{red}If we can not prove all the periodic metrics satisfy annular decay property, then state the result only for work metric} and $d$ is a word metric on $G$.
Assume that $G$ is of polynomial growth with a symmetric compact generating set $V$ and $d$ is the associated word metric on $G$, that is,
$$d(g,h)=\min\{n\in \mathbb N, \;g^{-1}h\in V^n\} .$$
{Let $m$ be a Haar measure.} Let $\mathbf{A}=\{{A}_r:r\in\mathbb{N}\}$  the corresponding sequence of averaging operators with respect to an action $T$.% given by~\eqref{averaging operator}.

\begin{enumerate}[\noindent]
 %  and $V_q(\mathbf{A})$ is of strong type $(p,p)$.
\item \emph{(i)}~If $T$ is an action induced by a measure-preserving measurable transformation, then $\lambda\sqrt{\mathcal{N}_{\lambda}(\mathbf{A})}$ is of weak type $(1,1)$ and of strong type $(p,p)$ for all $1<p<\infty$ uniformly in $\lambda>0$.
%$\lambda\sqrt{\mathcal{N}_{\lambda}(\mathbf{A})}$ is of weak type $(1,1)$ and strong type $(p,p)$ for all $p\in(1,\infty)$ uniformly in $\lambda>0$, and for $q\in(2,\8)$, $V_q(\mathbf{A})$ is of weak type $(1,1)$ and strong type $(p,p)$ for all $1<p<\infty$.
%\end{enumerate}% ({\color{red} one may relax this assumption})
\item \emph{(ii)}~If $T$ is a strongly continuous regular action of $G$ on $L^p(X,\mu)$ \emph{($1<p<\infty$)}, then  $\lambda\sqrt{\mathcal{N}_{\lambda}(\mathbf{A})}$ is of strong type $(p,p)$ uniformly in $\lambda>0$.
\end{enumerate}
\end{thm}
{The notions of regular action will be recalled in Subsection~\ref{strong-type}.}

{If we take $G$ to be the integer group $\mathbb Z$ and $d$ to be the usual word metric, then we recover the usual ergodic average $A_n = \frac1{2n+1}\sum^n_{k=-n} T^k$ for an automorphism $T$, as is treated in \cite{Bour89,JKRW98}. Moreover, from the definition of jump quantity, Theorems~\ref{main-thm2} imply that the underlying sequence of functions $A_rf$ converges almost everywhere as $r\rightarrow\infty$, that is the pointwise ergodic theorem.}

The jump quantity appeared first implicitly in \cite{PX88}, and then explicitly in \cite{JKRW98}, whose study was motivated by the research of variational inequality arising from probability thoery \cite{DL76}. For $q\in(0,\8]$, the $q$-variation (semi-)norm $V_q$ of a sequence
$\{\a_r(x):r\in(0,\8)\}$ of complex-valued functions is defined by
$$V_q(\a_r(x):r\in \mathcal I)=\sup_{\substack{0<r_0<\cdots<r_J\\ r_j\in \mathcal I}}\bigg(\sum_{j=0}^J|\a_{r_{j+1}}(x)-\a_{r_j}(x)|^q\bigg)^{1/q}.$$
The $\infty$-variation norm is nothing but equivalent to the maximal norm; moreover any $q$-variation norm dominates the maximal norm
\begin{equation*}%\label{control maximal norm}
  \sup_{r\in\mathcal I}|\a_r(x)|\le |\a_{r_1}(x)|+V_q(\a_{r}(x):r\in\mathcal I),~\forall~r_1>0.
\end{equation*}
As the jump quantity, finite $q$-variation norm with $q<\infty$ also deduces immediately the pointwise convergence of the underlying sequence of operators without density argument. This idea was firstly exploited by Bourgain~\cite{Bour89} to study the pointwise ergodic theorem for dynamical systems where the density argument is not available. Bourgain's work \cite{Bour89} inspired many studies on variational inequalities and ergodic theory, see \cite{JKRW98, JKRW03, Krause14, Kra-Mirek-Tro18, Mir-stein-Tro17, Zorin-Kranich15} and the references therein.

On the other hand, by Chebychev's inequality, the $q$-variation norm dominates the jump quanity
\begin{equation*}%\label{control maximal norm}
 \sup_{\lambda>0}\lambda(\mathcal N_\lambda(\a_r(x):r\in \mathcal I))^{1/q}\leq  V_q(\a_r(x):r\in \mathcal I).
\end{equation*}
However, it is shown in \cite{JW04, Lewko-Lewko12} that the $2$-variational inequality (and thus $q<2$) does not hold true in general. Conversely, for any family of linear operators $\mathcal T=(T_t)_{t\in\mathcal I}$, the mapping properties of $\lambda\sqrt{\mathcal{N}_{\lambda}(\mathcal T)}$ imply the corresponding $q$-variational inequalities for all $q>2$ by interpolation, see e.g. \cite{Bour89, Jones-Seeger-Wright08,Mirek-Stein-Zor20}. Thus for all $q>2$, the $q$-variational version of  Theorem~\ref{main-thm1} and~\ref{main-thm2}, in particular the maximal ergodic theorem, holds true.

Furthermore, restricted to bounded function, from Theorem  \ref{main-thm2}, we may deduce a kind of exponential decay estimates by using Vitali covering lemma and geometric argument.

\begin{thm}\label{exponential-estimate}
 Assume that $G$ is a group of polynomial growth with a symmetric finite generating set and $d$ is the resulting word metric. Let $T$ be an action induced by a measure-preserving measurable transformation  and $\mathbf{A}=\{A_r:r\in\mathbb{N}\}$ the sequence of averaging operators with respect to action $T$ given  by~\eqref{averaging operator}. Let $\lambda>0$. Then for every  $p\in[1,\8)$, {there are two constants {$\tilde{c}_1>0$ and $\tilde{c}_2\in (0,1)$} depending on $\lambda$, $p$ and group $G$ such that for any $f\in L^p(X,\mu)$ with $\|f\|_{L^{\8}(X,\mu)}\le 1$,}
\begin{equation*}%\label{upcrossings estimate}
  \mu\big(\{x\in X:\mathcal{N}_{\lambda}(\mathbf{A}f)(x)>n\}\big)\le \tilde{c}_1\tilde{c}_2^n\|f\|^p_{L^p(X,\mu)}.
\end{equation*}
\end{thm}

This result immediately yields the upcrossing inequalities with exponential decay. Recall that for two real numbers $a$ and $b$ with $b>a$, the upcrossings of {a family of real-valued functions} $\a=\{\a_r(x):r\in\mathcal{I}\}$,  $\mathcal{N}_{a,b}(\a)(x)$ is defined by
\begin{equation*}
  \sup\{N\in\mathbb{N}|\exists~s_1<r_1<\cdots<s_N<r_N, r_i,s_i\in\mathcal{I}~\textit{such~that}~\a_{s_l}(x)<a~\textit{and}~\a_{r_{l}}(x)>b\}.
\end{equation*}
Taking $\lambda=b-a$, it is easy to see that $\mathcal{N}_{a,b}(\a)(x)\leq 2\mathcal{N}_{\lambda/2}(\a)(x)$. Thus we obtain

\begin{corollary}
 Let $a$ and $b$ be two real numbers with $b>a$. Then for every  $p\in[1,\8)$, {there are two constants {$\tilde{c}_1>0$ and $\tilde{c}_2\in (0,1)$} depending on $\lambda$, $p$ and group $G$ such that for any real-valued function $f\in L^p(X,\mu)$ with $\|f\|_{L^{\8}(X,\mu)}\le 1$,}
\begin{equation}\label{upcrossings-estimate}
  \mu\big(\{x\in X:\mathcal{N}_{a,b}(\mathbf{A}f)(x)>n\}\big)\le \tilde{c}_1\tilde{c}_2^n\|f\|^p_{L^p(X,\mu)}.
\end{equation}

\end{corollary}
%Let $G$ be a locally compactly group of polynomial growth with a discrete symmetric compact generating set $V$. We denote by $dm$ the counting measure on $G$. Let $T$ be an action of $G$ induced by a $\mu$-preserving measurable transformation $\tau$ on $X$. Let $\mathbf{A}f=(A_rf:r\in\mathbb{N})$ be the sequence of averaging operators given  by~\eqref{averaging operator}.
%The constant $1/\sqrt{n}$ in the above inequality may cannot be improved, but we can establish the exponential estimate as follows. %(this terminology was introduced by )For $p\in(1,\8)$, let $T$ be a strongly continuous and uniformly bounded action of $G$ on $L^p(X)$ such that $T_g$ is a positive map for each $g\in G$ and for $p=1$,

A variant of (\ref{upcrossings-estimate}) for non-negative functions has been recently proved by Moriakov~\cite{Moriakov18}, as a generalization of  Kalikow and Weiss's exponential estimate for $G=\mathbb Z$ \cite{{Kalikow-Weiss99}}. Note that, our method here is completely different from the one used by Moriakov that is based on a generalization of Vitali covering theorem. On the other hand, Kalikow and Weiss's estimate was motivated by Bishop's fundamental result \cite{Bishop67}  and the similar estimates in the martingale setting \cite{Doo53, Dub62}. Recently the upcrossing inequalities with exponential decay have been extended to stationary process and find many applications to ergodic theory and information theory, see e.g. \cite{Hoc09}. % moreover, it is unclear that whether their upcrossing estimates can be extended to actions beyond automorphisms as in Corollary {} (ii).

\subsection{Methods and more results}
The theorems rely on several key results obtained in this paper.

The first step to show Theorem~\ref{main-thm1} is now standard, that is, to control the jump quantity by the `dyadic' jump  and the short variation operator, see e.g. \cite{Jones-Seeger-Wright08}. More precisely, let $f\in L^p(G,m)$, %and $\mathbf{A}^\prime f=(A^\prime_rf:r\ge r_0)$  be a sequence of averaging operators given  by~\eqref{averaging operator1},
we dominate the jump quantity $\lambda\sqrt{\mathcal{N}_{\lambda}(\mathbf{A}^\prime f)}$ in the following way,
  \begin{equation*}%\label{jump operator}
  \lambda\sqrt{\mathcal{N}_{\lambda}(\mathbf{A}^\prime f)}\le 2\lambda\sqrt{\mathcal{N}_{\lambda/6}(A^\prime_{\delta^n}f:n> n_{r_0})}+16\bigg(\sum_{{n\geq n_{r_{0}}}}V_2(A^\prime_{r}f:r\in[\delta^{n},\delta^{n+1}))^2\bigg)^{1/2},
\end{equation*}
where $\delta>1$ is a constant depending on $G$ that will be determined in Proposition~\ref{dyadic cube} and $n_{r_0}$ is the unique integer such that $\delta^{n_{r_0}}<r_0\leq \delta^{n_{r_0}+1}$.
%the dyadic jump $\mathcal{N}^{dyad}_\lambda(\mathbf{A}^\prime f)$ is defined as
%\begin{equation*}%N|\exists~r_0<r_1<\cdots<r_N,r_i\in\mathcal{I}:\min_{0<i\le N}
  %\mathcal{N}^{dyad}_\lambda(\mathbf{A}^\prime f)(x)=\sup\{J|\exists~n_0<n_1<\cdots<n_J,n_i\in\mathbb{N}:\min_{0<i\le J}|A_{\delta^{n_{i}}}f(x)-A_{\delta^{n_{i-1}}}f(x)|\ge \lambda\}
%\end{equation*}
Moreover, compared with the resulting martingale $\mathbb{E}f=(\mathbb{E}_nf:n\in\mathbb{N})$ (see Definition~\ref{martingale sequence}), the dyadic jump is controlled by
\begin{equation*}
 %\lambda\sqrt{\mathcal{N}^{dyad}_{\lambda/3}(\mathbf{A}^\prime f)} \le
 96\sqrt{2}\bigg(\sum_{n> n_{r_0}}|A^\prime_{\delta^n}f-\mathbb{E}_nf|^2\bigg)^{1/2}+ 2\sqrt{2}\lambda\sqrt{\mathcal{N}_{\lambda/24}(\mathbb{E}_nf:n>n_{r_0})}.
\end{equation*}
Thus we obtain the following pointwise estimate,
\begin{equation}\label{deal with variational operator}
  \begin{split}
    \lambda\sqrt{\mathcal{N}_{\lambda}(\mathbf{A}^\prime f)}&\le 96\sqrt{2}\bigg(\sum_{n> n_{r_0}}|A^\prime_{\delta^n}f-\mathbb{E}_nf|^2\bigg)^{1/2}\\
     &+16\bigg(\sum_{n\ge n_{r_0}}V_2(A^\prime_{r}f:r\in[\delta^{n},\delta^{n+1}))^2\bigg)^{1/2}
     +2\sqrt{2}\lambda\sqrt{\mathcal{N}_{\lambda/24}(\mathbb{E}_nf:n>n_{r_0})}.
  \end{split}
\end{equation}% Theorem~\ref{varaitional ineq in homo-space}
This inequality implies that in order to bound $\lambda\sqrt{\mathcal{N}_{\lambda}(\mathbf{A}^\prime f)}$, it suffices to estimate the three parts on the right hand, respectively. Note that the boundedness of the jump operator $\lambda\sqrt{\mathcal{N}_{\lambda}(\mathbb{E})}$ was proved in~\cite{JKRW98,PX88}, see Lemma~\ref{lem:jump}; and so for our purposes we only need to estimate the first two terms on the right hand side of inequality~\eqref{deal with variational operator}. % (see Lemma~\ref{lem:lepingle}).% Pisier-Xu~\cite{PX88} and Bourgain~\cite{Bour89}
For abbreviation, we denote
$$S(f):=\bigg(\sum_{n> n_{r_0}}|A^\prime_{\delta^n}f-\mathbb{E}_nf|^2\bigg)^{1/2}$$
and
$$SV(f):=\bigg(\sum_{n\ge n_{r_0}}V_2(A^\prime_{r}f:r\in[\delta^{n},\delta^{n+1}))^2\bigg)^{1/2}.$$
For these two operators, we will show more results including ($L^\infty$, BMO) estimate which is also necessary to obtain the result for $2<p<\infty$ (see Section~\ref{ST5} for the  definition of BMO). In what follows, $L^\8_c$ denotes the compactly supported $L^\8$ functions.

\begin{thm}\label{the estimate of square function}
%Let $(G,d)$ be a locally compact group endowed with a invariant metric. Let $m$ be a right Haar measure on $G$.
Assume that $(G,d,m)$ satisfies conditions~\eqref{decay property} and~\eqref{geo-doubling}, then
\begin{enumerate}[\noindent]
  \item\emph{(i)}~for every $p\in(1,\8)$, there exists a constant $c_{p}>0$ such that for every $f\in L^p(G,m)$,
  \begin{equation}\label{strong-type inequalities of square function}
  \|S(f)\|_{L^p(G,m)}\le c_{p}\|f\|_{L^p(G,m)};
\end{equation}
  \item \emph{(ii)}~there exists a constant $c>0$ such that for every $f\in L^1(G,m)$,
  \begin{equation}\label{weak-type inequalities of square function}
  m\big(\{g\in G:S(f)(g)>\gamma\}\big)\le \frac{c}{\gamma}\|f\|_{L^1(G,m)},~\forall~\gamma>0,
\end{equation}
 and for every $f\in L^\8_c(G,m)$,
\begin{equation}\label{BMO-type inequalities of square function}
  \|S(f)\|_{BMO}\le c\|f\|_{L^{\8}_c(G,m)}.
\end{equation}
%where $L^{\8}_c(G,m)$ denotes the space that consists of all compact supported $L^{\8}(G,m)$ function.
\end{enumerate}
\end{thm}
The same results hold still true for the short variation operator $SV$.
\begin{thm}\label{the estimate of short variation}
%Let $(G,d)$ be a locally compact group endowed with a invariant metric. Let $m$ be a right Haar measure on $G$.
Assume that $(G,d,m)$ satisfies conditions~\eqref{decay property} and~\eqref{geo-doubling}, then
\begin{enumerate}[\noindent]
  \item \emph{(i)}~for every $p\in(1,\8)$, there exists a constant $c_{p}>0$ such that for every $f\in L^p(G,m)$,
  \begin{equation}\label{strong-type inequalities of short variation}
  \|SV(f)\|_{L^p(G,m)}\le c_{p}\|f\|_{L^p(G,m)};
\end{equation}
  \item \emph{(ii)}~there exists a constant $c>0$ such that for every $f\in L^1(G,m)$,
  \begin{equation}\label{weak-type inequalities of short variation}
  m\big(\{g\in G:SV(f)(g)>\gamma\}\big)\le \frac{c}{\gamma}\|f\|_{L^1(G,m)},~\forall~\gamma>0,
\end{equation}
 and for every $f\in L^\8_c(G,m)$,
\begin{equation}\label{BMO-type inequalities of short variation}
  \|SV(f)\|_{BMO}\le c\|f\|_{L^{\8}_c(G,m)}.
\end{equation}
%where $L^{\8}_c(G,m)$ denotes the space that consists of all compact supported $L^{\8}(G,m)$ function.
\end{enumerate}
\end{thm}

%Our strategy of the proofs of Theorem~\ref{the estimate of square function} and~\ref{the estimate of short variation} can be described as follows. We first prove the $L^p(X)$ estimates with $1<p\le 2$. The proof is based on the general almost orthogonality principles, Calder\'{o}n-Zygmund decomposition and some technical lemmas. We remark that the proof is heavily depending on the geometric structure of the balls. For case $p\in (2,\8)$,  we establish the BMO-estimates for the square function and the short variation operator. Then using the interpolation argument, we obtain the high-$L^p(X)$ estimates. Note that the property of invariant metric plays a key role in our proof. genera The key step to prove the $L^p$-boundedness of maximal inequalities is prove weak type (1,1) estiamte

%The method of establishing the jump  estimate is quite different than the maximal estimate. Comparing with the results of maximal inequalities, the proof of jump (or variational) estimate is more complicated since the jump (or variational) operator is unbounded from $L^\8$ to $L^\8$ (cf.~\cite{JW04}).  Our proof of Theorem~\ref{main-thm1} and Theorem~\ref{main-thm2} need overcome the gaps in twofold.  On the one hand, there are two new difficulties in the transfer principle, namely Theorem~\ref{thm:trans}: the first one is the jump quantity is not a norm a priori, but there exists an equivalent norm by~\cite{Mirek-Stein-Zor20}; the second one, for $p>1$ the action $T_g$ satisfies a general condition such as Tempelman~\cite{Templeman15}'s or beyond. On the other hand,

The above two theorems, {parallel to Theorem 2.3 and 2.4 of \cite{GXHTM17}}, are not a surprise (cf. \cite{GXHTM17}). However, note that under conditions~\eqref{decay property} and~\eqref{geo-doubling}, $(G,d,m)$ is not necessarily a doubling metric measure space; this induces several new difficulties in showing Theorems \ref{the estimate of square function} and  \ref{the estimate of short variation}, such as, that the `dyadic cubes' constructed in Proposition \ref{dyadic cube} may not admit the small boundary property (cf.~\cite[Theorem 11]{Christ90}) and that the standard Calder\'{o}n-Zygmund decomposition for homogeneous space is not enough \emph{etc.}. For these reasons, we have to pay great attention to the geometric argument involving the cubes and the balls and explore the non-doubling Calder\'on-Zygmund theory \emph{etc.}.

\bigskip

To deduce Theorem \ref{main-thm2} from Theorem \ref{main-thm1},  there needs two transference principles: the first one is for actions induced by measure-preserving measurable transforms, while another one is for regular actions. We refer the reader to~\eqref{regular} for the definition of regular actions and the related constant $\|\cdot\|_r$.
{
A group $G$ is called amenable if it admits a F${\o}$lner sequence $(F_n)_{n\in \mathbb N}$, that is, for every $g\in G$,
\begin{align}\label{asymptotically invariant}
\lim_{n\rightarrow\infty} \frac{m((F_n g) \bigtriangleup F_n)}{m(F_n)}= 0,
\end{align}
where $\bigtriangleup$ denotes the usual symmetric difference of two sets. }

\begin{thm}\label{thm:trans}%~\eqref{doubling condition of ball}
 Let $G$ be an amenable group equipped with invariant metric $d$ and a right Haar measure $m$, and $T$ an action on $L^p(X,\mu)$. Let $\mathbf{A}^\prime=\{A^\prime_r:r\in\mathcal I\}$ and $\mathbf{A}=\{A_r:r\in\mathcal I\}$  be two sequences of averaging operators given  by~\eqref{averaging operator1} and~\eqref{averaging operator}, respectively.
 \begin{enumerate}[\noindent]
 \item \emph{(i)} Let $p\in[1,\infty)$. If $T$ is an action induced by a measure-preserving measurable transformation and $\lambda\sqrt{\mathcal{N}_{\lambda}(\mathbf{A}^\prime)}$ is of weak (resp. strong) type $(p,p)$ uniformly in $\lambda>0$, then $\lambda\sqrt{\mathcal{N}_{\lambda}(\mathbf{A})}$ is of weak (resp. strong) type $(p,p)$ uniformly in $\lambda>0$, {and moreover $\sup_{\lambda>0}\|\lambda\sqrt{\mathcal{N}_{\lambda}(\mathbf{A})}\|_{L^p\rightarrow L^{p,\8}}\le \sup_{\lambda>0}\|\lambda\sqrt{\mathcal{N}_{\lambda}(\mathbf{A}^\prime)}\|_{L^p\rightarrow L^{p,\8}}$ (resp. $\sup_{\lambda>0}\|\lambda\sqrt{\mathcal{N}_{\lambda}(\mathbf{A})}\|_{L^p\rightarrow L^{p}}\le \sup_{\lambda>0}\|\lambda\sqrt{\mathcal{N}_{\lambda}(\mathbf{A}^\prime)}\|_{L^p\rightarrow L^{p}}$)}.

 \item \emph{(ii)} Let $p\in(1,\infty)$. If $T$ is a strongly continuous {regular} action of $G$ on $L^p(X)$ and $\lambda\sqrt{\mathcal{N}_{\lambda}(\mathbf{A}^\prime)}$ is of strong type $(p,p)$ uniformly in $\lambda>0$, then $\lambda\sqrt{\mathcal{N}_{\lambda}(\mathbf{A})}$ is of strong type $(p,p)$ uniformly in $\lambda>0$, {and moreover there exists a constant $c_p>0$ such that $\sup_{\lambda>0}\|\lambda\sqrt{\mathcal{N}_{\lambda}(\mathbf{A})}\|_{L^p\rightarrow L^{p}}\le c_p\sup_{h\in G}\|T_h\|^2_r\sup_{\lambda>0}\|\lambda\sqrt{\mathcal{N}_{\lambda}(\mathbf{A}^\prime)}\|_{L^p\rightarrow L^p}$}. %{\color{red} modify this result with explicitly constant}
 % such that $T_g$ is a positive map for each $g\in G$. Assume that for every function $F(g)\in L^p(G,m)$, there exists a constant $c_p>0$ such that
    %  \begin{equation*}
       %\sup_{\lambda}\|\lambda\sqrt{\mathcal{N}_{\lambda}(\mathbf{A}^\prime F)(g)}\|_{L^p(G,m)}\le c_{p}\|F\|_{L^p(G,m)},
      %\end{equation*}
  %then for every function $f(x)\in L^p(X,\mu)$, there exists a constant $c^\prime_p>0$ such that
   %\begin{equation*}
      % \sup_{\lambda}\|\lambda\sqrt{\mathcal{N}_{\lambda}(\mathbf{A}f)(x)}\|_{L^p(X,\mu)}\le c^\prime_{p}\|f\|_{L^p(X,\mu)}.
      %\end{equation*}
  %item \emph{(ii)}~For $p=1$, let $T$ be an action induced by a measure-preserving measurable transformation. Suppose that for every function $F(g)\in L^1(G,m)$, there exists a  constant $c>0$ such that for every $\gamma>0$,
      %\begin{equation*}
      %\sup_{\lambda} m\{g\in G:\lambda\sqrt{\mathcal{N}_{\lambda}(\mathbf{A}^\prime F)(g)}>\gamma\}\le \frac{c}{\gamma}\|F\|_{L^1(G,m)},
      %\end{equation*}
  %then for every function $f(x)\in L^1(X,\mu)$ and $\gamma>0$, there exists a constant $c^\prime>0$ such that
   %\begin{equation*}
      %\sup_{\lambda} \mu\{x\in X:\lambda\sqrt{\mathcal{N}_{\lambda}(\mathbf{A}f)(x)}>\gamma\}\le \frac{c^\prime}{\gamma}\|f\|_{L^1(X,\mu)}.
      %\end{equation*}
  \end{enumerate}
\end{thm}

%{The notions of regular action and norm $\|\cdot\|_r$ will be recalled in Subsection~\ref{strong-type}.}%\ref{Strong type inequalities.}

It is a little bit surprising that Theorem \ref{thm:trans}(i) holds true due to the fact that $\lambda\sqrt{\mathcal{N}_{\lambda}(\cdot)}$ is \emph{a priori} not a norm; {while to prove Theorem \ref{thm:trans}(ii), in addition to the use of the fact that $\sup_{\lambda>0}\|\lambda\sqrt{\mathcal{N}_{\lambda}(\cdot)}\|_p$ is equivalent to a norm (cf. \cite{Mirek-Stein-Zor20}) for $p>1$, the argument is subtle since the appearance of the supremum over $\lambda$ outside $L^p$ norm.}

\bigskip

\bigskip

An outline of this paper is as follows. In Section~\ref{ST2}, we first recall some necessary preliminaries concerning the definition of `dyadic cubes', which was constructed by Hyt\"{o}nen and Kairema~\cite{Hyt-Anna12} in metric space; and then present some technical lemmas, which are essential in showing Theorems~\ref{the estimate of square function} and~\ref{the estimate of short variation} . The last two theorems will be proved in Section~\ref{ST3}-\ref{ST5}. In Section~\ref{ST6}, we prove the transfer principles, namely Theorem~\ref{thm:trans}. In Section~\ref{ST7}, we discuss the $(\epsilon,r_0)$-annular decay property, providing examples and formulating problems; in particular, the balls with respect to a word metric over groups of polynomial growth satisfies conditions~\eqref{decay property},~\eqref{geo-doubling} and~\eqref{asymptotically invariant}, and thus we obtain Theorem~\ref{main-thm2}. %We will give some examples of group with invariant metric which satisfy the $(\epsilon,r_0)$-annular decay property such as a group of polynomial volume growth. By these discussions, we can establish Theorem~\ref{main-thm2}.
 Finally, in Section~\ref{ST8}, we give a proof of Theorem~\ref{exponential-estimate}. %that is we extend the exponential estimate of Kalikow and Weiss~\cite{Kalikow-Weiss99} to a group  of polynomial growth with a discrete symmetric compact subset.we provide an application of jump averaging operator on ergodic theory in which

Throughout this paper, we always denote $C$ by a positive constant respectively that may vary from line to line, while $c_p$ denote positive constant possibly depending on the subscripts.
%%%%%%%%%%%%%%%%%%%%%%%%%%%%%%%%%%%%%%%%%%%%%%%%%%%%%%%%%%%%%%%%%%%%%%%%%%%%%%%%%%%%%%%%%%%%%%%%%%%%%%%%%%%
%%%%%%%%%%%%%%%%%%%%%%%%%%%%%%%%%%%%%%%%%%%%%%%%%%%%%%%%%%%%%%%%%%%%%%%%%%%%%%%%%%%%%%%%%%%%%%%%%%%%%%%%%%%
\section{Preliminaries and some technical lemmas}\label{ST2}

In this section, we first recall the `dyadic cubes' constructed on the measure space $(G,d,m)$ satisfying conditions~\eqref{decay property} and~\eqref{geo-doubling} which might not be a measure doubling metric space, and the resulting martingale inequalities will play a key role in the probabilistic approach to jump or variational inequalities. We then collect several technical lemmas which involve or rely on the estimates of the boundaries of  `dyadic cubes' or balls or certain configuration among them; these estimates are subtle, and thus we will set and fix in the whole paper several constants such as $k_1,n_0,n_1, c_0, C_0, C_1, L_0,L_1,\delta$ which depend on conditions~\eqref{decay property} and~\eqref{geo-doubling} and the construction. These preliminary results or technical lemmas, collected in such a way, will facilitate much the presentation of the proof of Theorems \ref{the estimate of square function} and~\ref{the estimate of short variation}.

We will exploit the system of `dyadic cubes' constructed in the setting of geometrically doubling metric measure spaces, which means that each ball of radius $r>0$ can be covered by fixed finitely many balls of radius $r/2$
 (cf. \cite[Theorem 2.2]{Hyt-Anna12}). It is not difficult to see that
 the measure space $(G,d,m)$ satisfying conditions~\eqref{decay property} and~\eqref{geo-doubling} is geometrically doubling, but which might not be measure doubling. Indeed, by a simple computation, the $(\epsilon,r_0)$-annular decay property---condition~\eqref{decay property}---implies the measure doubling condition for large balls, that is, for every $x\in G$ and $r_0<r\le R<\8$,%following estimate as~\eqref{doubling}% following inequality
\begin{align}\label{int}
\frac{m(B(x,R))}{m(B(x,r))}\leq (K+1)\Big(\frac{R}{r}\Big)^{\epsilon}.
\end{align}
%Let $D>0$, recall that the $(G,d)$ satisfies the $(D,2r_0)$-geometrically doubling property if for every $r\in(0,2r_0]$, there are at most $D$ balls $B(x_i,r/2)$ such that
%\begin{equation*}%\label{geo-doubling}
%  B_r\subseteq \bigcup_{1\le i\le D}B(x_i,r/2).
%\end{equation*}
This yields the geometrically doubling condition for large balls, which, combined with condition~\eqref{geo-doubling}, deduces the geometrically doubling condition property of the space $(G,d,m)$, namely, each ball of radius $r>0$ in $G$ can be covered by no more than $D=\max\{D_0, [9^\epsilon(K+1)]+1\}$ balls of radius $r/2$.
%As explained in introduction, the martingale technique will be indispensable for our argument, we first introduce the  system of `dyadic cubes' on $(G,d,m)$ which was constructed in~\cite[Theorem 2.2]{Hyt-Anna12}.%, Hyt\"{o}nen and Kairemawhich extend Christ's result~\cite{Christ90}. %{\color{red} Modify the following proposition to a lemma with $k\in\mathbb Z$ as in Jone-Seeger-Wright, and accordingly the martingale }0<a_0<\frac{1}{4}Hyt\"{o}nen et al.~\cite{Hyt-Anna12} Christ~\cite{Christ90}Viewed as a metric space without any other assumption of measure
{Moreover, from the geometrically doubling condition, one can easily deduce the following property.
\begin{prop}\label{geometry-doubling}
Let $(G,d,m)$ satisfy the conditions of Theorem \ref{main-thm1}. Let $0<r\le R$, any ball $B(x,R)$ can be covered by no more than  $D^{\log_2[{R}/{r}]+1}$ balls of radius $r$.
\end{prop}
For more information about geometrically doubling property we refer the reader to~\cite{Coifman-Weiss71}.

\begin{prop}\label{dyadic cube}\cite[Theorem 2.2]{Hyt-Anna12}
Let $(G,d,m)$ satisfy the conditions of Theorem \ref{main-thm1}. Fix constants $0<c_0<C_0<\8$ and $\delta>1$ such that
$$18C_0\delta^{-1}\le c_0.$$
Let $I_k$ \emph{($k\in\mathbb Z$)} be an index set and $\{z_\alpha^k\in G:\alpha\in I_k,k\in\mathbb{Z}\}$ be a collection of points with the properties that
\begin{equation}\label{distance}
  d(z_\alpha^k,z_\beta^k)\ge c_0\delta^k~(\alpha\neq\beta),~\min_{\alpha}d(x,z_\alpha^k)<C_0\delta^k,~\forall~x\in G,~k\in\mathbb{Z}.
\end{equation}
Then there exist a family of sets $\big\{Q_\alpha^k\big\}_{\alpha\in I_k}$ associating with $\{z_\alpha^k\}_{\alpha\in I_k}$, and constants $a_0:=c_0/3$ and $C_1:=2C_0$ such that
\begin{enumerate}[\noindent]
\item\emph{(i)}~$\forall~k\in\mathbb{Z}$,~$\cup_{\alpha\in I_k} Q_\alpha^k=G$.
\item \emph{(ii)}~If $k\le l$ then either $Q_\alpha^k\subset Q_\beta^l$ or $Q_\alpha^k\cap Q_\beta^l=\emptyset$.%for every (k + 1; ), there is exactly one (k; )  (k + 1; ), called its parent;
\item \emph{(iii)}~For each $(k,\alpha)$ and each $k<n$ there is a unique $\beta$ such that $Q_\alpha^k\subset Q_\beta^n$, and for $n=k+1$, we call such $Q_\beta^{k+1}$ the parent of $Q_\alpha^{k}$.
\item \emph{(iv)}~$B(z_\alpha^k, a_0\delta^k)\subseteq Q_\alpha^k\subseteq B(z_\alpha^k, C_1\delta^k)$. %Denote by $d( Q_\alpha^k)$ the diameter of $ Q_\alpha^k$.%Each $Q_\alpha^k$  contains some ball $B(z_\alpha^k, a_0\delta^k)$ and the diam($Q_\alpha^k)\le C_1\delta^k$.
%\item \emph{(v)}~For every $k,~\alpha$ and $t>0$, $m\big\{x\in Q_\alpha^k:d(x,G\setminus Q_\alpha^k)\le t\delta^k\big\}\le C_2 t^\eta m(Q_\alpha^k)$.
\end{enumerate}
\end{prop}
{We remark that the geometrically doubling property ensures that the minimum of the second inequality in~\eqref{distance} is attained.}

For $k\in\mathbb Z$, let $\mathcal{F}_k$ be the $\sigma$-algebra generated by the `dyadic cubes' $\{Q_\alpha^{k}:\alpha\in I_k\}$.  We recall the following notions associated to the  (reverse) martingale theory.

\begin{definition}\label{martingale sequence}%[martingale]
 Let $f:G\rightarrow \mathbb{C}$ be a locally integrable function and $k\in\mathbb Z$, the conditional expectation of $f$ with respect
to $\mathcal{F}_k$ is defined by
\begin{equation}\label{martingale}
\mathbb{E}_kf(x)=\sum_{\alpha\in I_k}\frac{1}{m(Q_\alpha^k)}\int_{Q_\alpha^k}f(y)dm(y)\mathds{1}_{Q_\alpha^k}(x);
\end{equation}
the resulting martingale difference operator $\mathbb{D}_k$ is defined as%and $\mathbb{E}_{-1}$ is denoted by the identity operator. The martingale difference by
\begin{equation*}
\mathbb{D}_kf=\mathbb{E}_{k-1}f-\mathbb{E}_{k}f.
\end{equation*}
\end{definition}

We check at once $\mathbb{E}_k\circ \mathbb{E}_j=\mathbb{E}_{\max(j,k)}$ and that for $f\in L^2$, $f=\sum_{k\in\mathbb{Z}}\mathbb{D}_kf$ and $$\|\big(\sum_{k\in\mathbb{Z}}|\mathbb{D}_kf|^2\big)^{1/2}\|_{L^2}=\|f\|_{L^2}.$$

%It is easy to check that the operators $M_d$, $M^\sharp_d$ and $M$ are strong type $(p,p)$ and weak type $(1,1)$ for  measure $m$.

%As mentioned earlier, the $L^p$-boundedness of the jump operator of martingale was proved by Bourgain~\cite{Bour89} for $p\in(1,\8)$ as well as by Jones et al.~\cite{JKRW98} for $p=1$.%for the jump operator of martingale.%( obtained the result for $p\in(1,\8)$).% ((Pisier-Xu~\cite{PX88} and Bourgain~\cite{Bour89})).Let $\mathbb{E}=(\mathbb{E}_k:k\in\mathbb{Z})$ be a martingale sequence given by~\eqref{martingale}. The fundamental boundedness result concerning the variational inequalities for martingale is due to L\'{e}pingle~\cite{DL76} in which he proved $L^p$, $1<p<\8$ and weak type (1,1) bounds for $V_q(\mathbb{E})$ whenever $q>2$.  A simple proofs of L\'{e}pingle's results can be found in Pisier and Xu~\cite{PX88} and Bourgain~\cite{Bour89} by reducing matters to the corresponding jump inequalities.
Denote $\mathbb{E}=\{\mathbb{E}_k:k\in\mathbb Z\}$. We remark that the strong type $(p,p)$ with $1<p<\8$ and the weak type (1,1) estimates for operator $\lambda\sqrt{\mathcal{N}_{\lambda}(\mathbb{E})}$ were given implicitly in \cite{PX88} and explicitly in~\cite{JKRW98}. We state the results as follows.%We only exhibit the result of jump inequalities for martingale. %The result concerning jump inequalities for martingale is the following lemma.
\begin{comment}
\begin{lemma}\label{lem:lepingle}
Let $\mathbb{E}=(\mathbb{E}_k:k\in\mathbb{Z})$ be a martingale sequence given by~\eqref{martingale}. Let $q\in(2,\8)$, then
\begin{enumerate}[\noindent]
  \item\emph{(i)}~for $p\in(1,\8)$, there is a positive constant $C_{p,q}$ such that
  \begin{equation*}
    \|V_q(\mathbb{E}_kf:k\in\mathbb{Z})\|_{L^p(G,m)}\le C_{p,q}\|f\|_{L^p(G,m)},\forall~f\in L^p(G,m);
  \end{equation*}
  \item \emph{(ii)}~For $p=1$, there is a positive constant $C_q$ such that
  \begin{equation*}
    m\{x\in G:V_q(\mathbb{E}_kf(x):k\in\mathbb{Z})>\gamma\}\le\frac{C_q}{\gamma}\|f\|_{L^1(G,m)},~\forall~\gamma>0, f\in L^1(G,m).
  \end{equation*}
\end{enumerate}
\end{lemma}
\end{comment}

\begin{lemma}\label{lem:jump}
Let $\mathbb{E}=\{\mathbb{E}_k:k\in\mathbb Z\}$ be defined as above.
\begin{enumerate}[\noindent]
  \item\emph{(i)}~When $p\in(1,\8)$, there is a constant $c_p>0$ such that for all $f\in L^p(G,m)$,
  \begin{equation*}
    \sup_{\lambda>0}\|\lambda\sqrt{\mathcal{N}_{\lambda}(\mathbb{E}f)}\|_{L^p(G,m)}\le c_p\|f\|_{L^p(G,m)}.
  \end{equation*}
  \item \emph{(ii)}~For $p=1$, {there is a constant $c>0$ such that for every $\gamma>0$  and $f\in L^1(G,m)$,}
  \begin{equation*}
    \sup_{\lambda>0}m\big(\{x\in G:\lambda\sqrt{\mathcal{N}_{\lambda}(\mathbb{E}f)(x)}>\gamma\}\big)\le\frac{c}{\gamma}\|f\|_{L^1(G,m)}.
  \end{equation*}
\end{enumerate}
\end{lemma}

\bigskip

%From now on, we collect some technical lemmas for further convenient use.

Since $(G,d,m)$ might not be a measure doubling metric space, the small boundary property of the `dyadic cubes' constructed from \eqref{dyadic cube} (see e.g. \cite{Christ90}) does not hold in general. However, from the $(\epsilon,r_0)$ annular decay property, we do have some boundary property---Lemma \ref{boundary}, which will be enough for our purpose in the present paper.

Set
$$L_0=[\log_\delta(12/c_0)]+1, L_1=[\log_\delta(36r_0/c_0)]+1.$$
\begin{lemma}\label{boundary condition}
 Let $k,L\in\mathbb{Z}$  satisfy $L_0<L<k+L_0-L_1$ and $\alpha\in I_k$. Then we have
\begin{equation*}
 m\big(\{x\in Q_\alpha^k:d(x,G\setminus Q_\alpha^k)\le \delta^{k-L}\}\big)\le \frac{(K+1)^2}{(L-L_0+1)}\bigg(\frac{72C_0}{c_0}\bigg)^{2\epsilon}m(Q_\alpha^k).
\end{equation*}
\end{lemma}
%Comparing with Lemma 17 in~\cite{Christ90},  the above lemma not implies the smallness of the boundary region of the dyadic cubes, since for every $k$, $L$ can't be taken large enough. The proofs of the following two lemmas are inspired by~\cite[(3.6)]{Christ90}.%Combining Lemma~\ref{boundary condition} with the proof of (3.6) in~\cite{Christ90}
%This result goes back to Christ~\cite{Christ90}. The proof of the lemma is based on~\cite[Lemma 17]{Christ90} and~\cite[Lemma 5.12]{Hyt-Anna12}.%when the diameter of dyadic cube is large enough.
%Hyt\"{o}nen and Kairema~\cite[Page 20]{Hyt-Anna12} say that if the space $(G,d,m)$ satisfies the doubling condition~\eqref{doubling condition of ball},  the "dyadic cubes" in Proposition~\ref{dyadic cube} also satisfies the conclusion in~\cite[Theorem 11,(3.6)]{Christ90}, that is, there are the constants $C_2>0$ and $\eta>0$ such that %alternative constructions of Proposition~\ref{Hyt-Anna12} The notation $(m,\beta)\preceq (m+1,\gamma)$ means that $d(z_\beta^m,z_\gamma^{m+1})\le C_0\delta^{m+1}$.
\begin{proof}%[Proof of Lemma~\ref{boundary condition}]
For a fixed point $x\in Q_\alpha^k$ with $d(x,G\setminus Q_\alpha^k)\le\delta^{-L}\delta^k$, we claim that there exists a chain $Q_{\sigma_{k-L+L_0}}^{k-L+L_0}\subset\cdots\subset Q_{\sigma_{k-1}}^{k-1}\subset Q_{\sigma_k}^k=Q_\alpha^k$ such that $x\in Q_{\sigma_{k-L+L_0}}^{k-L+L_0}$ and% d(z_{\sigma_{i-1}}^{i-1},z_{\sigma_i}^i)\le C_0\delta^{i},~
\begin{equation}\label{points}
d(z_{\sigma_j}^j,z_{\sigma_i}^i)\ge c_0\delta^i/12,~\forall~k-L+L_0\le j< i\le k.
\end{equation}
Indeed, {let $x\in Q_\alpha^k$}, by Proposition~\ref{dyadic cube}(i)-(iii), for any $n\le k$, there exists a `dyadic cube' $Q_{\sigma}^n$ such that $x\in Q_{\sigma}^n\subseteq Q_\alpha^k$. Hence there exists a chain $Q_{\sigma_{k-L+L_0}}^{k-L+L_0}\subset\cdots\subset Q_{\sigma_{k-1}}^{k-1}\subset Q_{\sigma_k}^k=Q_\alpha^k$ with $x\in Q_{\sigma_{k-L+L_0}}^{k-L+L_0}$. %that for every $k-L+L_0\le j<i\le k$, $d(z_{\sigma_j}^j,z_{\sigma_i}^i)\ge c_0\delta^i/12$.
Now we verify \eqref{points}. First, by {Proposition~\ref{dyadic cube}(iv)}, for every $k-L+L_0\le j\le k$, we have $x\in B(z_{\sigma_{j}}^j, C_1\delta^j)$. We also have $B(z_{\sigma_i}^i,a_0\delta^i)\subset Q_{\sigma_i}^i\subset Q_\alpha^k$. If~\eqref{points} were not true, then
\begin{align*}
  a_0\delta^i&\le d(z_{\sigma_i}^i,G\setminus Q_\alpha^k)\le d(z_{\sigma_i}^i,z_{\sigma_j}^j)+d(z_{\sigma_j}^j,x)+d(x,G\setminus Q_\alpha^k)< \frac{c_0}{12}\delta^i+C_1\delta^j+\delta^{-L}\delta^k\\
  &\le \frac{a_0}{4}\delta^i+\frac{a_0}{3}\delta^{j+1}+\delta^{-L_0}\delta^{i}\le \frac{a_0}{4}\delta^i+\frac{a_0}{3}\delta^{i}+\frac{a_0}{4}\delta^{i}< a_0\delta^i,
\end{align*}
 since $a_0=c_0/3$, $C_1=2C_0$, $3C_1\le a_0\delta$ and $\delta^{-L_0}\le \delta^{-\log_\delta(12/c_0)}=c_0/12$. This leads to a contradiction and the claim is proved.

% We use the pair $(i,\beta(i))$ stands for the "dyadic cube" $Q_{\beta(i)}^i$ and The notation $(i,\beta(i,x))\preceq(i+1,\beta(i+i,x))$ stands for $Q_{\beta(i,x)}^i\subset Q_{\beta(i+1,x)}^{i+1}$.
For every `dyadic cube' $Q_\beta^m$, we denote it briefly by $(m,\beta)$. Let $E=\big\{x\in Q_\alpha^k:d(x,G\setminus Q_\alpha^k)\le \delta^{-L}\delta^k\big\}$. For each $x\in E$, there exists a chain of pair $(i,\beta(i,x))$ with the properties proved in the first paragraph. Set $S_i=\cup_{x\in E}\{z_{\beta(i,x)}^i\}$ for $k-L+L_0\le i\le k$. In the following, we abbreviate $z^i_{\beta(i,x)}$ to $z_{i(x)}$. We have the following observation: for $z_{i(x)}\neq z_{j(y)}$,
\begin{equation}\label{intersect}
  B(z_{i(x)},c_0\delta^i/36)\cap B(z_{j(y)},c_0\delta^j/36)=\emptyset,~\forall~z_{i(x)}\in S_i,~z_{j(y)}\in S_j,~k-L+L_0\le i,j\le k.
\end{equation}
For $i=j$, by~\eqref{distance}, the above assertion is trivially right. For $i\neq j$,  without loss of generality, we assume $i>j$. Note that by the definition of $S_j$, for each $z_{j(y)}\in S_j$, there exists a point $\zeta\in E$ such that $d(z_{j(y)},G\setminus Q_\alpha^k)\le d(z_{j(y)},\zeta)+d(\zeta,G\setminus Q_\alpha^k)\le C_1\delta^j+\delta^{-L}\delta^k$. It follows that if $B(z_{i(x)},c_0\delta^i/36)\cap B(z_{j(y)},c_0\delta^j/36)\neq \emptyset$, then there exists a point $z\in B(z_{i(x)},c_0\delta^i/36)\cap B(z_{j(y)},c_0\delta^j/36)$ such that
\begin{align*}
a_0\delta^i&\le d(z_{i(x)},G\setminus Q_\alpha^k)\le d(z_{i(x)},z_{j(y)})+d(z_{j(y)},G\setminus Q_\alpha^k)\\
&\le d(z_{i(x)},z)+d(z_{j(y)},z)+d(z_{j(y)},G\setminus Q_\alpha^k)\\
&\le \frac{c_0}{36}\delta^i+\frac{c_0}{36}\delta^j+C_1\delta^j+\delta^{-L}\delta^k\le \frac{c_0}{18}\delta^i+C_1\delta^j+\delta^{-L}\delta^k\le \frac{a_0}{6}\delta^i+\frac{a_0}{3}\delta^{j+1}+\delta^{-L_0}\delta^{i}\\
&\le \frac{a_0}{6}\delta^i+\frac{a_0}{3}\delta^{i}+\frac{a_0}{4}\delta^{i}< a_0\delta^i,
\end{align*}
where we used $a_0=c_0/3$, $C_1=2C_0$, $3C_1\le a_0\delta$ and $\delta^{-L_0}\le \delta^{-\log_\delta(12/c_0)}=c_0/12$ again. This leads to a contradiction and so~\eqref{intersect} holds.

% By the condition $L<k+L_0-L_1$, then for each $k-L+L_0\le i\le k$ and $L_1=[\log_\delta(36r_0/c_0)]+1$, we have
%$$\delta^i\ge \delta^{k-L+L_0}\ge \delta^{L_1}> \frac{36r_0}{c_0},$$
We now prove the desired result. In the following, we write $z_i$ {for} $z_{i(x)}$. {Setting} $G_i=\cup_{z_{i}\in S_i}B(z_{i},c_0\delta^i/36)$, for any $k-L+L_0\le i\le k$, we have
\begin{align*}%\cup_{z_{k-L+L_0}\in S_{k-L+L_0}}B(z_{k-L+L_0}, C_1\delta^{k-L+L_0})
m(E)&\le\sum_{z_{k-L+L_0}\in S_{k-L+L_0}}m(B(z_{k-L+L_0}, C_1\delta^{k-L+L_0}))\\
&\le (K+1)\bigg(\frac{36C_1}{c_0}\bigg)^\epsilon\sum_{z_{k-L+L_0}\in S_{k-L+L_0}}m(B(z_{k-L+L_0},c_0\delta^{k-L+L_0}/36))\\
&=(K+1)\bigg(\frac{36C_1}{c_0}\bigg)^\epsilon\sum_{w_i\in S_i}\sum_{\substack{z_{k-L+L_0}\preceq w_i,\\z_{k-L+L_0}\in S_{k-L+L_0}}}m(B(z_{k-L+L_0},c_0\delta^{k-L+L_0}/36))\\
&\le (K+1)\bigg(\frac{36C_1}{c_0}\bigg)^\epsilon\sum_{w_i\in S_i}m(B(w_i,C_1\delta^i))\\
&\le(K+1)^2\bigg(\frac{36C_1}{c_0}\bigg)^{2\epsilon} \sum_{w_i\in S_i}m(B(w_i,c_0\delta^i/36))=(K+1)^2\bigg(\frac{72C_0}{c_0}\bigg)^{2\epsilon}m(G_i),
\end{align*}
where $z_{k-L+L_0}\preceq w_i$ in the third line means that the inclusion of the corresponding `dyadic cubes' $Q_{\beta(k-L+L_0)}^{k-L+L_0}\subset Q^i_{\beta(i)}$, and we used~\eqref{int} in the second and last inequalities since $\delta^i\ge \delta^{k-L+L_0}\ge \delta^{L_1}> 36r_0/c_0$. The equality in the third line follows from~\eqref{intersect}, as does the equality in the last line.
From the above inequalities and the disjointness of the sets $G_i$, we obtain
\begin{equation*}%\label{ineq}
  m(E)\le \frac{(K+1)^2}{(L-L_0+1)}\bigg(\frac{72C_0}{c_0}\bigg)^{2\epsilon}\sum_{i=k-L+L_0}^km(G_i)\le \frac{(K+1)^2}{(L-L_0+1)}\bigg(\frac{72C_0}{c_0}\bigg)^{2\epsilon}m(Q_\alpha^k),
\end{equation*}
and the lemma follows.
\end{proof}

Set
\begin{align*}
&L_2=[\log_{\delta}(4C_0+1)]+1,~L_3=[2(K+1)^2\big(\frac{72C_0}{c_0}\big)^{2\epsilon}]+L_0+L_2,\\
&\eta=(\log_\delta2)/L_3,~C_2=4(K+1)^2(72C_0/c_0)^{2\epsilon},~{C_2^\prime=(K+1)(3C_1/a_0)^{\epsilon}.}
\end{align*}%Combining Lemma~\ref{boundary condition} with the proof of (3.6) in~\cite{Christ90}, we have the following lemma.
\begin{lemma}\label{boundary}%With the conditions given by Lemma~\ref{boundary condition},Let $k$ and $L$ satisfy that $L_0<L<k+L_0-L_1$,
Under the assumption of Lemma~\ref{boundary condition}, we have
\begin{equation*}
\begin{split}
 &m\big(\{x\in Q_\alpha^k:d(x,G\setminus Q_\alpha^k)\le \delta^{k-L}\}\big)\le C_2\delta^{-L\eta}m(Q_\alpha^k),\\
 &m\big(\{x\in G\setminus Q_\alpha^k:d(x,Q_\alpha^k)\le \delta^{k-L}\}\big)\le C_2C^\prime_2\delta^{-L\eta}m(Q_\alpha^k).
\end{split}
\end{equation*}
\end{lemma}
%As mentioned before, under the doubling measure, Christ~\cite{Christ90} proved that the above inequality holds for every $k,L\in\mathbb{Z}$ , namely the `dyadic cube' $Q_\alpha^k$ satisfies the small boundary property.

\begin{proof}
%Recall that the constants $L_2=[\log_{\delta}(4C_0\delta^{L_0}+1)]+1$ and $L_3=[2(K+1)^2\big(\frac{72C_0}{c_0}\big)^{2\epsilon}]+L_0+L_2$. Set%For $k-L+L_0\le i\le k$,(C_1\delta^{L_0}+1)
{Let us  focus on the first inequality.} Let $\ell\in\mathbb{N}$, set
\begin{equation*}
  E_{-\ell}(Q_\alpha^k)=\{Q_\beta^{k-\ell}\subset Q_\alpha^k:d(Q_\beta^{k-\ell},G\setminus Q_\alpha^{k})\le (C_1+1)\delta^{k-\ell}\},
\end{equation*}
where $d(Q_\beta^{k-\ell},G\setminus Q_\alpha^{k})=\inf_{x\in Q_\beta^{k-\ell}}d(x,G\setminus Q_\alpha^{k})$. We denote by
\begin{equation*}
  e_{-\ell}(Q_\alpha^k)=\{x:x\in Q_\beta^{k-\ell}~\textit{with}~Q_\beta^{k-\ell}\in E_{-\ell}(Q_\alpha^k)\}
\end{equation*}
the underlying point set.  We proceed to show that
\begin{equation}\label{inequality}
  \big\{x\in Q_\alpha^k:d(x,G\setminus Q_\alpha^k)\le \delta^{k-\ell}\big\}\subseteq e_{-\ell}(Q_\alpha^k)\subseteq\{x\in Q_\alpha^k:d(x,G\setminus Q_\alpha^k)\le\delta^{k-\ell+L_2}\}.
\end{equation}%by the first paragraph of the proof of Lemma~\ref{boundary condition}, we know that such that $x\in Q_{\beta}^{k-L+L_0}$
Fix $x\in Q_\alpha^k$ such that $d(x,G\setminus Q_\alpha^k)\le \delta^{k-\ell}$. There exists `dyadic cube' $Q_{\beta}^{k-\ell}$ such that $x\in Q_{\beta}^{k-\ell}\subset Q_\alpha^k$, and then
\begin{align*}
d(Q_{\beta}^{k-\ell},G\setminus Q_\alpha^k)&\le C_1\delta^{k-\ell}+d(x,G\setminus Q_\alpha^k)\le(C_1+1)\delta^{k-\ell}.
\end{align*}
On the other hand, fix a point $x\in e_{-\ell}(Q_\alpha^k)$, then there exists $ Q_\beta^{k-\ell}\in E_{-\ell}(Q_\alpha^k)$ such that%by the triangle inequality,
\begin{equation*}
  d(x,G\setminus Q_\alpha^k)\le C_1\delta^{k-\ell}+d(Q_\beta^{k-\ell},G\setminus Q_\alpha^k)\le (2C_1+1)\delta^{k-\ell}\le\delta^{k-\ell+L_2},
\end{equation*}
since $C_1=2C_0$, and~\eqref{inequality} is proved.%We also need some notations.Recall that $L_3=[2(K+1)^2\big(\frac{72C_0}{c_0}\big)^{2\epsilon}]+L_0+L_2$.

 To achieve our goal, we split $L$ into two cases: $L_0<L\le2L_3$ and $L>2L_3$. We first prove the case $L>2L_3$.  %Note that plugging~\eqref{inequality} into~Lemma~\ref{boundary condition}, we have
%$$m(e_{-L_3}(Q_\alpha^k))\le\frac{(K+1)^2}{(L_3-L_2-L_0+1)}\bigg(\frac{72C_0}{c_0}\bigg)^{2\epsilon}m(Q_\alpha^k)\le\frac{1}{2}m(Q_\alpha^k).$$
Set $M_0=[L/L_3]$; {here and below, $[t]$ denotes the integer part of a real number $t$}. Let $F_1(Q_\alpha^k)= E_{-L_3}(Q_\alpha^k)$ and
\begin{equation*}
  F_{n}(Q_\alpha^k)=\bigcup_{Q_\beta^{k-(n-1)L_3}\in F_{n-1}(Q_\alpha^k)}E_{-L_3}(Q_\beta^{k-(n-1)L_3}),
\end{equation*}
for $2\le n\le M_0$. Let $f_n(Q_\alpha^k)$ be the underlying point set. It is easy to check that
\begin{equation}\label{contain}
e_{-nL_3}(Q_\alpha^k)\subset f_n(Q_\alpha^k).
\end{equation}
Moreover, for each $Q_\beta^{k-(n-1)L_3}\in F_{n-1}(Q_\alpha^k)$ and $1\le n\le M_0$, plugging~\eqref{inequality} into Lemma~\ref{boundary condition}, we obtain
\begin{align*}
  m(e_{-L_3}(Q_\beta^{k-(n-1)L_3}))&\le m\{x\in Q_\beta^{k-(n-1)L_3}:d(x,G\setminus Q_\beta^{k-(n-1)L_3})\le\delta^{k-nL_3+L_2}\}\\
  &\le\frac{(K+1)^2}{(L_3-L_2-L_0+1)}\bigg(\frac{72C_0}{c_0}\bigg)^{2\epsilon}m(Q_\beta^{k-(n-1)L_3})\le\frac{1}{2}m(Q_\beta^{k-(n-1)L_3}),
\end{align*}%and $L_0<L_3-L_2<k-(n-1)L_3+L_0-L_1$
since $nL_3\le L<k+L_0-L_1$. Then
\begin{equation*}
  m(f_n(Q_\alpha^k))\le \frac{1}{2}m(f_{n-1}(Q_\alpha^k)),
\end{equation*}
and iterating the above inequality we have $m(f_{M_0}(Q_\alpha^k))\le 2^{-M_0}m(Q_\alpha^k)$. From this inequality and~\eqref{contain}, one has
\begin{align*}
m\big(\{x\in Q_\alpha^k:&d(x,G\setminus Q_\alpha^k)\le \delta^{k-L}\}\big)\le m(e_{-L}(Q_\alpha^k))\le m(e_{-M_0L_3}(Q_\alpha^k))\\
&\le m(f_{M_0}(Q_\alpha^k))\le 2^{-M_0}m(Q_\alpha^k)\le 2^{-L/L_3+1}m(Q_\alpha^k)=2\delta^{-\eta L}m(Q_\alpha^k),
\end{align*}
where the second inequality in the first line follows from the definition of $e_{-\ell}(Q_\alpha^k)$.

For the case $L_0<L\le 2L_3$, using Lemma~\ref{boundary condition} again we have
\begin{align*}
m\big(\{x\in Q_\alpha^k:d(x,G\setminus Q_\alpha^k)\le \delta^{k-L}\}\big)&\le(K+1)^2\bigg(\frac{72C_0}{c_0}\bigg)^{2\epsilon}m(Q_\alpha^k)\\
&\le \delta^{-\eta L}\delta^{2\eta L_3}(K+1)^2\bigg(\frac{72C_0}{c_0}\bigg)^{2\epsilon}m(Q_\alpha^k)\\
&=4(K+1)^2\bigg(\frac{72C_0}{c_0}\bigg)^{2\epsilon}\delta^{-\eta L}m(Q_\alpha^k).
\end{align*}

{We now prove the second inequality.} Let $\widetilde{E}=\{x\in G\setminus Q_\alpha^k:d(x,Q_\alpha^k)\le \delta^{k-L}\}$. {Define the set $\widetilde{I}=\{\beta: Q_\beta^k\cap\widetilde{E}\neq\emptyset,\beta\in I_k\}$ and $e(Q_\beta^k)=\{x\in Q_\beta^k:d(x,G\setminus Q_\beta^k)\le \delta^{k-L}\}$. Then by Proposition~\ref{dyadic cube}(i), we obtain $\widetilde{E}=\cup_{\beta\in \widetilde{I}}e(Q_\beta^k)$.

Fix $\beta\in \widetilde{I}$. There exists a point $y_0\in Q_\beta^k\cap \widetilde{E}$. By Proposition~\ref{dyadic cube}(iv), for every $y\in Q_\beta^k$, we have
$$d(y,z_\alpha^k)\le d(y,y_0)+d(y_0,z_\alpha^k)\le C_1\delta^k+\delta^{k-L}+C_1\delta^k\le 3C_1\delta^k.$$
Hence $Q_\beta^k\subseteq B(z_\alpha^k,3C_1\delta^k)$. By Proposition~\ref{dyadic cube}(ii), the 'dyadic cubes' $Q_\beta^k$ are disjoint, hence $\cup_{\beta\in\widetilde{I}}Q_\beta^k\subseteq B(z_\alpha^k,3C_1\delta^k)$}. Then by the first inequality and~\eqref{int}, {we obtain
\begin{align*}
  m(\widetilde{E})&\le m\big(\cup_{\beta\in\widetilde{I}}e(Q_\beta^k)\big)\le \sum_{\beta\in\widetilde{I}}m(e(Q_\beta^k))\\
  &\le \sum_{\beta\in\widetilde{I}}C_2\delta^{-L\eta}m(Q_\beta^k)\le C_2\delta^{-L\eta}m(\cup_{\beta\in\widetilde{I}}Q_\beta^k)\le C_2\delta^{-L\eta}m(B(z_\alpha^k,3C_1\delta^k))\\
  &\le C_2(K+1)(3C_1/a_0)^{\epsilon}m(B(z_\alpha^k,a_0\delta^k))\\
  &\le C_2(K+1)(3C_1/a_0)^{\epsilon}m(Q_\alpha^k),
  %\le (K+1)(2C_1/a_0)^\epsilon \delta^{-L\eta} m(B_{C_1\delta^k})\\
  %&\le 2^\epsilon C_2(K+1)^2(C_1/a_0)^{2\epsilon}\delta^{-L\eta}m(B_{a_0\delta^k})\le 2^\epsilon C_2(K+1)^2(C_1/a_0)^{2\epsilon}\delta^{-L\eta}m(Q_\alpha^k),
\end{align*}
}which completes the proof.
%where
%$$E_{k-L_3}(Q_\beta^{k-(n-1)L_3})=\{Q_\gamma^{k-nL_3+L_0}\subset Q_\beta^{k-(n-1)L_3+L_0}:d(Q_\gamma^{k-nL_3+L_0},G\setminus Q_\alpha^{k})\le \delta^{k-L_3+L_2}\}.$$
\end{proof}

Given a ball $B_{\delta^{n}}$ and a `dyadic cube' $Q^{k}_\alpha$, define
$$\mathcal{H}(B_{\delta^{n}},Q^{k}_\alpha)=\{x\in Q_\alpha^{k}:B(x,\delta^{n})\cap (Q^{k}_\alpha)^{c}\neq \emptyset\}.$$
Set
$$n_0=\max\{L_1-L_0,0\}.$$
\begin{lemma}\label{measure estimate of cube}
%Set $k_1=\min\{k:a_0\delta^{k-1}\ge 2\delta\}$.Set $k_1=\min\{k:a_0\delta^{k-1}>2\delta^2\}$.Since for every $k\ge k_1$, $a_0\delta^{k-1}>2\delta^2$, it follows that $diam(Q^{n+k-1}_\alpha)\ge a_0\delta^{n+k-1}>2\delta^{n+1}$.There exists a positive constant $C$ such that
%Let $n>n_0$ and $k\in\mathbb{N}$,
Let $n>n_0$ and $k>L_0$. Then for every `dyadic cube' $Q_\alpha^{n+k}$, we have
\begin{equation*}
  m(\mathcal{H}(B_{\delta^{n}},Q^{n+k}_\alpha))\le C_2\delta^{-k\eta} m(Q^{n+k}_\alpha).
\end{equation*}
\end{lemma}

\begin{proof}
 We check at once that for every $x\in \mathcal{H}(B_{\delta^{n}},Q^{n+k}_\alpha)$, the distance $d(x, (Q^{n+k}_\alpha)^{c})$ is not bigger than $\delta^{n}$, and so
\begin{equation*}
 \begin{split}
  \mathcal{H}(B_{\delta^{n}},Q^{n+k}_\alpha)%&\subseteq \{x\in Q_\alpha^{n+k-1}:d(x,G\setminus Q_\alpha^{n+k-1})\le r\}\\
&\subseteq \{x\in Q_\alpha^{n+k}:d(x,G\setminus Q_\alpha^{n+k})\le \delta^{n}\}.
 \end{split}
\end{equation*}
Note that $k>L_0$ and $n>L_1-L_0$, then by Lemma~\ref{boundary}, we obtain
\begin{equation*}%\label{measure estimate of cube}
  \begin{split}
   m(\mathcal{H}(B_{\delta^{n}},Q^{n+k}_\alpha))\le m\big(\{x\in Q_\alpha^{n+k}:d(x,G\setminus Q_\alpha^{n+k})\le \delta^{n+k-k}\}\big)\le C_2\delta^{-k\eta} m(Q^{n+k}_\alpha),
  \end{split}
\end{equation*}%By the same proof of \cite[Theorem A]{JKRW1}(or \cite[Theorem 2.3]{GXHTM})
which is the desired conclusion.
\end{proof}
\begin{remark}
\emph{
Set $\widetilde{\mathcal{H}}(B_{\delta^{n}},Q^{k}_\alpha)=\{x\in G\setminus Q_\alpha^{k}:B(x,\delta^{n})\cap (Q^{k}_\alpha)^{c}\neq \emptyset\}$. Similar to Lemma~\ref{measure estimate of cube},  by Lemma~\ref{boundary}, we have
\begin{equation*}
  m(\widetilde{\mathcal{H}}(B_{\delta^{n}},Q^{n+k}_\alpha))\le C_2\delta^{-k\eta} m(Q^{n+k}_\alpha).
\end{equation*}}
\end{remark}

Set
$$K_\epsilon=(2^\epsilon+1)K+2^\epsilon.$$

\begin{lemma}\label{lem:decay}
Let $r\ge2r_0$, then for every $s\in(0,r]$, we have
\begin{equation*}%\label{decay}
\begin{split}
 %&m(B(x, r+s)\setminus B(x,r))\le C\bigg(\frac{s}{r}\bigg)^{\epsilon}m(B(x,r)),\\
 m(B(x, r+s))-m(B(x,r-s))\le K_{\epsilon}\bigg(\frac{s}{r}\bigg)^{\epsilon}m(B(x,r)).
\end{split}
\end{equation*}
\end{lemma}
\begin{proof}
We only need to estimate the part $m(B(x, r))-m(B(x,r-s))$, since by the $(\epsilon,r_0)$-annular decay property---\eqref{decay property}, we have $m(B(x, r+s))-m(B(x,r))\le K(s/r)^{\epsilon}m(B(x,r))$.
In the following we split the $s$ into two cases $s\in(0,r/2)$ and  $s\in[r/2,r)$.
For $s\in(0,r/2)$, since $r\ge 2r_0$, then $r-s>r/2\ge r_0$, applying the $(\epsilon,r_0)$-annular decay property again, we have
\begin{align*}
  m(B(x, r))-m(B(x,r-s))&\le K\bigg(\frac{s}{r-s}\bigg)^{\epsilon}m(B(x,r-s))\\
  &\le K\bigg(\frac{r}{r-s}\bigg)^{\epsilon}\bigg(\frac{s}{r}\bigg)^{\epsilon}m(B(x,r))\le  K2^{\epsilon}\bigg(\frac{s}{r}\bigg)^{\epsilon}m(B(x,r)).
\end{align*}

 For $s\in[r/2,r)$, since $1/2\le s/r<1$, so $ 2^\epsilon(s/r)^{\epsilon}\ge 1$ and $ m(B(x, r))-m(B(x,r-s))\le2^\epsilon(s/r)^{\epsilon}m(B(x, r))$.
 {Combining the two cases}, we obtain
  \begin{equation*}
     m(B(x, r+s))-m(B(x,r-s))\le \big((2^\epsilon+1)K+2^\epsilon\big)\bigg(\frac{s}{r}\bigg)^{\epsilon}m(B(x,r)),
  \end{equation*}
  and the lemma follows.
 %for $s\in(0,r-r_0)$ the weak $\epsilon$-annular decay property implies~\eqref{decay}. Note that by the doubling property~\eqref{int}, we have $m(B(x,r-s))/m(B(x,r))\ge C_m^{-1}\big(1-s/r\big)^{c_m}$. Since $r>2r_0$ and $s\in[r-r_0,r)$, then $s/r\in(1/2,1)$, and we can choose a constant $C$ such that $1- C_m^{-1}\big(1-s/r\big)^{c_m}\le C \big(s/r\big)^{\epsilon}$.
\end{proof}

% For each $n\in\mathbb{N}$ and $k\in\mathbb{Z}$, we define
{Given a ball $B(x,r)$ and an integer $n$, we define}
$$\mathcal{I}(B(x,r),n)=\cup_\alpha\{Q_\alpha^{n}\cap B(x,r): Q_\alpha^{n}\cap\partial B(x,r) \neq\emptyset\}.$$ % and $\tilde{\mathcal{I}}(B(x,r),n+k)=\cup_\alpha\{Q_\alpha^{n+k}: Q_\alpha^{n+k}\cap\partial B(x,r) \neq\emptyset\}$.
Set
$$n_1=\min\{n\in\mathbb{N}:\delta^n\ge2r_0\}, k_1=\max\{k\in\mathbb{Z}:C_1\delta^k\le 1\}.$$
Unless otherwise stated, we assume that $n>n_1$, $k<k_1$  in the following three lemmas.%, we assume that $n>n_0$ and $k<k_1$. we only assume that the space $(G,d,m)$ satisfies the weak $(\epsilon,r_0)$-annular decay property.

%Recall that for every number $a\ge 0$,  $[a]$ denotes the largest integer $\le a$.
\begin{comment}
From the doubling condition \eqref{doubling condition of ball}, it is well-known that
\begin{align}\label{int}
\frac{m(B(x,R))}{m(B(x,r))}\leq C_m\Big(\frac{R}{r}\Big)^{\theta},\;\forall~0<r\leq R,~x\in G,
\end{align}
where $\theta=\log_2 C_m$ and $C_m$ is the smallest constant such that \eqref{doubling condition of ball} holds. %But, we need more precise information about $m(B_R)/m(B_r)$.
\end{comment}

% The above lemma implies the following results.
%For every number $a\ge 0$, $[a]$ denotes the largest integer $\le a$.
\begin{lemma}\label{measure estimate of annulus}
 %Set $k_2=\max\{k:C_1\delta^k\le\delta^{-1}\}$ and $n_1=\max\{-k_2,\min\{n:[\delta^{n}]\ge [C_1]+1\}\}$, then $k_2$ and $n_1$ only depend on $C_1$ and $\delta$. For every $n\ge n_1$ and $-n\le k\le k_2$, there is a positive constant $C$ such that Let $k_2=\min\{-\tilde{k},\max\{k:\delta^k\le (a_0\delta^{l_0}-1)\delta^{-\tilde{k}+1}\}\}$.Let the space $G$ satisfies the weak $\epsilon$-annular decay property.  where $C_1$ is a constant defined in Proposition~\ref{dyadic cube}.$n>n_0$ and every $k<k_1$,
 For any $x\in G$, we have%there exists a positive constant $C_{m,\epsilon,K}=K_\epsilon C_1^\epsilon$ such that
\begin{equation*}
 \begin{split}
   &\sup_{r\in[\delta^{n},\delta^{n+1}]}\frac{m(\mathcal{I}(B(x,r),n+k))}{m(B(x,r))}\le K_\epsilon C_1^\epsilon\delta^{\epsilon k}.\\ %&\sup_{r\in[\delta^{n},\delta^{n+1}]}\frac{m(\tilde{\mathcal{I}}(B(x,r),n+k))}{m(B(x,r))}\le K_\epsilon C_1^\epsilon\delta^{\epsilon k}.
 \end{split}
\end{equation*}
%where
%$$\mathcal{I}(B(x,r),n+k)=\cup_\alpha\{Q_\alpha^{n+k}\cap B(x,r):\partial B(x,r)\cap Q_\alpha^{n+k}\neq\emptyset\}$$
%and
%$$\tilde{\mathcal{I}}(B(x,r),n+k)=\cup_\alpha\{Q_\alpha^{n+k}:\partial B(x,r)\cap Q_\alpha^{n+k}\neq\emptyset\}.$$
\end{lemma}

\begin{proof}
%{\color{red} The proof should be based on the fact $\tilde{\mathcal{I}}(B(x,r),n+k)$ is contained in some annulus $B(x,r+\delta^{n+k+k_1})\setminus B(x,r-\delta^{n+k+k_1})$ where $k_1$ is determined by $C_1$ in Proposition iv). Then use \eqref{int} to estimate the volumn of the annnulus.}
%Since $\mathcal{I}(B(x,r),n+k)\subset \tilde{\mathcal{I}}(B(x,r),n+k)$, hence we only need to estimate the part $m(\tilde{\mathcal{I}}(B(x,r),n+k))$.
Note that for $k<k_1$ and $r\in[\delta^n,\delta^{n+1}]$, ${\mathcal{I}}(B(x,r),n+k)$ is contained in the annulus $B(x,r+C_1\delta^{n+k})\setminus B(x,r-C_1\delta^{n+k})$. For every $n>n_1$ and $r\in[\delta^{n},\delta^{n+1}]$, Lemma~\ref{lem:decay} yields %r\in[\delta^n,\delta^{n+1}]$ and $k<-k_$, choosing $Q_\alpha^m$ where $m=n+1+l_0$ in Lemma~\ref{annulus}, we obtain %and in this case $r+\delta^{n+k+k_1}\le a_0\delta^{m}$.To use Lemma~\ref{measure estimate of annulus},and $r\in[\delta^n,\delta^{n+1}]$
\begin{equation*}
\begin{split}
    m({\mathcal{I}}(B(x,r),n+k))&\le m\bigg(B(x,r+C_1\delta^{n+k})\setminus B(x,r-C_1\delta^{n+k})\bigg)\\
    &\le K_\epsilon\bigg(\frac{C_1\delta^{n+k}}{r}\bigg)^{\epsilon}m(B(x,r))\\
    &\le K_\epsilon C_1^\epsilon\delta^{\epsilon k}m(B(x,r)),\\
    %&\le C_m\delta^\theta K_\epsilon C_1^\epsilon\delta^{\epsilon k}m(B(x,\delta^{n})),
\end{split}
\end{equation*}
which is the desired conclusion.%We complete the proof.
\end{proof}

To state the next technical lemmas, we need the following estimate (cf. \cite[Theorem 3.5]{AL19}). Recall that $A^\prime_r$ is the averaging operator given by~\eqref{averaging operator1}.

\begin{prop}\label{aver}
Let $r>0$, then for every $p\in[1,\8]$, we have $\|A^\prime_r\|_{L^p(G,m)\rightarrow L^p(G,m)}\le D^{1/p}$.
\end{prop}

With $\mathcal{I}(B(x,r),n+k)$ being defined as above, we define %and $\tilde{\mathcal{I}}(B(x,r),n+k)$
\begin{equation*}
    \begin{split}
      &M_{n+k}f(x)=\sup_{r\in[\delta^n,\delta^{n+1}]}|\frac{1}{m(B(x,r))} \int_{\mathcal{I}(B(x,r),n+k)}f(y)dm(y)|.\\
    %&\tilde{M}_{n+k}f(x)=\sup_{r\in[\delta^n,\delta^{n+1}]}|\frac{1}{m(B(x,r))} \int_{\tilde{\mathcal{I}}(B(x,r),n+k)}f(y)dm(y)|.
    \end{split}
\end{equation*}
%The maximal operator is defined by
%\begin{equation*}
% M_{n+k}f(x)=\sup_{r\in[\delta^n,\delta^{n+1}]}|A_{r,n+k}f(x)|.
%\end{equation*}
%\begin{equation*}
%  \tilde{M}_{n+k}f(x)=\sup_{\delta^n\le r_0<\cdots<r_J\le\delta^{n+1}}\bigg(\sum_{i=1}|\tilde{A}_{r_i,n+k}f(x)|^2\bigg)^{1/2}.
%\end{equation*}
%The following lemma implies the $L^p$-boundedness of maximal operators $A_{n+k}$ and $\tilde{A}_{n+k}$. %and $\tilde{M}_{n+k}$.
\begin{lemma}\label{lem:Akn}
%Let the space $G$ satisfies the weak $\epsilon$-annular decay property. With $n_0$ and $k_2$ defined in Lemma~\ref{measure estimate of annulus}.
Let {$p\in[1,\8]$} and $p^\prime$ be its conjugate index. There exists a constant {$D_{p}=\big((K+1)(\delta+1)^\epsilon\big)^{1/p}D^{1/p}(K_\epsilon C_1^\epsilon)^{1/p^\prime}$} such that for all $f\in L^p(G,m)$,
\begin{equation*}
  \|M_{n+k}f\|_{L^p(G,m)}\le D_{p}\delta^{\epsilon k/p^\prime}\|f\|_{L^p(G,m)}.%,~ \|\tilde{M}_{n+k}f\|_{L^p(G,m)}\le C_{3} \delta^{\epsilon k/p^\prime}\|f\|_{L^p(G,m)}.
\end{equation*}
\end{lemma}

\begin{proof}
For $p=\8$, using Lemma~\ref{measure estimate of annulus}, the conclusion is obvious. For $p\in[1,\8)$, fix $x\in G$, note that ${\mathcal{I}}(B(x,r),n+k)\subseteq B(x,r+C_1\delta^{n+k})$. Then using the H\"{o}lder inequality and Lemma~\ref{measure estimate of annulus}, we obtain%where $k_1$ is determined by $C_1$ in (iv) of Proposition~\ref{dyadic cube}\sup_{r\in[\delta^n,\delta^{n+1}]}.
\begin{equation*}
   \begin{split}
      {M}_{n+k}f(x)&\le \sup_{r\in[\delta^n,\delta^{n+1}]}\frac{m({\mathcal{I}}(B(x,r),n+k))^{1/p^\prime}}{m(B(x,r))}\bigg(\int_{{\mathcal{I}}(B(x,r),n+k)}
  |f(y)|^pdm(y)\bigg)^{1/p}\\
  &\le (K_\epsilon C_1^\epsilon)^{1/p^\prime}\delta^{\epsilon k/p^\prime}\bigg(\frac{1}{m(B(x,\delta^n))}\int_{B(x,\delta^{n+1}+C_1\delta^{n+k})}|f(y)|^pdm(y)\bigg)^{1/p},\\
   \end{split}
\end{equation*}
{and  by inequality~\eqref{int}, the above inequality yields
\begin{align*}
   &{M}_{n+k}f(x)\le (K_\epsilon C_1^\epsilon)^{1/p^\prime}\delta^{\epsilon k/p^\prime}\bigg(\frac{(K+1)(\delta+1)^\epsilon}{m(B(x,\delta^{n+1}+C_1\delta^{n+k}))}\int_{B(x,\delta^{n+1}+C_1\delta^{n+k})}|f(y)|^pdm(y)\bigg)^{1/p}
\end{align*}
{Then applying Proposition~\ref{aver}, we have}
\begin{equation*}
  \|{M}_{n+k}f\|_{L^{p}}\le\big((K+1)(\delta+1)^\epsilon\big)^{1/p}D^{1/p}(K_\epsilon C_1^\epsilon)^{1/p^\prime}\delta^{\epsilon k/p^\prime}\|f\|_{L^p}.
\end{equation*}}
and the lemma follows.
\end{proof}
Given an annulus $B(x,r)\setminus B(x,s)$ and integers $k,n$, define
\begin{align*}
  &\mathcal{I}(B(x,r)\setminus B(x,s),n)\\
  &=\cup_\alpha\{Q_\alpha^{n}\cap (B(x,r)\setminus B(x,s)): Q_\alpha^{n}\cap\partial (B(x,r)\setminus B(x,s)) \neq\emptyset\},
\end{align*}
 and%and $\tilde{\mathcal{I}}(B(x,r_i)\setminus B(x,r_{i-1}),n+k)=\cup_\alpha\{Q_\alpha^{n+k}: Q_\alpha^{n+k}\cap\partial (B(x,r_i)\setminus B(x,r_{i-1})) \neq\emptyset\}$.
\begin{equation*}
M^S_{n+k}f(x)=\sup_{\delta^n\le r_0<\cdots r_J\le\delta^{n+1}}\bigg(\sum_{i=1}^J|\frac{1}{m(B(x,r_{i}))} \int_{\mathcal{I}(B(x,r_i)\setminus B(x,r_{i-1}),n+k)}f(y)dm(y)|^2\bigg)^{1/2}.
\end{equation*}

\begin{lemma}\label{lem:Mkn}
%Let the space $G$ satisfies the weak $\epsilon$-annular decay property. With $n_0$ and $k_2$ defined in Lemma~\ref{measure estimate of annulus}.  Let ,Let $p\in[2,\8]$.
{There exists a constant $C_{3}=\big(2(K+1)D\delta^\epsilon C_1^\epsilon K_\epsilon\big)^{1/2}$ such that for all $f\in L^2(G,m)$,
\begin{equation*}
  \|M^S_{n+k}f\|_{L^2(G,m)}\le C_3 \delta^{\epsilon k/2}\|f\|_{L^2(G,m)}.
\end{equation*}}
\end{lemma}

\begin{proof}
 Let $x\in G$. Fixing $\delta^n\le r_{i-1}<r_i\le \delta^{n+1}$, by the Cauchy-Schwarz inequality and Lemma~\ref{measure estimate of annulus}, we have%. Note that $\mathcal{I}(B(x,r_i)\setminus B(x,r_{i-1}),n+k)\subseteq \mathcal{I}(B(x,r_i),n+k)\cup\mathcal{I}(B(x,r_{i}),n+k)$ and $\cup_{i}\mathcal{I}(B(x,r_i)\setminus B(x,r_{i-1}),n+k)\subseteq B(x,\delta^{n+1})$. Thus by those observations, and
\begin{align*}
&|\frac{1}{m(B(x,r_{i}))}\int_{\mathcal{I}(B(x,r_i)\setminus B(x,r_{i-1}),n+k)}f(y)dm(y)|^2\\%\sup_{\delta^n\le r_0<\cdots<r_J\le\delta^{n+1}}\sup_{\delta^n\le r_0<\cdots<r_J\le\delta^{n+1}}\sum_{i=1}^J
&\le \frac{m(\mathcal{I}(B(x,r_{i})\setminus B(x,r_{i-1}),n+k))}{m(B(x,r_{i}))^2}\int_{\mathcal{I}(B(x,r_i)\setminus B(x,r_{i-1}),n+k)}|f(y)|^2dm(y)\\
&\le \frac{2C_1^\epsilon K_\epsilon\delta^{\epsilon k}}{m(B(x,\delta^{n}))} \int_{\mathcal{I}(B(x,r_i)\setminus B(x,r_{i-1}),n+k)}|f(y)|^2dm(y),
\end{align*}
where we used the fact that $\mathcal{I}(B(x,r_i)\setminus B(x,r_{i-1}),n+k)\subseteq \mathcal{I}(B(x,r_i),n+k)\cup\mathcal{I}(B(x,r_{i-1}),n+k)$ in the last inequality. From this, inequality~\eqref{int} and the observation that $\cup_{i}\mathcal{I}(B(x,r_i)\setminus B(x,r_{i-1}),n+k)\subseteq B(x,\delta^{n+1})$, it follows that
\begin{align*}
  M^S_{n+k}f(x)&\le \bigg(\frac{2C_1^\epsilon K_\epsilon\delta^{\epsilon k}}{m(B(x,\delta^{n}))} \int_{B(x,\delta^{n+1})}|f(y)|^2dm(y)\bigg)^{1/2}\\
  &\le \bigg(\frac{2(K+1)\delta^\epsilon C_1^\epsilon K_\epsilon\delta^{\epsilon k}}{m(B(x,\delta^{n+1}))} \int_{B(x,\delta^{n+1})}|f(y)|^2dm(y)\bigg)^{1/2}&.
\end{align*}
{Then using Proposition~\ref{aver}, we conclude}
\begin{equation*}
  \|M^{S}_{n+k}f\|_{L^2} \le\big(2(K+1)D\delta^\epsilon C_1^\epsilon K_\epsilon\big)^{1/2}\delta^{\epsilon k/2}\|f\|_{L^p}.%C\delta^{\epsilon k/2}\|\big(A_{\delta^{n+1}}(|f|^2)\big)^{1/2}\|_{L^p}
\end{equation*}
\end{proof}

\begin{comment}
Recall that for $f(x)\in L^p(G,m)$ and $r>0$, the ball averaging operator is defined by
\begin{equation*}
  A_rf(x)=\frac{1}{m(B(x,r))}\int_{B(x,r)}f(y)dm(y).
\end{equation*}%where $A_rf=f$ for $r<1$.
\end{comment}
%Recall that the (ball) Hardy-Littlewood maximal function is given by
%\begin{equation*}
%  Mf(x)=\sup_{r>0}A_r(|f|)(x).
%\end{equation*}

Let $f$ be a locally integrable function on $G$.  We define
\begin{equation*}
  S_n(f)=|A^\prime_{\delta^n}f-\mathbb{E}_nf|,~SV_n(f)=V_2(A^\prime_{r}f:r\in[\delta^{n},\delta^{n+1})).
\end{equation*}
Note that
\begin{equation*}
  \frac{1}{m(B_r)}\int_{B_r}f(xy)dm(y)=\frac{1}{m(B(x,r))}\int_{B(x,r)}f(y)dm(y).
\end{equation*}
Fixing $\delta^{n}\le r_{i-1}<r_i<\delta^{n+1}$,  we have
\begin{equation*}%\label{decomposition of averaging operator}
  \begin{split}
   \big|\frac{1}{m(B(x,r_i))}&\int_{B(x,r_i)}f(y)dm(y)-\frac{1}{m(B(x,r_{i-1}))}
  \int_{B(x,r_{i-1})}f(y)dm(y)\big|\\
  &\le\big|\frac{1}{m(B(x,r_i))}\int_{B(x,r_i)\setminus B(x,r_{i-1})}f(y)dm(y)\big|\\
  &+\big|\bigg(\frac{1}{m(B(x,r_i))}-\frac{1}{m(B(x,r_{i-1}))}\bigg)\int_{B(x,r_{i-1})}f(y)dm(y)\big|.
  \end{split}
\end{equation*}
%Moreover, $B(x,r)=x\cdot B_r$ and $m$ is a Haar measure, hence
%\begin{equation*}
%  \frac{1}{m(B_r)}\int_{B_r}f(xy)dm(y)=\frac{1}{m(B(x,r))}\int_{B(x,r)}f(y)dm(y).
%\end{equation*}
From the above observations and the triangle inequality of $\ell^2$-norm, we obtain
\begin{equation}\label{controll the short variation}
SV_n(f)(x)\le SV_I(f)(x)+SV_{II}(f)(x),
\end{equation}
where
\begin{equation*}%\label{definition of SI}
 \begin{split}
   SV_I(f)(x)&=\bigg(\sup_{\delta^{n}\le r_0<\cdots<r_J<\delta^{n+1}}\sum_{i=1}^J\frac{1}{m(B(x,r_i))^2}\big|\int_{B(x,r_i)\setminus B(x,r_{i-1})}f(y)dm(y)\big|^2\bigg)^{1/2}
 \end{split}
\end{equation*}
and
\begin{equation*}%\label{definition of SII}
 \begin{split}
   &SV_{II}(f)(x)=\\
&\bigg(\sup_{\delta^{n}\le r_0<\cdots<r_J<\delta^{n+1}}\sum_{i=1}^J
\big|\big(\frac{1}{m(B(x,r_{i-1}))}-\frac{1}{m(B(x,r_i))}\big)\int_{B(x,r_{i-1})}f(y)dm(y)\big|^2\bigg)^{1/2}.
 \end{split}
\end{equation*}
Moreover, since the $\ell^1$ norm is not less than the $\ell^2$ norm, then
\begin{equation}\label{estimate of short variation}
\begin{split}
&SV_n(f)(x)\le \sup_{\delta^{n}\le r_0<\cdots<r_J<\delta^{n+1}}\sum_{i=1}^J\frac{1}{m(B(x,r_i))}\big|\int_{B(x,r_i)\setminus B(x,r_{i-1})}f(y)dm(y)\big|\\
&+\sup_{\delta^{n}\le r_0<\cdots<r_J<\delta^{n+1}}\sum_{i=1}^J
\big|\bigg(\frac{1}{m(B(x,r_{i-1}))}-\frac{1}{m(B(x,r_i))}\bigg)\int_{B(x,r_{i-1})}f(y)dm(y)\big|\\
&\le \frac{1}{m(B(x,\delta^{n}))}
\int_{B(x,\delta^{n+1})\setminus B(x,\delta^{n})}|f(y)|dm(y)\\
&+\bigg(\frac{1}{m(B(x,\delta^{n}))}-\frac{1}{m(B(x,\delta^{n+1}))}\bigg)\int_{B(x,\delta^{n+1})}|f(y)|dm(y)\\
&\le \frac{2}{m(B(x,\delta^{n}))}\int_{B(x,\delta^{n+1})}|f(y)|dm(y).\\
%&\le\frac{2(K+1)\delta^\epsilon}{m(B(x,\delta^{n+1}))}\int_{B(x,\delta^{n+1})}|f(y)|dm(y) .%\le C_{m,\delta}Mf(x).
\end{split}
\end{equation}
{By~\eqref{int} and Proposition~\ref{aver}, the above discussions imply} %the following lemma without any other assumption of $(G,d,m)$.
\begin{lemma}\label{controlled by maximal operator}
Let $p\in[1,\8]$. There exists a constant $C_4=2(K+1)D^{1/p}\delta^\epsilon$ such that % there exists a positive constant $C_5$ such that
\begin{equation*}
  \sup_{{n\ge n_{r_0}}}\|S_n\|_{L^p(G,m)\rightarrow L^p(G,m)}\le D^{1/p}+1,~\sup_{n\ge n_{r_0}}\|SV_n\|_{L^p(G,m)\rightarrow L^p(G,m)}\le C_4.
\end{equation*}
\end{lemma}

Finally, we state the following version of the almost orthogonality principle, see e.g. \cite{JKRW03} for a proof.
\begin{lem}\label{orthogonality principle}
Let $\{T_n\}_{n\in\mathbb{N}}$ be a sequence of sub-linear operators from $L^2$ to $L^2$ on some $\sigma$-finite measure space. Let $\{u_n\}_{n\in\mathbb{Z}}$ and $\{v_n\}_{n\in\mathbb{Z}}$ be two sequences of $L^2$ functions. Let $N$ be a positive integer number. Assume that for $n\ge N$, there exists a sequence of  positive constant $\{a(j)\}_{j\in\mathbb{Z}}$ with $w=\sum_{k\in\mathbb{Z}}a(k)<\8$ such that
\begin{equation}\label{condition of orth-prin2}
\|T_n(u_{k+n})\|_{L^2}\le a(k)\|v_{k+n}\|_{L^2},% For every $n\in\mathbb{N}\setminus [0,\tilde{n}_0]$,
\end{equation}
then
\begin{equation*}
 \sum_{n\ge N}\|\sup_{j,m}|T_n\big(\sum_{j\le k\le m}u_{k+n}\big)|\|^2_{L^2}\le w^2\sum_{n\in\mathbb{Z}}\|v_{n}\|^2_{L^2}.
\end{equation*}
Furthermore, if $T_*:=\max_{0<n\le N}\|T_n\|_{L^2\rightarrow L^2}<\8$, $T_n$ is strongly continuous for every $n\in\mathbb{N}$, $f=\sum_{n\in\mathbb{Z}}u_n$
 and $\sum_{n\in\mathbb{Z}}\|v_n\|^2_{L^2}\le C\|f\|^2_{L^2} $,  then we have
\begin{equation*}
   \sum_{n\in\mathbb{N}} \|T_nf\|^2_{L^2}\le(Cw^2+NT_*^2)\|f\|^2_{L^2}.
\end{equation*}
\end{lem}
\section{Strong type $(2,2)$ estimates}\label{ST3}
In this section, we prove that the square function $S(f)$ and the short variation operator $SV(f)$ are of strong type $(2,2)$. We begin with the square function $S(f)$.
%\section{variational inequalities for ball averaging under the group of polynomial growth}\label{ST3}$\delta^{n_{r_0}}<r_0\leq \delta^{n_{r_0}+1}$.
%%%%%%%%%%%%%%%%%%%%%%%%%%%%%%%%%%%%%%%%%%%%%%%%%%%%%%%%%%%%%%%%%%%%%%%%%%%%%%%%%%%%%%%%%%%%%%%%%%%%%%%%%%%%%%%%
%%%%%%%%%%%%%%%%%%%%%%%%%%%%%%%%%%%%%%%%%%%%%%%%%%%%%%%%%%%%%%%%%%%%%%%%%%%%%%%%%%%%%%%%%%%%%%%%%%%%%%%%%%%%%%%%the proof of Theorem~\ref{main-thm1}
\begin{proof}[Proof of \eqref{strong-type inequalities of square function} in the case $p=2$]
Fix $f\in L^2(G,m)$. Recall that
$$S(f)=\bigg(\sum_{n>n_{r_0}}|A^\prime_{\delta^n}f-\mathbb{E}_nf|^2\bigg)^{1/2}.$$
%Let $\{\mathbb{E}_{n}f\}_{n\in\mathbb{Z}}$ be a martingale sequence defined in Definition~\ref{martingale sequence}.
%,  by the orthogonality of martingale differences, we have
%\begin{equation*}%\label{martingale equation}
%  \sum_{n\in\mathbb{Z}}\|\mathbb{D}_nf\|^2_{L^2}=\|f\|^2_{L^2}.
%\end{equation*}%As mentioned earlier, the proof is based on almost orthogonality principle. Hence
Set $n_2=\max\{n_0,n_1,n_{r_0}\}$. By Lemma~\ref{controlled by maximal operator}, we first have
\begin{equation*}
  \bigg\|\bigg(\sum_{n_{r_0}<n\le n_2}|A^\prime_{\delta^n}f-\mathbb{E}_nf|^2\bigg)^{1/2}\bigg\|_{L^2}\le 2(n_2-n_{r_0})\|f\|_{L^2}.
\end{equation*}
Note that $f=\sum_{n\in\mathbb{Z}}\mathbb{D}_{n}f$. Set $T_n=A^\prime_{\delta^n}-\mathbb{E}_n$, $u_n=v_n=\mathbb{D}_{n}f$, $N=n_2$ in Lemma~\ref{orthogonality principle}; it suffices for our purposes to prove that for every $n>n_2$ and $k\in\mathbb{Z}$, there exists a sequence  $a(k)$ of positive number with $w=\sum_{k\in\mathbb{Z}}a(k)<\8$ such that% the following norm inequality holds
\begin{equation}\label{L2-estimate of square function}
  \|A^\prime_{\delta^n}\mathbb{D}_{n+k}f-\mathbb{E}_n\mathbb{D}_{n+k}f\|_{L^2}\le a(k)\|\mathbb{D}_{n+k}f\|_{L^2}.
\end{equation}

In order to achieve our goal, we first set
\begin{equation}\label{k2}
  k_2=\max\{L_0+1,|k_1|,\min\{k\in\mathbb{N}:a_0\delta^{k-1}>1\}\},
\end{equation}
and then divide $k$ into three cases $-k_2\le k\leq k_2$, $k>k_2$ and $k<-k_2$.\\

\noindent{\bf{Case $-k_2\le k\le k_2$}}. Applying Lemma~\ref{controlled by maximal operator} for function $\mathbb{D}_{n+k}f$, we obtain
\begin{equation*}%\label{inequalities of square function}%+\|\mathbb{D}_{k+n}f\|_{L^2}
  \begin{split}
 \|A^\prime_{\delta^n}\mathbb{D}_{n+k}f-\mathbb{E}_n\mathbb{D}_{n+k}f\|_{L^2}\le {(D^{1/2}+1)}\|\mathbb{D}_{n+k}f\|_{L^2}.
  \end{split}
\end{equation*}

\noindent{\bf{Case $k>k_2$}}. Note that $\mathbb{E}_n\mathbb{D}_{n+k}f=\mathbb{D}_{n+k}f$, we write
\begin{equation*}
 \begin{split}
    \|A^\prime_{\delta^n}\mathbb{D}_{n+k}f-\mathbb{E}_n\mathbb{D}_{n+k}f\|^2_{L^2(G,m)}
    =\sum_\alpha\int_{Q_\alpha^{n+k-1}}|A^\prime_{\delta^n}\mathbb{D}_{n+k}f(x)-\mathbb{D}_{n+k}f(x)|^2dm(x).
 \end{split}
\end{equation*}
%Recall that $\mathcal{H}(B(x,\delta^n),Q^{n+k-1}_\alpha)$ denotes the set $\{x\in Q_\alpha^{n+k-1}:B(x,\delta^n)\cap (Q^{n+k-1}_\alpha)^{c}\neq \emptyset\}$.
 Fix a `dyadic cube' $Q^{n+k-1}_\alpha$. Note that $\mathbb{D}_{n+k}f(x)$ is a constant valued function on $Q^{n+k-1}_\alpha$, and so it follows that on $Q^{n+k-1}_\alpha$, $A^\prime_{\delta^n}\mathbb{D}_{n+k}f(x)-\mathbb{D}_{n+k}f(x)\neq 0$ only if $x$ belongs to the set $\mathcal{H}(B_{\delta^n},Q^{n+k-1}_\alpha)$. Applying Lemma~\ref{measure estimate of cube}, we have
\begin{equation}\label{ball}
  m(\mathcal{H}(B_{\delta^n},Q^{n+k-1}_\alpha))\le C_2\delta^{(1-k)\eta} m(Q^{n+k-1}_\alpha).
\end{equation}%By the same proof of \cite[Theorem A]{JKRW1}(or \cite[Theorem 2.3]{GXHTM}) $\mathcal{H}(B(x,\delta^n),Q^{n+k-1}_\alpha)\subseteq$there are finite atoms in $\mathfrak{}\mathcal{F}_{n+k-1}$ adjacent with $Q_\alpha^{n+k-1}$, then and $|\cdot|$ be the counting measure
Let $x\in \mathcal{H}(B_{\delta^n},Q^{n+k-1}_\alpha)$ and define the set $I_{x}=\{\beta:B(x,\delta^n)\cap Q_{\beta}^{n+k-1}\neq\emptyset\}$. Setting $I_{\alpha}=\cup_{x\in\mathcal{H}(B_{\delta^n},Q^{n+k-1}_\alpha)}I_x$. {Fix $\beta\in I_{\alpha}$, there exist a point $x\in\mathcal{H}(B_{\delta^n},Q^{n+k-1}_\alpha)$ and $y_0\in Q_{\beta}^{n+k-1}$ such that $y_0\in B(x,\delta^n)\cap Q_{\beta}^{n+k-1}$. Then for every $y\in Q_{\beta}^{n+k-1}$, by Proposition~\ref{dyadic cube}(iv), we have
\begin{equation*}
  d(y, z_{\alpha}^{n+k-1})\le d(y,y_0)+d(y_0, x)+d(x,z_{\alpha}^{n+k-1})\le 2C_1\delta^{n+k-1}+\delta^n.
\end{equation*}
It follows that for every $\beta\in I_{\alpha}$, $Q_\beta^{n+k-1}\subseteq B(z_{\alpha}^{n+k-1},2C_1\delta^{n+k-1}+\delta^n))$.  By Proposition~\ref{dyadic cube}(ii), the `dyadic cubes' $Q_\beta^k$ are disjoint, hence $\cup_{\beta\in I_{\alpha}}Q_\beta^{n+k-1}\subseteq B(z_\alpha^k,2C_1\delta^{n+k-1}+\delta^n)$. On the other hand, by Proposition~\ref{geometry-doubling}, $B(z_{\alpha}^{n+k-1},2C_1\delta^{n+k-1}+\delta^n))$ is covered by at most $D^{\log_2[(2C_1+1)/a_0]+1}$ balls of radius of $a_0\delta^{n+k-1}$. Since by Proposition~\ref{dyadic cube}(iv), we know that each $Q_{\beta}^{n+k-1}$ contains a ball $B(z^{n+k-1}_\beta,a_0\delta^{n+k-1})$, hence
\begin{equation}\label{measure}
  \#\{I_{\alpha}\}\le D^{\log_2[(2C_1+1)/a_0]+1},%\frac{m(B_{2C_1\delta^k})}{m(B_{a_0\delta^k})}\le (K+1)(2C_1/a_0)^\epsilon;
\end{equation}}
{Here and} in what follows, $\#\{A\}$ stands for the number of the set $A$. %Since the doubling property implies that for every point $x\in\mathcal{H}(B(x,\delta^n),Q^{n+k-1}_\alpha)$, there are finite numbers of  $Q_\beta^{n+k-1}\in\mathcal{F}_{n+k-1}$ such that $B(x,\delta^n)\cap Q_\beta^{n+k-1}\neq \emptyset$, we denote this number by $N$. Note that $N$ is independent of  $Q_\alpha^{n+k-1}$. Let $m_{Q_\alpha^{n+k-1}}$ be the maximum of $|\mathbb{D}_{n+k}f|$ on such $Q_\beta^{n+k-1}$, then~\eqref{ball} yields
Set
$$m_{\alpha}=\max_{\beta\in I_{\alpha}}|\mathbb{D}_{n+k}f(z_\beta^{n+k-1})|,$$
%\begin{equation*}%\label{square estimate+}
%  \|A_{\delta^n}\mathbb{D}_{n+k}f-\mathbb{D}_{n+k}f\|_{L^2(G,m)}\le C\delta^{-\frac{k\eta}{2}}\|\mathbb{D}_{n+k}f\|_{L^2(G,m)}.
%\end{equation*}
where $z_\beta^{n+k-1}$ is the point associated with the corresponding `dyadic cube' $Q_\beta^{n+k-1}$. {Since $\mathbb{D}_{n+k}f(x)$ is a constant-valued function on $Q_\beta^{n+k-1}$. It follows that
\begin{equation}\label{maximal}
  m_{\alpha}^2\le\sum_{\beta\in I_{\alpha}}\frac{1}{m(Q_\beta^{n+k-1})}\int_{Q_\beta^{n+k-1}}|\mathbb{D}_{n+k}f(x)|^2dm(x).%????????
\end{equation}}
By the above inequalities, Proposition~\ref{dyadic cube}(iv),~\eqref{ball} and \eqref{int}, we conclude
\begin{equation*}%\label{square estimate+}
   \begin{split}
     %\int_G|A_{\delta^n}\mathbb{D}_{n+k}f(x)-&\mathbb{D}_{n+k}f(x)|^2dm(x)=
     &\sum_\alpha\int_{Q_\alpha^{n+k-1}}|A^\prime_{\delta^n}\mathbb{D}_{n+k}f(x)
     -\mathbb{D}_{n+k}f(x)|^2dm(x)\\
&= \sum_\alpha \int_{\mathcal{H}(B_{\delta^n},Q^{n+k-1}_\alpha)}|A^\prime_{\delta^n}\mathbb{D}_{n+k}f(x)
     -\mathbb{D}_{n+k}f(x)|^2dm(x)\\
     &\le 2\sum_\alpha m_{\alpha}^2m(\mathcal{H}(B_{\delta^n},Q^{n+k-1}_\alpha))\\
     &\le 2C_2\delta^{(1-k)\eta}\sum_\alpha m_{\alpha}^2m(Q^{n+k-1}_\alpha)\\
     &\le {2C_2\delta^{(1-k)\eta}\sum_\alpha\sum_{\beta\in I_{\alpha}}\frac{m(Q^{n+k-1}_\alpha)}{m(Q_\beta^{n+k-1})}\int_{Q_\beta^{n+k-1}}|\mathbb{D}_{n+k}f(x)|^2dm(x)}\\
    &\le { 2C_2\delta^{(1-k)\eta}\sum_\alpha\sum_{\beta\in I_{\alpha}}\frac{m(B(z_\beta^{n+k-1},2C_1\delta^{n+k-1}+\delta^n))}{m(B(z_\beta^{n+k-1},a_0\delta^{n+k-1}))}
    \int_{Q_\beta^{n+k-1}}|\mathbb{D}_{n+k}f(x)|^2dm(x)}\\
    &\le{ 2C_2(K+1)\big({(2C_1+1)}/{a_0}\big)^\epsilon D^{\log_2[(2C_1+1)/a_0]+1}\delta^{(1-k)\eta}\int_{G}|\mathbb{D}_{n+k}f(x)|^2dm(x)}.%:=\tilde{C}^2\delta^{-k\eta}\int_G|\mathbb{D}_{n+k}f(x)|^2dm(x).
   \end{split}
\end{equation*}

\noindent{\bf{Case}} $k<-k_2$. Since $\mathbb{E}_{n}\mathbb{D}_{n+k}f=0$ and %and $diam(Q_\alpha^{n+k})\le C_1\delta^{k+n}<\delta^{n}$. %Moreover by the definition of expectation $\mathbb{E}_k$ (see Definition~\ref{martingale sequence}), for $n<0$, $\mathbb{D}_{k+n}f=0$. Thus we only need to consider the case $n\ge 0$ and the part $A_{\delta^n}\mathbb{D}_{n+k}f$.As noted above for $n+k<0$, $$
%Recall that $\mathcal{I}(B(x,\delta^n),n+k)$ denotes the set $\cup_\alpha\{Q_\alpha^{n+k}\cap B(x,\delta^n):\partial B(x,\delta^n)\cap Q_\alpha^{n+k}\neq\emptyset\}$.
 $\int_{Q^{k+n}_\alpha}\mathbb{D}_{n+k}f=0$ for every $Q^{k+n}_\alpha\in\mathcal{F}_{k+n}$, {thus for any $x\in G$,}%Recall that, for any ball $B$, $\mathcal{I}(B,n+k)$ denotes the set $\cup_\alpha\{B\cap Q_\alpha^{n+k}:\partial B\cap Q_\alpha^{n+k}\neq\emptyset\}$, from this, % $A_{\delta^n}\mathbb{D}_{n+k}f(x)\neq 0$ only if $x$ belongs to the set $\mathcal{I}(B(x,\delta^n),n+k)$ and
\begin{align*}%\label{square function estimate-}
   {|A^\prime_{\delta^n}\mathbb{D}_{n+k}f(x)|}&=|\frac{1}{m(B(x,\delta^n))}\sum_{\alpha}\int_{Q^{k+n}_\alpha\cap B(x,\delta^n)}\mathbb{D}_{n+k}f(y)dm(y)|\\
   &=|\frac{1}{m(B(x,\delta^n))}\int_{\mathcal{I}(B(x,\delta^n),n+k)}\mathbb{D}_{n+k}f(y)dm(y)|\\
   &\le M_{n+k}\mathbb{D}_{n+k}f(x).
\end{align*}
%where $\mathcal{I}(B(x,\delta^n),n+k)=\cup_\alpha\{B(x,\delta^n)\cap Q_\alpha^{n+k}:\partial B(x,\delta^n)\cap Q_\alpha^{n+k}\neq\emptyset\}$.
Form this and Lemma~\ref{lem:Akn}, we obtain
\begin{equation*}
  \|A^\prime_{\delta^n}\mathbb{D}_{n+k}f\|_{L^2}\le\|M_{n+k}\mathbb{D}_{n+k}f\|_{L^2}\le D_2 \delta^{k\epsilon/2}\|\mathbb{D}_{n+k}f\|_{L^2}.
\end{equation*}
Therefore, for every $n>n_2$,  we determine{
\begin{equation*}
a(k)= \left\{
    \begin{array}{ll}
     \bigg( 2C_2(K+1)({(2C_1+1)}/{a_0}\big)^\epsilon D^{\log_2[(2C_1+1)/a_0]+1}\delta^{\eta}\bigg)^{1/2}\delta^{-\eta k/2}, & \hbox{$k>k_2$,} \\
      D^{1/2}+1, & \hbox{$-k_2\le k\le k_2$,} \\
     D_2\delta^{\epsilon k/2}, & \hbox{$k<-k_2$,}
    \end{array}
  \right.
\end{equation*}}%
and $\sum_{k\in\mathbb{Z}}a(k)<\8$, which completes the proof.
\end{proof}

The proof of strong type $(2,2)$ estimate for $SV(f)=\bigg(\sum_{n\ge n_{r_0}}|SV_n(f)|^2\bigg)^{1/2}$ is similar in spirit to that of the square function $S(f)$. %The proof is also based on Lemma~\ref{orthogonality principle}.%By Lemma~\ref{controlled by maximal operator}, we know that $\mathbb{S}_nf\le CMf$. short variation operator in Theorem~\ref{the estimate of square function} , where $SV_n(f)=V_2(A_rf:r\in[\delta^n,\delta^{n+1}))$
\begin{proof}[Proof of \eqref{strong-type inequalities of short variation} in the case $p=2$.]
First by Lemma~\ref{controlled by maximal operator}, we have
\begin{equation*}
  \bigg\|\bigg(\sum_{n_{r_0}\le n\le n_2}|SV_n(f)|^2\bigg)^{1/2}\bigg\|_{L^2}\le 2(n_2-n_{r_0}+1)C_4\|f\|_{L^2}.
\end{equation*}
Set $T_n=SV_n$, $u_n=v_n=\mathbb{D}_{n}$, $N=n_2$ in Lemma~\ref{orthogonality principle}. For every $n>n_2$, we divide $k$ into three cases $-k_2\le k\leq k_2$, $k> k_2$ and $k<-k_2$, where $k_2$ is defined in~\eqref{k2}.\\

\noindent{\bf{Case}} $-k_2\le k\leq k_2$. For function $\mathbb{D}_{n+k}f$, using Lemma~\ref{controlled by maximal operator} again, we have%since $V_2(A_{r}f(x):r\in[\delta^{n},\delta^{n+1}))\le CMf$, so by the $L^2$-boundedness of Hardy-Littlewood maximal function, one can see that
\begin{equation*}%\label{L2-estimate of short variation}
  \|V_2(A^\prime_{r}\mathbb{D}_{n+k}f:r\in[\delta^{n},\delta^{n+1}))\|_{L^2}\le C_4\|\mathbb{D}_{n+k}f\|_{L^2}.
\end{equation*}
%hence we can choice $a(k)=1$ in this case.
%We will finish the proof,  if we can decide $a(k)$ in other two cases satisfying Lemma~\ref{orthogonality principle}'s condition. We now prove our desired results.

\noindent{\bf{Case}} $k>k_2$.  We write
\begin{equation*}
 \begin{split}
    \|&V_2(A^\prime_{r}\mathbb{D}_{n+k}f:r\in[\delta^{n},\delta^{n+1}))\|^2_{L^2(G,m)}\\
    &=\sum_\alpha\int_{Q_\alpha^{n+k-1}}|V_2(A^\prime_{r}\mathbb{D}_{n+k}f(x):r\in[\delta^{n},\delta^{n+1}))|^2dm(x).
 \end{split}
\end{equation*}%intersects the complement of $Q_\alpha^{n+k-1}$,
%Recall that  $\mathcal{H}(B(x,r),Q_\alpha^{n+k-1})$ denotes the set $\{x\in Q^{n+k-1}_\alpha:B(x,r)\cap (Q^{n+k-1}_\alpha)^{c}\neq \emptyset\}$.
Fix a `dyadic cube' $Q_\alpha^{n+k-1}$.  Recall that $\mathbb{D}_{n+k}f(x)$ is a constant valued funtion on $Q_\alpha^{n+k-1}$, it follows that for every $\delta^n\le r_i<r_{i+1}<\delta^{n+1}$ and $x\in Q_\alpha^{n+k-1}$, $|A^\prime_{r_{i+1}}\mathbb{D}_{n+k}f(x)-A^\prime_{r_i}\mathbb{D}_{n+k}f(x)|\neq 0$ only if there exists at least one ball $B(x,r_i)$ or $B(x,r_{i+1})$ intersecting with $(Q^{n+k-1}_\alpha)^{c}$. {So $V_2(A^\prime_{r}\mathbb{D}_{n+k}f:r\in[\delta^{n},\delta^{n+1}))$} is supported on $\mathcal{H}(B_{\delta^{n+1}},Q^{n+k-1}_\alpha)$.
 %where $\mathcal{H}(B_{\delta^{n+1}},Q^{n+k-1}_\alpha)=\{x\in Q^{n+k-1}_\alpha:B(x,\delta^{n+1})\cap (Q^{n+k-1}_\alpha)^{c}\neq \emptyset\}$. % Hence for $x\in Q_\alpha^{n+k-1}$, the support of $V_2(A_{r}\mathbb{D}_{n+k}f(x):r\in[\delta^{n},\delta^{n+1}))$ is belong to the set $\mathcal{H}(B(x,\delta^{n+1}),Q^{n+k-1}_\alpha)$, and
By Lemma~\ref{measure estimate of cube}, we have %{measure estimate of cube}
\begin{equation}\label{ball2}
m(\mathcal{H}(B_{\delta^{n+1}},Q^{n+k-1}_\alpha))\le C_2\delta^{(2- k)\eta} m(Q^{n+k-1}_\alpha).
\end{equation}
%\begin{align*}
%  &\int_{Q_\alpha^{n+k-1}}|V_2(A_{r}\mathbb{D}_{n+k}f(x):r\in[\delta^{n},\delta^{n+1}))|^2dm(x)\\
%&\le\int_{\mathcal{H}(B(x,\delta^{n+1}),Q_\alpha^{n+k-1})}|V_2(A_{r}\mathbb{D}_{n+k}f(x):r\in[\delta^{n},\delta^{n+1}))|^2dm(x).
%\end{align*}
%Moreover $\cup_i\mathcal{H}(B(x,r_i),Q_\alpha^{n+k-1})\subseteq \mathcal{H}(B(x,\delta^{n+1}),Q_\alpha^{n+k-1})$ and
%Since  $V_2(A_{r}\mathbb{D}_{n+k}f(x):r\in[\delta^{n-1},\delta^n))$ is controlled by the summation of $S_I$ and $S_{II}$
 On the other hand, by~\eqref{controll the short variation}, we know that $V_2(A^\prime_{r}\mathbb{D}_{n+k}f(x):r\in[\delta^{n},\delta^{n+1}))$ is controlled by the sum of $SV_I(\mathbb{D}_{n+k}f)(x)$ and $SV_{II}(\mathbb{D}_{n+k}f)(x)$, where
\begin{equation*}%\label{definition of SI}
 \begin{split}
   SV_I(\mathbb{D}_{n+k}f)(x)&=\bigg(\sup_{\delta^{n}\le r_0<\cdots<r_J<\delta^{n+1}}\sum_{i=1}^J\frac{1}{m(B(x,r_i))^2}\big|\int_{B(x,r_i)\setminus B(x,r_{i-1})}\mathbb{D}_{n+k}f(y)dm(y)\big|^2\bigg)^{1/2}
 \end{split}
\end{equation*}
and
\begin{equation*}%\label{definition of SII}
 \begin{split}
   &SV_{II}(\mathbb{D}_{n+k}f)(x)=\\
&\bigg(\sup_{\delta^{n}\le r_0<\cdots<r_J<\delta^{n+1}}\sum_{i=1}^J
\big|\big(\frac{1}{m(B(x,r_{i-1}))}-\frac{1}{m(B(x,r_i))}\big)\int_{B(x,r_{i-1})}\mathbb{D}_{n+k}f(y)dm(y)\big|^2\bigg)^{1/2}.
 \end{split}
\end{equation*}%by the definition of $k_0$, we have $diam(Q_\alpha^{n+k-1})>2\delta^n$.
Let $x\in \mathcal{H}(B_{\delta^{n+1}},Q^{n+k-1}_\alpha)$ and set $I_{x}=\{\beta:B(x,\delta^{n+1})\cap Q_{\beta}^{n+k-1}\neq\emptyset\}$. We define the set
~$I_\alpha=\cup_{x\in\mathcal{H}(B_{\delta^{n+1}},Q^{n+k-1}_\alpha)}I_x$. {Similar to}~\eqref{measure}, we have {$\#\{I_{\alpha}\}\le  D^{\log_2[(2C_1+1)/a_0]+1}$}. Write
$$m_{\alpha}=\max_{\beta\in I_\alpha}|\mathbb{D}_{n+k}f(z_\beta^{n+k-1})|.$$
{It follows from~\eqref{int} that
\begin{equation*}
  \begin{split}
   SV_I&(\mathbb{D}_{n+k}f)^2(x)\\
   &\le \frac{m(B(x,\delta^{n+1})\setminus B(x,\delta^{n}))}{m(B(x,\delta^{n}))^2} \sup_{\delta^{n}\le r_0<\cdots<r_J<\delta^{n+1}}\sum_{i=1}^J\int_{B(x,r_i)\setminus B(x,r_{i-1})}|\mathbb{D}_{n+k}f(y)|^2dm(y)\\
   &\le \bigg(\frac{m(B(x,\delta^{n+1})\setminus B(x,\delta^{n}))}{m(B(x,\delta^{n}))}\bigg)^2m_{\alpha}^2\le (K+1)^2\delta^{2\epsilon} m_{\alpha}^2
  \end{split}
\end{equation*}}
 and
\begin{equation*}
 \begin{split}
   &SV_{II}(\mathbb{D}_{n+k}f)^2(x)\\
   &\le  m_{\alpha}^2m(B(x,\delta^{n+1}))^2 \bigg(\sup_{\delta^{n}\le r_0<\cdots<r_J<\delta^{n+1}}\sum_{i=1}^J
|\frac{1}{m(B(x,r_{i-1}))}-\frac{1}{m(B(x,r_i)))}|\bigg)^2\\
&\le  m_{\alpha}^2m(B(x,\delta^{n+1}))^2\bigg(\frac{1}{m(B(x,\delta^{n}))}-\frac{1}{m(B(x,\delta^{n+1}))}\bigg)^2\le (K+1)^2\delta^{2\epsilon} m_{\alpha}^2.
 \end{split}
\end{equation*}
{Combining the above two inequalities with~\eqref{maximal} and~\eqref{ball2}}, we have
\begin{equation*}
  \begin{split}
    &\int_{Q_\alpha^{n+k-1}}|V_2(A^\prime_{r}\mathbb{D}_{n+k}f(x):r\in[\delta^{n},\delta^{n+1}))|^2dm(x)\\
&=\int_{\mathcal{H}(B_{\delta^{n+1}},Q^{n+k-1}_\alpha)}|V_2(A^\prime_{r}\mathbb{D}_{n+k}f(x):r\in[\delta^{n},\delta^{n+1}))|^2dm(x)\\
    &\le 2\int_{\mathcal{H}(B_{\delta^{n+1}},Q_\alpha^{n+k-1})}SV_I(\mathbb{D}_{n+k}f)^2(x)+SV_{II}(\mathbb{D}_{n+k}f)^2(x)dm(x)\\
    &\le 4(K+1)^2\delta^{2\epsilon} m(\mathcal{H}(B_{\delta^{n+1}},Q_\alpha^{n+k-1}))m_{\alpha}^2\\
    &\le{4C_2(K+1)^2\delta^{2\epsilon}\delta^{(2-k)\eta}\sum_{\beta\in I_{\alpha}}\frac{m(Q_\alpha^{n+k-1})}{m(Q_\beta^{n+k-1})}\int_{Q_{\beta}^{n+k-1}}|\mathbb{D}_{n+k}f(x)|^2dm(x)}\\
    &\le 4C_2(K+1)^2\big((2C_1+1)/a_0\big)^\epsilon \delta^{2\epsilon}\delta^{(2-k)\eta}\sum_{\beta\in I_{\alpha}}\int_{Q_{\beta}^{n+k-1}}|\mathbb{D}_{n+k}f(x)|^2dm(x),
  \end{split}
\end{equation*}%-\frac{\mu(\mathcal{I}(B(x,\delta^n),Q_\alpha^{j+n-1}))}{\mu(B(x,\delta^{n-1}))}
and summing over all $\alpha$ shows,
\begin{align*}%\label{short variation estimate+}
    &\|V_2(A^\prime_{r}\mathbb{D}_{n+k}f:r\in[\delta^{n},\delta^{n+1}))\|_{L^2}\\
    &\le  \bigg(4C_2(K+1)^2\big((2C_1+1)/a_0\big)^\epsilon D^{\log_2[(2C_1+1)/a_0]+1}\delta^{2\epsilon}\delta^{(2-k)\eta}\bigg)^{1/2}\|\mathbb{D}_{n+k}f\|_{L^2}.
\end{align*}%By our assumptions of measure $\mu$ and metric $d$ , and following the proof of \eqref{square estimate-}, it suffices to
%and so for $k> k_0$,  we can set $a(k)=\delta^{-\frac{\eta k}{2}}$. Recall that, for any ball or ball annular $B$, $\mathcal{I}(B,n+k)$ denotes the set $\cup_\alpha\{B\cap Q_\alpha^{n+k}:\partial B\cap Q_\alpha^{n+k}\neq\emptyset\}$.

\noindent{\bf{Case}} $k<-k_2$. Let $x\in G$. Note that for every $\delta^n\le r_{i-1}<r_i\le\delta^{n+1}$,
\begin{equation*}%\label{square function estimate-}
\begin{split}
&\int_{B(x,r_i))\setminus B(x,r_{i-1})}\mathbb{D}_{n+k}f(y)dm(y)=\int_{\mathcal{I}(B(x,r_i))\setminus B(x,r_{i-1}),n+k)}\mathbb{D}_{n+k}f(y)dm(y),\\
   &\int_{B(x,r_{i-1})}\mathbb{D}_{n+k}f(y)dm(y)=\int_{\mathcal{I}(B(x,r_{i-1}),n+k)}\mathbb{D}_{n+k}f(y)dm(y),
 \end{split}
\end{equation*}
so $SV_I(\mathbb{D}_{n+k}f)(x)= M^S_{n+k}(\mathbb{D}_{n+k}f)(x)$. Moreover, using the estimate~\eqref{int} and the fact that the $\ell^1$-norm is greater than the $\ell^2$-norm, we obtain%one can see that
\begin{align*}
&SV_{II}(\mathbb{D}_{n+k}f)(x)\\
&\le\sup_{\delta^{n}\le r_0<\cdots<r_J<\delta^{n+1}}\sum_{i=1}^J
\big|\big(\frac{1}{m(B(x,r_{i-1}))}-\frac{1}{m(B(x,r_i))}\big)\int_{\mathcal{I}(B(x,r_{i-1}),n+k)}\mathbb{D}_{n+k}f(y)dm(y)\big|\\
&\le (K+1)\delta^\epsilon \sup_{\delta^{n}\le r_0<\cdots<r_J<\delta^{n+1}}\sum_{i=1}^J
\big|\frac{m(B(x,\delta^{n+1}))}{m(B(x,r_{i-1}))}-\frac{m(B(x,\delta^{n+1}))}{m(B(x,r_i))}\big|{M}_{n+k}(\mathbb{D}_{n+k}f)(x)\\
& \le (K+1)^2\delta^{2\epsilon}{M}_{n+k}(\mathbb{D}_{n+k}f)(x).%(\frac{m(x,\delta^{n+1})}{m(B(x,\delta^n))}-1)
\end{align*}
Using Lemma~\ref{lem:Mkn} and Lemma~\ref{lem:Akn}, respectively, we have
\begin{align*}%\label{short variational estimate-}
  \|&V_2(A^\prime_{r}\mathbb{D}_{n+k}f:r\in[\delta^{n},\delta^{n+1}))\|_{L^2}\le\|SV_I(\mathbb{D}_{n+k}f)\|_{L^2}+\|SV_{II}(\mathbb{D}_{n+k}f)\|_{L^2}\\
  &\le \|M^S_{n+k}(\mathbb{D}_{n+k}f)\|_{L^2}+(K+1)^2\delta^{2\epsilon}\|{M}_{n+k}(\mathbb{D}_{n+k}f)\|_{L^2}\\
  &\le(C_3+ (K+1)^2\delta^{2\epsilon} D_2)\delta^{\epsilon k/2}\|\mathbb{D}_{n+k}f\|_{L^2}.
\end{align*}
Hence, for every $n>n_2$, we determine
\begin{equation*}
a(k)= \left\{
    \begin{array}{ll}
     {\bigg(4C_2(K+1)^2\big((2C_1+1)/a_0\big)^\epsilon D^{\log_2[(2C_1+1)/a_0]+1}\delta^{2(\epsilon+\eta)}\bigg)^{1/2}\delta^{-k\eta/2}}, & \hbox{$k>k_2$,} \\
      C_4, & \hbox{$-k_2\le k\le k_2$,} \\
    (C_3+ (K+1)^2\delta^{2\epsilon} D_2)\delta^{\epsilon k/2}, & \hbox{$k<-k_2$,}
    \end{array}
  \right.
\end{equation*}
and $\sum_k a(k)<\8$, which completes the proof.%
\end{proof}

\section{Weak type $(1,1)$ estimates}\label{ST4}
%Note that for~\eqref{deal with variational operator}, if we prove that operators $S_lf$ and $S_sf$ are  bounded from $L^1(X)$ to $L^{1,\8}(X)$, then so is $V_q(A_rf:r>0)$.Let us introduce some notations which will be used later.
%Let $\{Q_\alpha^k\}_{(\alpha,k)}$ be a family of "dyadic cube" defined in Proposition~\ref{dyadic cube} and $\mathcal{F}_k$ be a $\sigma$-algebra generated by the "dyadic cubes" $\{Q_\alpha^{k}\}_{\alpha\in I_k}$.  Associating with the $\sigma$-algebra $\mathcal{F}_k$, the dyadic maximal function is defined
In this section, we prove that the square function $f\rightarrow S(f)$ and the short variation operator $f\rightarrow SV(f)$ are of weak type $(1,1)$.

Under the conditions \eqref{decay property} and \eqref{geo-doubling}, the space $(G,d,m)$ might not be a doubling measure space and the usual Calder\'on-Zygmund decomposition does not work any more. We need the following version  of Calder\'{o}n-Zygmund decomposition, which is motivated by Gundy's decomposition from martingale theory. This decomposition was constructed in~\cite{Lop-Mar-Pra14} for non-doubling measure on $\real^d$, and construction works without alterations for the space $(G,d,m)$.

Let $f:G\rightarrow \mathbb{C}$ be a locally integrable function. The `dyadic' maximal function is defined by
\begin{equation*}
  M_df(x)=\sup_{k\in\mathbb{Z}}\mathbb{E}_k(|f|)(x).%\sup_{(\alpha,k),x\in Q^k_{\alpha}}\frac{1}{\mu(Q_\alpha^k)}\int_{Q_\alpha^k}f(y)d\mu(y);
\end{equation*}
 Given a `dyadic cube' $Q_\alpha^k$, let $\widetilde{Q}_\alpha^k=\{y\in G:d(y,z_\alpha^k)\le 3C_1\delta^{k+1}\}$ and $\widehat{Q}_\alpha^k$ be its parent. We denote  $\widehat{z}_\alpha^k$ the corresponding point of $\widehat{Q}_\alpha^k$.  Set $\langle f\rangle_{Q_\alpha^k}=\frac{1}{m(Q_\alpha^k)}\int_{Q_\alpha^k}f$.

%we require much more work and details on the weak type (1,1) estimates due to the lack of information for the small ball (the radius is less than $r_0$).
\begin{lem}\label{C-Z decomposition}
Let $f\in L^1(G,m)$, and $\gamma>0$, let%denote the set $\Omega=\{Q_\alpha^k:M_df>\gamma\}$.
\begin{equation*}
  \{x\in G:M_df(x)>\gamma\}=\cup_{k\in\mathbb{Z}}\Omega_k,
\end{equation*}
where
\begin{equation*}
  \Omega_k=\{x\in G:\mathbb{E}_k(|f|)(x)>\gamma, \mathbb{E}_j(|f|)(x)\le\gamma, j>k\}.
\end{equation*}
Then we decompose
\begin{equation*}
  \Omega=\cup_{k\in\mathbb{Z}}\Omega_k=\cup_{k\in\mathbb{Z}}(\cup_{\alpha\in\Lambda_k} Q_\alpha^k)%=\cup_{(\alpha,k)\in\Lambda}Q_\alpha^k
\end{equation*}
into a disjoint union of maximal `dyadic cubes' $Q_\alpha^k$, where $\{\Lambda_k\}_k$ stands for the sequence of the corresponding index set. Let% Moreover the family of maximal "dyadic cubes" $\{Q_\alpha^k\}_{(\alpha,k)\in\Lambda}$ satisfies
%{\color{red}{
%\begin{equation*}
%  \gamma{m(Q_\alpha^k)}\le\int_{Q_\alpha^k}|f(y)|dm(y)\le (C_1\delta/a_0)^\epsilon
%    {\gamma}m(Q_\alpha^k),~\forall~(\alpha,k)\in\Lambda.
%\end{equation*}}}
% exists a index set $\Gamma$ such that %$\Omega^c$ can be decomposed into a sequence of disjoint union "dyadic cubes" $Q_\alpha^{n_3}$ such that
%\begin{equation*}
% \Omega^c=\cup_{\beta\in\Gamma}Q_\beta^{n_3},~\textit{and}~\frac{1}{m(Q_\beta^{n_3})}\int_{Q_\beta^{n_3}}|f(y)|dm(y)\le\gamma,
% ~\forall~\beta\in\Gamma.(\Gamma,n_3)\cup
%\end{equation*}
\begin{equation*}
  \begin{split}
    &g(x)=f\mathds{1}_{\Omega^c}+\sum_{k}\sum_{\alpha\in\Lambda_k}\langle f \rangle_{\widehat{Q}_\alpha^k}\mathds{1}_{Q_\alpha^k}(x)+\sum_{k}\sum_{\alpha\in\Lambda_k}\big(\langle f\rangle_{Q_\alpha^k}-\langle f\rangle_{\widehat{Q}_\alpha^k}\big)\frac{m(Q_\alpha^k)}{m(\widehat{Q}_\alpha^k)}\mathds{1}_{\widehat{Q}_\alpha^k}(x),\\
    &b(x)=\sum_{k}b_k=\sum_{k}\sum_{\alpha\in\Lambda_k}b_{\alpha}^{k}(x)=\sum_{k}\sum_{\alpha\in\Lambda_k}\big(f(x)-\langle f \rangle_{{Q}_\alpha^k}\big)\mathds{1}_{Q_\alpha^k}(x),\\
    &\xi(x)=\sum_{k}\xi_k=\sum_{k}\sum_{\alpha\in\Lambda_k}\xi_{\alpha}^{k}(x)=\sum_{k}\sum_{\alpha\in\Lambda_k}\big(\langle f\rangle_{Q_\alpha^k}-\langle f\rangle_{\widehat{Q}_\alpha^k}\big)\bigg(\mathds{1}_{Q_\alpha^k}(x)-\frac{m(Q_\alpha^k)}{m(\widehat{Q}_\alpha^k)}\mathds{1}_{\widehat{Q}_\alpha^k}(x)\bigg).
  \end{split}
\end{equation*}
Then
\begin{enumerate}[\noindent]%\sum_{{\beta\in\Lambda_1}}\frac{1}{m(Q_\beta^{n_3})}\int_{Q_\beta^{n_3}}|f(y)|dm(y)\mathds{1}_{Q_\alpha^k}(x)+
   \item\emph{(i)}~$f=g+b+\xi$;
    \item\emph{(ii)}~let $p\in[1,\8)$ and $m=[p]+1$, the function $g$ satisfies
    $$\|g\|^p_{L^p}\le 3\cdot2^p(m!)^{\frac{p-1}{m-1}}\gamma^{p-1}\|f\|_{L^1};$$
    \item \emph{(iii)}~the function $b$ satisfies
     $$\int_G b_\alpha^k=0,~\|b\|_{L^1}=\sum_{k}\sum_{\alpha\in\Lambda_k}\|b_\alpha^k\|_{L^1}\le 2\|f\|_{L^1};$$
     \item\emph{(iv)}~the function $\xi$ satisfies
     $$\int_G \xi_\alpha^k=0,~\|\xi\|_{L^1}=\sum_{k}\sum_{\alpha\in\Lambda_k}\|\xi_\alpha^k\|_{L^1}\le 4\|f\|_{L^1}.$$
  \end{enumerate}
\end{lem}

With the above decomposition, the almost orthogonality principle which was exploited well in~\cite{JKRW98} or~\cite{GXHTM17}, doesn't seem to work. Thereby, we have to provide another method to achieve our goal. A new input is the observation that the operators $S$ and $SV$ are essentially dyadic operator, in particular they are perfect dyadic in the small scale---small cubes {\it versus} large balls.

We first deal with the operator $S$.

\begin{proof}[Proof of \eqref{weak-type inequalities of square function}]
Fix $f\in L^1(G,m)$ and $\gamma>0$. Keeping the notations in Lemma~\ref{C-Z decomposition}, we get $f=g+b+\xi$. We first have%we set $f=g+\tilde{b}$ where $\tilde{b}=b+\xi$. Then
\begin{align*}
  m\big(\{x\in G:S(f)(x)>\gamma\}\big)&\le m\big(\{x\in G:S(g)(x)>{\gamma}/{3}\}\big)\\
  &+m\big(\{x\in G:S(b)(x)>{\gamma}/{3}\}\big)+m\big(\{x\in G:S(\xi)(x)>{\gamma}/{3}\}\big).
\end{align*}
By the $L^2$-boundedness of the operator $S$ and Lemma~\ref{C-Z decomposition}(ii), the first term on the right side is controlled by
\begin{equation*}
  \begin{split}
   m\big(\{x\in G:S(g)(x)>{\gamma}/{3}\}\big)&\le \frac{9}{\gamma^2}\int_G|S(g)(x)|^2dm(x)\le\frac{9c_2}{\gamma^2}\int_G|g(x)|^2dm(x)\\
   &\le \frac{108\sqrt{6} c_2}{\gamma}\int_G|f(x)|dm(x).
  \end{split}
\end{equation*}
 It remains to handle the other two terms, we  set
\begin{equation*}%\label{set}
  k_3=\min\{k:a_0\delta^k>r_0\},~\widetilde{\Omega}=\big(\cup_{k<k_3}\Omega_k\big)\cup\big(\cup_{k\ge k_3}\widetilde{\Omega}_k\big),
\end{equation*}
where $\widetilde{\Omega}_k=\cup_{\alpha\in\Lambda_k} \widetilde{Q}_\alpha^k$.  By~\eqref{int} and the fact that $M_d$ is of weak type $(1,1)$,  $m(\widetilde{\Omega})$ is controlled by
\begin{equation*}%\label{measure of set}
\begin{split}
  m(\widetilde{\Omega})&=\sum_{k<k_3}m(\Omega_k)+\sum_{k\ge k_3}m(\widetilde{\Omega}_k)\le \sum_{k<k_3}\sum_{\alpha\in\Lambda_k}m(Q_\alpha^k)+\sum_{k\ge k_3}\sum_{\alpha\in\Lambda_k}m(\widetilde{Q}_\alpha^k)\\
  &\le\sum_{k<k_3}\sum_{\alpha\in\Lambda_k}m(Q_\alpha^k)+\sum_{k\ge k_3}\sum_{\alpha\in\Lambda_k}\frac{m(\widetilde{Q}_\alpha^k)}{m(Q_\alpha^k)}m(Q_\alpha^k)\\
  &\le \sum_{k<k_3}\sum_{\alpha\in\Lambda_k}m(Q_\alpha^k)+(K+1)(3C_1\delta/a_0)^{\epsilon}\bigg(\sum_{k\ge k_3}\sum_{\alpha\in\Lambda_k}m(Q_\alpha^k)\bigg)\\
  &\le \frac{(K+1)(3C_1\delta/a_0)^{\epsilon}}{\gamma}\|f\|_{L^1}.
\end{split}
\end{equation*}

We now focus on the term $m\big(\{x\in G\setminus\widetilde{\Omega}:S(\xi)(x)>{\gamma}/{3}\}\big)$. Recall that $n_2=\max\{n_0, n_1\}$,  we decompose
\begin{equation*}%=m\big(\{x\in G\setminus\tilde{\Omega}:\bigg(\sum_{n\in\mathbb{N}}|A_{\delta^n}b(x)|^2\bigg)^{1/2}>\frac{\gamma}{2}\}\big)
\begin{split}
     m\big(\{x\in G\setminus\widetilde{\Omega}:S(\xi)(x)>{\gamma}/{3}\}\big)&\le m\bigg(\{x\in G\setminus\widetilde{\Omega}:\big(\sum_{n_{r_0}<n\le n_2}|A_{\delta^n}\xi(x)-\mathbb{E}_n\xi(x)|^2\big)^{1/2}>{\gamma}/{6}\}\bigg)\\
     &+ m\bigg(\{x\in G\setminus\widetilde{\Omega}:\big(\sum_{n>n_2}|A_{\delta^n}\xi(x)-\mathbb{E}_n\xi(x)|^2\big)^{1/2}>{\gamma}/{6}\}\bigg).
\end{split}
\end{equation*}
For the first part of the right side of the above inequality, using Lemma~\ref{controlled by maximal operator} and Lemma~\ref{C-Z decomposition}(iv), we have
\begin{equation*}
  \begin{split}
    m&\bigg(\{x\in G\setminus\widetilde{\Omega}:\big(\sum_{n_{r_0}<n\le n_2}|A^\prime_{\delta^n}\xi(x)-\mathbb{E}_n\xi(x)|^2\big)^{1/2}>{\gamma}/{6}\}\bigg)\\
&\le\sum_{n_{r_0}<n\le n_2}\frac{6}{\gamma} \int_{G\setminus\widetilde{\Omega}}|A^\prime_{\delta^n}\xi(x)-\mathbb{E}_n\xi(x)|dm(x)\\
&\le\frac{12(n_2-n_{r_0})}{\gamma}\|\xi\|_{L^1(G,m)}\le\frac{48(n_2-n_{r_0})}{\gamma}\|f\|_{L^1(G,m)}.
  \end{split}
\end{equation*}
For the second part, we first have
\begin{equation}\label{bad-function}
\begin{split}
   &m\bigg(\{x\in G\setminus\widetilde{\Omega}:\big(\sum_{n>n_2}|A^\prime_{\delta^n}\xi(x)-\mathbb{E}_n\xi(x)|^2\big)^{1/2}>{\gamma}/{6}\}\bigg)\\
&\le\frac{6}{\gamma}\int_{ G\setminus\widetilde{\Omega}}\bigg(\sum_{n>n_2}|A^\prime_{\delta^n}\xi(x)-\mathbb{E}_n\xi(x)|^2\bigg)^{1/2}dm(x)\\
   &\le \frac{6}{\gamma}\int_{ G\setminus\widetilde{\Omega}}\sum_{n>n_2}|A^\prime_{\delta^n}\xi(x)-\mathbb{E}_n\xi(x)|dm(x)\\
&\le\frac{6}{\gamma}\sum_{n>n_2}\sum_{k\in\mathbb{Z}}\sum_{\alpha\in\Lambda_{n+k}}\int_{G\setminus {\widetilde{\Omega}}}|A^\prime_{\delta^n}\xi_\alpha^{n+k}(x)-\mathbb{E}_n\xi_\alpha^{n+k}(x)|dm(x).
\end{split}
\end{equation}
We now deal with integral term $\int_{G\setminus \widetilde{\Omega}}|A_{\delta^n}\xi_\alpha^{n+k}(x)-\xi_\alpha^{n+k}(x)|dm(x)$. Let
$$k_4=\max\{k:k<0~\&~C_1\delta^{k+1}\le 1\},~k_5=\max\{|k_4|,k_3\}.$$
We split the $k$ into three cases: $-k_5\le k\le k_5$, $k>k_5$, and $k<-k_5$. We will prove that
%\begin{equation}\label{bad}
% \int_{G\setminus \widetilde{\Omega}}|A^\prime_{\delta^n}\xi_\alpha^{n+k}-\mathbb{E}_{n}\xi_\alpha^{n+k}|\le a(k)\|\xi_\alpha^{n+k}\|_{L^1(G,m)},
%\end{equation}
%where
\begin{equation*}%\label{bad}
 \int_{G\setminus \widetilde{\Omega}}|A^\prime_{\delta^n}\xi_\alpha^{n+k}(x)-\mathbb{E}_{n}\xi_\alpha^{n+k}(x)|dm(x)=a(k)\|\xi_\alpha^{n+k}\|_{L^1(G,m)},
\end{equation*}
where
\begin{equation}\label{bad}
 a(k)=
\left\{
 \begin{array}{ll}
{(D+1)}, & \hbox{$-k_5\le k\le k_5$;} \\
 0, & \hbox{$k>k_5$;} \\
   {6^{\epsilon}(K+1)^2C_1^\epsilon DK_\epsilon\delta^{\epsilon k}}, & \hbox{$k<-k_5$.}
\end{array}
     \right.
\end{equation}
Assume this result momentarily. Combining~\eqref{bad-function}, \eqref{bad} with Lemma~\ref{C-Z decomposition}(iv), one can see immediately that
\begin{align*}
&m\bigg(\{x\in G\setminus\widetilde{\Omega}:\big(\sum_{n>n_2}|A^\prime_{\delta^n}\xi(x)-\mathbb{E}_n\xi(x)|^2\big)^{1/2}>{\gamma}/{6}\}\bigg)\le
\frac{6}{\gamma}\sum_{n>n_2}\sum_{k\in\mathbb{Z}}\sum_{\alpha\in\Lambda_{n+k}}a(k)\|\xi_\alpha^{n+k}\|_{L^1(G,m)}\\
&\le \frac{6}{\gamma}\bigg(\sum_{k\in\mathbb{Z}}a(k)\bigg)\bigg(\sum_{n\in\mathbb{Z}}\sum_{\alpha\in\Lambda_{n}}\|\xi_\alpha^{n}\|_{L^1(G,m)}\bigg)\le
 \frac{C_{\epsilon,\delta}}{\gamma}\|f\|_{L^1(G,m)}.
 \end{align*}

%By the definition of $b$, let $\tilde{b}=\sum_{k\in\mathbb{Z}}h_k$ where $h_k=b_k+\phi_k$. Then by the (iii) and (iv) of Lemma~\ref{C-Z decomposition}, we have
We now prove~\eqref{bad}.\\

{\noindent \bf{Case}} $-k_5\le k\le k_5$. Using Lemma~\ref{controlled by maximal operator} for function $\xi_\alpha^{n+k}$, we have
\begin{align*}
  \int_{G\setminus\widetilde{\Omega}}|A^\prime_{\delta^n}\xi_\alpha^{n+k}(x)-\mathbb{E}_n\xi_\alpha^{n+k}(x)|dm(x)&\le\int_{G}|A^\prime_{\delta^n}\xi_\alpha^{n+k}(x)
-\mathbb{E}_n\xi_\alpha^{n+k}(x)|dm(x)\\
&\le{(D+1)}\|\xi_{\alpha}^{n+k}\|_{L^1(G,m)}.
\end{align*}%{\color{red}it follows that $\widetilde{Q}_\alpha^{n+k}=\{y\in G:d(y,z_\alpha^{n+k})\le 3C_1\delta^{n+k+1}\}$}Note that {$\xi_{\alpha}^{n+k}$ is supported on $\widehat{Q}_\alpha^{n+k}$.

{\noindent\bf{Case}} $k>k_5$. Note that $\xi^{n+k}_\alpha$ is supported on $\widehat{Q}_\alpha^{n+k}$. {Recall that $\widetilde{Q}_\alpha^{n+k}=\{y\in G:d(y,z_\alpha^{n+k})\le 3C_1\delta^{n+k+1}\}$}. Let $y\in\widehat{Q}_\alpha^{n+k}$, we have $d(y,z_\alpha^{n+k})\le d(y,\widehat{z}_\alpha^{n+k+1})+d(z_\alpha^{n+k},\widehat{z}_\alpha^{n+k+1})\le 2C_1\delta^{n+k+1}$. This gives $\widehat{Q}_\alpha^{n+k}\subseteq\widetilde{Q}_\alpha^{n+k}$.

Fix $x\in G\setminus \widetilde{\Omega}$. There exists an unique `dyadic cube' $Q_\beta^n$ containing $x$. We claim that
\begin{equation}\label{emptyset}
  B(x,\delta^n)\cap\widehat{Q}_\alpha^{n+k}=\emptyset,~Q_\beta^n\cap\widehat{Q}_\alpha^{n+k}=\emptyset.
\end{equation}
Since
$$A^\prime_{\delta^n}\xi_\alpha^{n+k}(x)=\frac{1}{m(B(x,r))}\int_{B(x,r)\cap \widehat{Q}_\alpha^{n+k}}\xi_\alpha^{n+k}(y)dm(y)$$
and
$$\mathbb{E}_n\xi_\alpha^{n+k}(x)=\frac{1}{m(Q_\beta^{n})}\int_{Q_\beta^n\cap\widehat{Q}_\alpha^{n+k}}\xi_\alpha^{n+k}(y)dm(y),$$
it follows from the claim that $A^\prime_{\delta^n}\xi_\alpha^{n+k}=\mathbb{E}_n\xi_\alpha^{n+k}=0$.

We now prove the claim.  If $Q_\beta^n\cap\widehat{Q}_\alpha^{n+k}\neq\emptyset$, by Proposition~\ref{dyadic cube}(ii), it follows that $x\in Q_\beta^n\subseteq\widehat{Q}_\alpha^{n+k}\subseteq \widetilde{Q}_\alpha^{n+k}$, a contradiction
with $x\in G\setminus \widetilde{\Omega}$. On the other hand, if there exists a point $z\in B(x,\delta^n)\cap \widehat{Q}_\alpha^{n+k}$,  then $d(x,\widehat{z}_\alpha^{n+k+1})\le d(x,z)+d(z,\widehat{z}_\alpha^{n+k+1})\le\delta^n+C_1\delta^{n+k+1}< 2C_1\delta^{n+k+1}$, and so it follows that
$$d(x,z_\alpha^{n+k})\le d(x,\widehat{z}_\alpha^{n+k+1})+d(z_\alpha^{n+k},\widehat{z}_\alpha^{n+k+1})<3C_1\delta^{n+k+1},$$
contrary to $x\in G\setminus \widetilde{\Omega}$, and \eqref{emptyset} is proved. This gives the conclusion.\\

{\noindent\bf{Case}} $k<-k_5$.  We also have
\begin{equation*}
  \mathbb{E}_n\xi_{\alpha}^{n+k}(x)=0,~\forall~x\in G\setminus\widetilde{\Omega}.
\end{equation*}
Indeed, let $x\in Q_\beta^{n}$, if $ Q_\beta^{n}\cap \widehat{Q}_\alpha^{n+k}\neq \emptyset$, then by Proposition~\ref{dyadic cube}(ii), it follows that $\widehat{Q}_\alpha^{n+k}\subseteq Q_\beta^{n}$ and
\begin{equation*}
  \mathbb{E}_n\xi_\alpha^{n+k}(x)=\int_{\widehat{Q}_\alpha^{n+k}}\xi_\alpha^{n+k}(y)dm(y)=0.
\end{equation*}

{ Given a ball $B(x,r)$ and `dyadic cube' $Q_\alpha^k$, we define the set $\mathcal{I}(B(x,r),Q^{k}_\alpha)=\{Q_\alpha^{k}\cap B(x,r): Q_\alpha^{k}\cap\partial B(x,r) \neq\emptyset\}$. }Since  $\xi_\alpha^{n+k}$ is supported on $\widehat{Q}_\alpha^{n+k}$ and $\int_{\widehat{Q}_\alpha^{n+k}}\xi_{\alpha}^{n+k}=0$, it follows that {$\forall x\in G$},
\begin{align*}%B(x,\delta^n)\cap Q_\alpha^{n+k}
 |A_{\delta^n}\xi_{\alpha}^{n+k}(x)|&=\big|\frac{1}{m(B(x,\delta^n))}\int_{\mathcal{I}(B(x,\delta^n),\widehat{Q}^{n+k}_\alpha)}\xi_{\alpha}^{n+k}(y)dm(y)\big|\\
 &\le \frac{1}{m(B(x,\delta^n))}\int_{\delta^n-C_1\delta^{n+k+1}\le d(x,y)\le\delta^n+C_1\delta^{n+k+1}}|\xi_{\alpha}^{n+k}(y)|dm(y).
\end{align*}
Then by the above estimate we obtain
\begin{equation}\label{L1-norm}
\begin{split}
&\int_{G}|A_{\delta^n}\xi_{\alpha}^{n+k}(x)|dm(x)\\
&\le\int_{G}\frac{1}{m(B(x,\delta^n))}\int_{\delta^n-C_1\delta^{n+k+1}\le d(x,y)\le\delta^n+C_1\delta^{n+k+1}}|\xi_{\alpha}^{n+k}(y)|dm(y)dm(x)\\
  &\le \int_{G}|\xi_{\alpha}^{n+k}(y)|\cdot\int_{G}\frac{\mathds{1}_{B(y,\delta^n+C_1\delta^{n+k+1})\setminus B(y,\delta^n-C_1\delta^{n+k+1})}(x)}{m(B(x,\delta^n))}dm(x)dm(y)
\end{split}
\end{equation}
Set $A(y)={B(y,\delta^n+C_1\delta^{n+k+1})\setminus B(y,\delta^n-C_1\delta^{n+k+1})}$. {Fix $y\in G$. By~\eqref{int}, we have
\begin{equation*}
  \frac{\mathds{1}_{A(y)}(x)}{m(B(x,\delta^n))}\le 2^\epsilon(K+1)\frac{\mathds{1}_{A(y)}(x)}{m(B(x,\delta^{n}+C_1\delta^{n+k+1}))}\mathds{1}_{B(y,\delta^{n}+C_1\delta^{n+k+1})}(x)
\end{equation*}
By the same argument as in the proof of~\cite[Theorem 3.5]{AL19}, we have the following properties. For any $0<\varepsilon<1$, there exist $M$-points $\{u_i:1\le i\le M\}$  with $M\le D$ in $B(y,\delta^{n}+C_1\delta^{n+k+1})$ such that $B(y,\delta^n+C_1\delta^{n+k+1})\setminus \bigcup_{i=1}^M B(u_i,\delta^n+C_1\delta^{n+k+1})=\emptyset$. Fix $x\in B(y,\delta^{n}+C_1\delta^{n+k+1})$, and let $j$ be the first index such that
\begin{equation*}
 x\in  B(u_j,\delta^n+C_1\delta^{n+k+1}),~m(B(u_j,\delta^n+C_1\delta^{n+k+1}))\le (1+\varepsilon) m(B(x,\delta^n+C_1\delta^{n+k+1})).
\end{equation*}
%{\color{red}???WHY HERE X IS NOT Y}
%\begin{align*}
% &\frac{1}{m(B(x,\delta^{n}+C_1\delta^{n+k+1}))}\mathds{1}_{B(y,\delta^{n}+C_1\delta^{n+k+1})}(x)\\
% &\le\sum_{i=1}^M\frac{1+\varepsilon}{m(B(u_i,\delta^{n}+C_1\delta^{n+k+1}))}\mathds{1}_{B(y,\delta^{n}+C_1\delta^{n+k+1})\cap B(u_i,\delta^{n}+C_1\delta^{n+k+1})}(x).
%\end{align*}
By the above discussions, we first have
\begin{equation}\label{L1-norm-1}
\begin{split}
 &\int_{G}\frac{\mathds{1}_{A(y)}(x)}{m(B(x,\delta^n))}dm(x)\\
 &\le 2^\epsilon(K+1)\int_G\frac{\mathds{1}_{A(y)}(x)}{m(B(x,\delta^{n}+C_1\delta^{n+k+1}))}\mathds{1}_{B(y,\delta^{n}+C_1\delta^{n+k+1})}(x)dm(x)\\
 &\le 2^\epsilon(K+1)\int_G\frac{(1+\varepsilon)\mathds{1}_{A(y)}(x)}{m(B(u_j,\delta^{n}+C_1\delta^{n+k+1}))}\mathds{1}_{B(y,\delta^{n}+C_1\delta^{n+k+1})\cap B(u_j,\delta^{n}+C_1\delta^{n+k+1})}(x)dm(x)\\
 &\le 2^\epsilon(K+1)\int_G\sum_{i=1}^M\frac{(1+\varepsilon)\mathds{1}_{A(y)\cap B(u_i,\delta^{n}+C_1\delta^{n+k+1})}(x)}{m(B(u_i,\delta^{n}+C_1\delta^{n+k+1}))}dm(x)\\
 &\le 2^\epsilon(K+1)\int_G\sum_{i=1}^M\frac{(1+\varepsilon)\mathds{1}_{A(y)\cap B(u_i,\delta^{n}+C_1\delta^{n+k+1})}(x)}{m(B(u_i,\delta^{n}))}dm(x).
 %&\times\int_G\mathds{1}_{B(y,\delta^{n}+C_1\delta^{n+k+1})\cap B(u_i,\delta^{n}+C_1\delta^{n+k+1})}(x)\mathds{1}_{A(y)}(x)dm(x)\bigg)\\
 %&\le {(1+\varepsilon)D(K+1)(\delta+1)^\epsilon}\frac{m(A(y))}{m(B(y,\delta^n))}\frac{m(B(y,\delta^n))}{m(B(u_i,\delta^{n}+C_1\delta^{n+k+1}))}.
\end{split}
\end{equation}
Note that each $u_i\in B(y,\delta^n+C_1\delta^{n+k+1})$, then $B(y,\delta^n)\subseteq B(u_i,2\delta^{n}+C_1\delta^{n+k+1})$. It follows from~\eqref{int} that
\begin{equation*}
  \frac{m(B(y,\delta^n))}{m(B(u_i,\delta^n))}\le \frac{m(B(u_i,2\delta^{n}+C_1\delta^{n+k+1}))}{m(B(u_i,\delta^n))}\le 3^\epsilon(K+1).
\end{equation*}
Moreover, by Lemma~\ref{lem:decay}, we have ${m(A(y))}/{m(B(y,\delta^n))}\le K_\epsilon C_1^\epsilon\delta^{\epsilon k}$.  Combining these two estimates with~\eqref{L1-norm-1}, we conclude
\begin{equation}\label{L1-norm-2}
  \int_{G}\frac{\mathds{1}_{A(y)}(x)}{m(B(x,\delta^n))}dm(x)\le (1+\varepsilon)6^{\epsilon}(K+1)^2C_1^\epsilon DK_\epsilon\delta^{\epsilon k}.
\end{equation}}
 Combining \eqref{L1-norm-2} with~\eqref{L1-norm}, and then letting $\varepsilon\rightarrow 0$, we obtain $\|A_{\delta^n}\xi_{\alpha}^{n+k}\|_{L^1}\le 6^{\epsilon}(K+1)^2C_1^\epsilon DK_\epsilon\delta^{\epsilon k}\|\xi_{\alpha}^{n+k}\|_{L^1}$.
 Together with the above three cases for $k$, \eqref{bad} is proved.

\bigskip

The estimate of $m\big(\{x\in G\setminus\widetilde{\Omega}:S(b)(x)>{\gamma}/{3}\}\big)$ is similar, and {will only be indicated briefly}. Fix $n>n_2$. We first prove that
\begin{equation}\label{cancellation}
\mathbb{E}_n{b}(x)=0,~\forall~x\in G\setminus{\widetilde{\Omega}}.
\end{equation}
 Note that $b=\sum_{k}\sum_{\alpha\in\Lambda_k}b_\alpha^k$, by the linearity of operator $\mathbb{E}_n$, we only need to prove that for each $b_\alpha^k$, $\mathbb{E}_n{b_\alpha^k}=0$.

 Fix $x\in G\setminus\widetilde{\Omega}$. Let $Q_\beta^n\ni x$. Since $b_\alpha^k$ is supported on $Q_\alpha^k$, it follows that $\mathbb{E}_n{b^k_\alpha}(x)=\int_{Q_\alpha^k\cap Q_\beta^n}b_\alpha^k(y)dm(y)$. If $Q_\alpha^k\cap Q_\beta^n\neq \emptyset$, we split $k$ into two cases: $k\ge n$ and $k<n$. For $k\ge n$, by Proposition~\ref{dyadic cube}(ii), it follows that $x\in Q_\beta^n\subseteq Q_\alpha^k$, this leads to a contradiction since $Q_\alpha^k\subseteq \widetilde{\Omega}$. For $k<n$,  using Proposition~\ref{dyadic cube}(ii) again, we have $Q_\alpha^k \subseteq Q_\beta^n$, it follows immediately that
\begin{equation*}
 \mathbb{E}_n{b^k_\alpha}(x)=\frac{1}{m(Q_\beta^n)}\int_{Q_\alpha^k}b_\alpha^k(y)dm(y)=0.
\end{equation*}
This gives~\eqref{cancellation}. Similar to~\eqref{bad}, we can also establish the following inequality
\begin{equation}\label{bad-1}%\label{bad}
  \int_{G\setminus {\widetilde{\Omega}}}|A_{\delta^n}b_\alpha^{n+k}(x)|dm(x)\le
\left\{
 \begin{array}{ll}
  D\|b_\alpha^{n+k}\|_{L^1}, & \hbox{$-k_5\le k\le k_5$;} \\
  0, & \hbox{$k>k_5$;} \\
   6^{\epsilon}(K+1)^2C_1^\epsilon DK_\epsilon\delta^{\epsilon k}\|b_\alpha^{n+k}\|_{L^1}, & \hbox{$k<-k_5$.}
\end{array}
     \right.
\end{equation}
Note that for $-k_5\le k\le k_5$, the estimate holds  by Proposition~\ref{aver}. For $k>k_5$, the estimate holds true  by \eqref{emptyset} and the fact that $b_\alpha^{n+k}$ is supported on $Q_\alpha^{n+k}$. For $k<-k_5$. We first have
\begin{align*}%B(x,\delta^n)\cap Q_\alpha^{n+k}
 |A_{\delta^n}b_{\alpha}^{n+k}(x)|&=\big|\frac{1}{m(B(x,\delta^n))}\int_{\mathcal{I}(B(x,\delta^n),{Q}^{n+k}_\alpha)}b_{\alpha}^{n+k}(y)dm(y)\big|\\
 &\le \frac{1}{m(B(x,\delta^n))}\int_{\delta^n-C_1\delta^{n+k}\le d(x,y)\le\delta^n+C_1\delta^{n+k}}|b_{\alpha}^{n+k}(y)|dm(y).
\end{align*}
{By the same estimates as~\eqref{L1-norm} and~\eqref{L1-norm-2}, the above inequality yields
\begin{align*}
  \int_{G\setminus{\Omega}}|A_{\delta^n}b_{\alpha}^{n+k}(x)|dm(x)&\le \frac{1}{m(B(x,{\delta^n}))}\int_{G}\int_{\delta^n-C_1\delta^{n+k}\le d(x,y)\le\delta^n+C_1\delta^{n+k}}|b_{\alpha}^{n+k}(y)|dm(y)dm(x)\\
  &\le 6^{\epsilon}(K+1)^2C_1^\epsilon DK_\epsilon\delta^{\epsilon k}\|b_{\alpha}^{n+k}\|_{L^1(G,m)};
\end{align*}}
together with the above two cases, this proves~\eqref{bad-1}, and the proof is complete.
%By an argument similar to that in the proof for $m\{x\in G\setminus\widetilde{\Omega}:S(\xi)(x)>{\gamma}/{3}\}$, the above inequality implies $ m\{x\in G\setminus\widetilde{\Omega}:S(b)(x)>{\gamma}/{3}\}\le\frac{C_{\epsilon,\delta}}{\gamma}\|f\|_{L^1(G,m)}$, which completes the proof.
\end{proof}% \eqref{weak-type inequalities of short variation}

The proof of weak type $(1,1)$ estimate for operator $SV$ is similar to $S$. Let us explain it briefly.
\begin{proof}[Proof of \eqref{weak-type inequalities of short variation}]
Fix $f\in L^1(G,m)$. Decompose $f=g+b+\xi$.  The desired estimate for $g$ is true by the fact that $SV$ is of strong type $(2,2)$. In what follows, we only state the proof for $\xi$ since the proof for $b$ is similar.

By similar arguments as in the previous proof, we mainly need to prove the following inequality for $n>n_2$,% which is similar to~\eqref{bad},
%\begin{equation*}%\label{inequality-b}
%\begin{split}
  %&\int_{G\setminus\tilde{\Omega}}V_2(A_{r}b^{n+k}_{\alpha}(x):r\in[\delta^{n},\delta^{n+1}))dm(x)\le (2+(K+1)\delta^\epsilon)K_\epsilon(2C_1)^\epsilon\delta^{\epsilon k}\|b_{n+k}^\alpha\|_{L^1(G,m)};\\
 % &\int_{G\setminus\widetilde{\Omega}}V_2(A^\prime_{r}\xi^{n+k}_{\alpha}:r\in[\delta^{n},\delta^{n+1}))dm(x)\le a(k)\|\xi_\alpha^{n+k}\|_{L^1(G,m)},
 % \end{split}
%\end{equation*}
%where
\begin{equation*}%\label{inequality-b}
\begin{split}
  %&\int_{G\setminus\tilde{\Omega}}V_2(A_{r}b^{n+k}_{\alpha}(x):r\in[\delta^{n},\delta^{n+1}))dm(x)\le (2+(K+1)\delta^\epsilon)K_\epsilon(2C_1)^\epsilon\delta^{\epsilon k}\|b_{n+k}^\alpha\|_{L^1(G,m)};\\
  &\int_{G\setminus\widetilde{\Omega}}V_2(A^\prime_{r}\xi^{n+k}_{\alpha}(x):r\in[\delta^{n},\delta^{n+1}))dm(x)\le a(k)\|\xi_\alpha^{n+k}\|_{L^1(G,m)},
  \end{split}
\end{equation*}
where
\begin{equation*}
  a(k)=\left\{
    \begin{array}{ll}
    C_4, & \hbox{$-k_5\le k\le k_5$;} \\
      0, & \hbox{$k>k_5$;} \\
      2\cdot 3^\epsilon (\delta+1)^\epsilon (K+1)^2C_1^\epsilon D K_\epsilon\delta^{\epsilon k}, & \hbox{$k<-k_5$.}
    \end{array}
  \right.
\end{equation*}

We now focus on the above inequality. Fix $n>n_2$. Note that for $-k_5\le k\le k_5$, by Lemma~\ref{controlled by maximal operator}, the estimate is true. For $k>k_5$, similar to~\eqref{emptyset}, we can prove that for every $x\in G\setminus\widetilde{\Omega}$ and $r\in[\delta^n,\delta^{n+1}]$, $B(x,r)\cap \widehat{Q}_\alpha^{n+k}=\emptyset$, hence $V_2(A^\prime_{r}\xi^{n+k}_{\alpha}:r\in[\delta^{n},\delta^{n+1}))=0$ . It remains to prove the case $k<-k_5$.

%Recall the definition of $\mathcal{I}(B(x,R)\setminus B(x,r),Q)$ for an annular and a cube.
 Given an annular $B(x,R)\setminus B(x,r)$ and `dyadic cube' $Q_\alpha^k$, we define the set $\mathcal{I}(B(x,R)\setminus B(x,r),Q^{k}_\alpha)=\{Q_\alpha^{k}\cap(B(x,R)\setminus B(x,r)): Q_\alpha^{k}\cap\partial(B(x,R)\setminus B(x,r)) \neq\emptyset\}$. Since $\int_{\widehat{Q}_\alpha^{n+k}}\xi_{\alpha}^{n+k}=0$, it follows that
\begin{align*}
  &V_2(A^\prime_{r}\xi^{n+k}_{\alpha}(x):r\in[\delta^{n},\delta^{n+1}))\\
  &\le\sup_{\delta^{n}\le r_0<\cdots<r_J<\delta^{n+1}}\sum_{i=1}^J\frac{1}{m(B(x,r_i))}\big|\int_{\mathcal{I}(B(x,r_i)\setminus B(x,r_{i-1}),\widehat{Q}^{n+k}_\alpha)}\xi_\alpha^{n+k}(y)dm(y)\big|\\
&+\sup_{\delta^{n}\le r_0<\cdots<r_J<\delta^{n+1}}\sum_{i=1}^J
\big|\bigg(\frac{1}{m(B(x,r_{i-1}))}-\frac{1}{m(B(x,r_i))}\bigg)\int_{\mathcal{I}(B(x,r_i), \widehat{Q}_\alpha^{n+k})}\xi_\alpha^{n+k}(y)dm(y)\big|\\
&\le \frac{2}{m(B(x,\delta^{n}))}\int_{\delta^{n+1}-C_1\delta^{n+k+1}\le d(x,y)\le\delta^{n+1}+C_1\delta^{n+k+1}}|\xi_{\alpha}^{n+k}(y)|dm(y).\\
%&+\frac{2}{m(B(x,\delta^{n}))}\int_{\delta^{n+1}-C_1\delta^{n+k+1}\le d(x,y)\le\delta^{n+1}+C_1\delta^{n+k+1}}|\xi_{\alpha}^{n+k}(y)|dm(y)).
\end{align*}
{Using the same argument as~\eqref{L1-norm} and~\eqref{L1-norm-2}, the above inequality yields
\begin{align*}
\int_{G\setminus\widetilde{\Omega}}V_2(A^\prime_{r}b^{n+k}_{\alpha}(x):r\in[\delta^{n},\delta^{n+1}))dm(x)\le  2\cdot 3^\epsilon(\delta+1)^\epsilon (K+1)^2C_1^\epsilon D K_\epsilon\delta^{\epsilon k}\|b_{n+k}^\alpha\|_{L^1(G,m)},
\end{align*}}
which completes the proof.
\end{proof}
%%%%%%%%%%%%%%%%%%%%%%%%%%%%%%%%%%%%%%%%%%%%%%%%%%%%%%%%%%%%%%%%%%%%%%%%%%%%%%%%%%%%%%%%%%%%%%%%%%%%%%%%%%%%%%%%%%%%%%%%%%%%%%%%%%%%%%%%%%%%
%%%%%%%%%%%%%%%%%%%%%%%%%%%%%%%%%%%%%%%%%%%%%%%%%%%%%%%%%%%%%%%%%%%%%%%%%%%%%%%%%%%%%%%%%%%%%%%%%%%%%%%%%%%%%%%%%%%%%%%%%%%%%%%%%%%%%%%%%%%%
\section{($L^\infty$, BMO) estimates}\label{ST5}
In this section, we prove that both operator $S$ and operator $SV$ map $L_c^\infty$ to dyadic BMO space.

Given a locally integrable function $f\in L^1_{loc}(G,m)$,
the dyadic sharp maximal function is defined by
\begin{equation*}
  M^\sharp f(x)=\sup_{(\alpha,k),x\in Q^k_{\alpha}}\inf_c\frac{1}{m(Q^k_\alpha)}\int_{Q^k_\alpha}|f(y)-c|dm(y),
\end{equation*}
and then the dyadic BMO space is defined as $BMO=\{f\in L^1_{loc}(G,m):\| M^\sharp f\|_{L^{\8}}<\8\}$ with the norm $\|f\|_{BMO}=\| M^\sharp f\|_{L^{\8}}$.

We first handle operator $S$, which was defined as
$$S(f)=\bigg(\sum_{n>n_{r_0}}|A^\prime_{\delta^n}f-\mathbb{E}_nf|^2\bigg)^{1/2}.$$ As in dealing with the weak type $(1,1)$ estimate, we also need the non-doubling analysis.

\begin{proof}[Proof of~\eqref{BMO-type inequalities of square function}]%%and $n_0$ is a constant defined in Lemma~\ref{measure estimate of annulus}., where $C_1$ and $a_0$ are the constants defined in Proposition~\ref{dyadic cube}
Fix a `dyadic cube' $Q^k_\beta$. Recall that $\widetilde{Q}_\beta^k=\{y\in G:d(y,z_\alpha^k)\le 3C_1\delta^{k+1}\}$ and $k_3=\min\{k:a_0\delta^k>r_0\}$. Set
\begin{equation}\label{dilation}
  {Q^k_\beta}^*=\left\{
              \begin{array}{ll}
                \widetilde{Q^k_\beta}, & \hbox{$k>k_3$;}\\
                 Q^k_\beta, & \hbox{$k\le k_3$,}
              \end{array}
            \right.
\end{equation}
and
\begin{equation*}%\label{decomposition of f}
  f(x)=f(x)\mathds{1}_{{Q^k_\beta}^*}(x)+f(x)\mathds{1}_{G\setminus {Q^k_\beta}^*}(x):=f_1(x)+f_2(x).
\end{equation*}
 By the triangle inequality, it follows that
\begin{equation*}
 \begin{split}
  \frac{1}{m(Q_\beta^k)}\int_{Q_\beta^k}|S(f)(y)-c|&dm(y)\le \frac{1}{m(Q_\beta^k)}\int_{Q_\beta^k}|S(f)(y)-S(f_2)(y)|dm(y)\\
  &+\frac{1}{m(Q_\beta^k)}\int_{Q_\beta^k}|S(f_2)(y)-c|dm(y)\\
  &\le  \frac{1}{m(Q_\beta^k)}\int_{Q_\beta^k}|S(f_1)(y)|dm(y)+\frac{1}{m(Q_\beta^k)}\int_{Q_\beta^k}|S(f_2)(y)-c|dm(y).
 \end{split}
\end{equation*}
We now focus on the first term of the right hand side. By the Cauchy-Schwarz inequality, \eqref{dilation}, \eqref{int} and the fact that the operator $S$ is of strong type $(2,2)$,  we have
\begin{equation*}
  \begin{split}
     \frac{1}{m(Q_\beta^k)}\int_{Q_\beta^k}|S(f_1)(y)|dm(y)&\le \bigg(\frac{1}{m(Q_\beta^k)}\int_{Q_\beta^k}|S(f_1)(y)|^2dm(y)\bigg)^{1/2}\\
     &\le \bigg(\frac{c_2^2}{m(Q_\beta^k)}\int_{{Q^k_\beta}^*}|f(y)|^2dm(y)\bigg)^{1/2}\\
&\le c_2(3C_1\delta/a_0)^{\epsilon/2}\|f\|_{L^\8}.%(K+1)^{1/2}c_2(3C_1/a_0)^{\epsilon/2}\|f\|_{L^\8}.
   \end{split}
\end{equation*}
It remains to handle the term $\frac{1}{m(Q_\beta^k)}\int_{Q_\beta^k}|S(f_2)(y)-c|dm(y)$.  Taking $c=S(f_2)(z^k_\beta)$, we first claim that
 $$\mathbb{E}_{n}f_2(x)-\mathbb{E}_{n}f_2(z^k_\beta)=0,~\forall~x\in Q^k_\beta.$$
Indeed, let $Q_\alpha^n\ni x$, if $n\le k$, by Proposition~\ref{dyadic cube}(ii), then $Q^n_{\alpha}\subseteq Q^k_{\beta}$. Since $f_2$ is supported on $G\setminus {Q_\alpha^k}^*$, it follows that  $\mathbb{E}_{n}f_2(x)=\mathbb{E}_{n}f_2(z^k_\beta)=0$. If $n>k$, using Proposition~\ref{dyadic cube}(ii) again, we have $Q_\beta^k\subseteq Q_\alpha^n$ and $x,z_\beta^k\in Q_\alpha^n$. It follows  that $\mathbb{E}_{n}f_2(x)=\mathbb{E}_{n}f_2(z^k_\beta)$, then the claim is proved.

On the other hand, let $n_3=k_3+[\log_{\delta}(2C_1/a_0)]+2$. By the triangle inequality, we first  have
\begin{equation}\label{BMO-ineq}
\begin{split}%\label{BMO-estimate of square function+}
  &\bigg(\sum_{n>n_{r_0}}|A^\prime_{\delta^n}f_2(x)-A^\prime_{\delta^n}f_2(z^k_\beta)|^2\bigg)^{1/2}\le \bigg(\sum_{n_{r_0}<n\le n_3}|A^\prime_{\delta^n}f_2(x)-A^\prime_{\delta^n}f_2(z^k_\beta)|^2\bigg)^{1/2}\\
&+\bigg(\sum_{n>n_3}|A^\prime_{\delta^n}f_2(x)-A^\prime_{\delta^n}f_2(z^k_\beta)|^2\bigg)^{1/2}\\
&\le 2(n_3-n_{r_0})\|f\|_{L^\8}+\bigg(\sum_{n>n_3}|A^\prime_{\delta^n}f_2(x)-A^\prime_{\delta^n}f_2(z^k_\beta)|^2\bigg)^{1/2}.
\end{split}
\end{equation}
 Let $d(Q_\beta^k)$ be the diameter of $Q_\beta^k$. Set $n_4=\min\{n:\delta^n>d(Q_\beta^k)\}$. We now consider two cases for $n_3$: $n_3\ge n_4$ and $n_3<n_4$.

 \noindent{\bf{Case} $n_3\ge n_4$}. Note that $\delta^{n_3}>d(Q_\beta^k)$. By~\eqref{BMO-ineq}, the proof will be completed if we proved the following inequality
\begin{equation*}
  \bigg(\sum_{n>n_3}|A^\prime_{\delta^n}f_2(x)-A^\prime_{\delta^n}f_2(z^k_\beta)|^2\bigg)^{1/2}\le \frac{2\big(K+2^\epsilon K(K+1)\big)}{1-\delta^{-\epsilon}}\|f\|_{L^\8},~\forall~x\in Q_\beta^k.
\end{equation*}

We now focus on the above inequality. By the triangle inequality, we have %Using the facts that the metric $d$ is invariant and the measure $m$ is translation invariant, we obtain
{
\begin{equation*}
  \begin{split}
   &|A^\prime_{\delta^n}f_2(x)-A^\prime_{\delta^n}f_2(z^k_\beta)|\\
&=|\frac{1}{m(B(x,\delta^n))}\int_{B(x,\delta^n)}f_2(y)dm(y)-
  \frac{1}{m(B(z_\beta^k,\delta^n))}\int_{B(z_\beta^k,\delta^n)}f_2(y)dm(y)|\\
  &\le |\frac{1}{m(B(x,\delta^n))}\int_{B(x,\delta^n)}f_2(y)dm(y)-
  \frac{1}{m(B(x,\delta^n))}\int_{B(z_\beta^k,\delta^n)}f_2(y)dm(y)|\\
  &+|\frac{1}{m(B(x,\delta^n))}\int_{B(z_\beta^k,\delta^n)}f_2(y)dm(y)-
  \frac{1}{m(B(z_\beta^k,\delta^n))}\int_{B(z_\beta^k,\delta^n)}f_2(y)dm(y)|\\
%&\le\frac{1}{m(B(x,\delta^n))}\int_{B(x,\delta^n)\triangle B(z_\beta^k,\delta^n)}|f_2(y)|dm(y)+\big|\frac{1}{m(B(x,\delta^n))}-
 % \frac{1}{m(B(z_\beta^k,\delta^n))}\big|\int_{B(z_\beta^k,\delta^n)}|f_2(y)|dm(y)\\
&\le \frac{m(B(x,\delta^n)\triangle B(z_\beta^k,\delta^n))}{m(B(x,\delta^n))}\|f\|_{L^\8}+\big|\frac{m(B(z_\beta^k,\delta^n))}{m(B(x,\delta^n))}-1\big|\|f\|_{L^\8}\\
&\le 2\frac{m(B(x,\delta^n)\triangle B(z_\beta^k,\delta^n))}{m(B(x,\delta^n))}\|f\|_{L^\8}.
  %&\le|\big(\frac{1}{\mu(B(x,\delta^k))}-\frac{1}{\mu(B(z_\beta^n,\delta^k))}\big)\int_{B(x,\delta^k)}f_2(y)d\mu(y)|\\
  \end{split}
\end{equation*}}%\frac{\mu(B(x,\delta^k))}{\mu(B(x,\delta^k))}By the definition of averaging operator, we only need to consider the case $k\ge 0$, hence the above inequality is controlled byOn one hand if $d(x,d(x,z_\beta^n))\le 1$,
Note that{
$$B(x,\delta^n)\triangle B(z_\beta^k,\delta^n)\subseteq \bigg(B(x,d(x,z_\beta^k)+\delta^n)\setminus B(x,\delta^n)\bigg)\cup \bigg(B(z_\beta^k,d(x,z_\beta^k)+\delta^n)\setminus B(z_\beta^k,\delta^n)\bigg).$$}
For $n>n_3$, we have $\delta^n>r_0$. { Note that $x\in Q_\beta^k$, then by Proposition~\ref{dyadic cube}(iv), we have $B(z_\beta^k, \delta^n)\subseteq B(x,\delta^n+d(Q_\beta^k))$. It follows from~\eqref{int} that
\begin{equation*}
  \frac{m(B(z_\beta^k,\delta^n))}{m(B(x,\delta^n))}\le \frac{m(B(x,\delta^n+d(Q_\beta^k)))}{m(B(x,\delta^n))}\le 2^\epsilon(K+1).
\end{equation*}
Combining the above inequality with condition~\eqref{decay property}, we have
\begin{equation*}
 \begin{split}
  &\frac{m(B(x,\delta^n)\triangle B(z_\beta^k,\delta^n))}{m(B(x,\delta^n))}\\
  &\le\frac{m\bigg(B(x,d(x,z_\beta^k)+\delta^n)\setminus B(x,\delta^n)\bigg)}{m(B(x,\delta^n))}+\frac{m\bigg(B(z_\beta^k,d(x,z_\beta^k)+\delta^n)\setminus B(z_\beta^k,\delta^n)\bigg)}{m(B(z_\beta^k,\delta^n))}\frac{m(B(z_\beta^k,\delta^n))}{m(B(x,\delta^n))}\\
  &\le \big(K+2^\epsilon K(K+1)\big)\bigg(\frac{d(x,z_\beta^k)}{\delta^n}\bigg)^{\epsilon}.
   %2C_{m,\delta}\bigg(\frac{d(x,z_\beta^n)}{\delta^k}\bigg)^{\epsilon}\le2C_{m,\delta}\bigg(\frac{d(Q_\beta^n)}{\delta^k}\bigg)^{\epsilon}.
%2Cd(x,z_\beta^n)^{\theta_0}\frac{1}{\delta^{\theta_0k}-1}.
 \end{split}
\end{equation*}}
Therefore,
\begin{equation}\label{estimate of averaging operator}
  |A^\prime_{\delta^n}f_2(x)-A^\prime_{\delta^n}f_2(z^k_\beta)|\le 2\big(K+2^\epsilon K(K+1)\big)\bigg(\frac{d(Q_\beta^k)}{\delta^n}\bigg)^{\epsilon}\|f\|_{L^\8}.
\end{equation}%Moreover, combining the fact that the $\ell^1$ norm is not less than the $\ell^2$ norm with \eqref{estimate of averaging operator}
By the above inequality, we have
\begin{equation*}%\label{BMO estimate of square function}
\begin{split}
    \bigg(\sum_{n>n_3}|A^\prime_{\delta^n}f_2(x)-A^\prime_{\delta^n}f_2(z_\beta^k)|^2\bigg)^{1/2}
    &\le 2\big(K+2^\epsilon(K+1)\big)\|f\|_{L^\8}\sum_{n:\delta^n>d(Q_\beta^k)}\bigg(\frac{d(Q_\beta^k)}{\delta^n}\bigg)^{\epsilon}\\
    &\le \frac{2\big(K+2^\epsilon K(K+1)\big)}{1-\delta^{-\epsilon}}
    \|f\|_{L^\8}.
\end{split}
\end{equation*}%le\sum_{n>n_4}|A^\prime_{\delta^n}f_2(x)-A^\prime_{\delta^n}f_2(z^k_\beta)|\\

\noindent{\bf{Case} $n_3<n_4$}. Note that $d(Q_\beta^k)\ge \delta^{n_3-1}$. It follows {from Proposition \ref{dyadic cube}}(iv) that
\begin{equation*}
  a_0\delta^{k}\ge ({a_0}/{{2}C_1})\delta^{(n_3-1)}>\delta^{k_3}\ge r_0,
\end{equation*}
and $f_2=f\mathds{1}_{G\setminus \widetilde{Q^k_\beta}}$. We claim that for $n_3<n\le n_4$,
\begin{equation}\label{BMO-ineq1}%\label{BMO-estimate of square function+}
  A^\prime_{\delta^n}f_2(x)=0, ~\forall~x\in Q_\beta^k.
\end{equation}
Indeed, fix $x\in Q_\beta^k$ and let $y\in B(x,\delta^n)$, then $d(y,z_\beta^k)\le d(y,x)+d(z_\beta^k,x)\le \delta^n+C_1\delta^k\le d(Q_\beta^k)+C_1\delta^k\le 3C_1\delta^k$. From this, we have
$B(x,\delta^n)\subset\widetilde{Q}_{\beta}^k$. Then the claim {follows from the fact that $f_2$ is supported on $G\setminus \widetilde{Q^k_\beta}$.}

Combining~\eqref{BMO-ineq1} with a similar argument as in the previous proof, we also have

\begin{equation*}
  \bigg(\sum_{n>n_3}|A^\prime_{\delta^n}f_2(x)-A^\prime_{\delta^n}f_2(z^k_\beta)|^2\bigg)^{1/2}\le {\frac{2\big(K+2^\epsilon K(K+1)\big)}{1-\delta^{-\epsilon}}}\|f\|_{L^\8},~\forall~x\in Q_\beta^k,
\end{equation*}
and  the proof is complete.
\end{proof}
%there exists a ball $B(z_\beta^n,a_0\delta^n)\subseteq Q_\beta^n$. With

The proof of $(L^\8_c, BMO)$ boundedness of $SV$ is similar to  $S$. Let us give a sketch of the proof.
\begin{proof}[Proof of \eqref{BMO-type inequalities of short variation}]
Let $f$ be a $L^\8_c$ function. Recall that the short variation operator is defined by $SV(f)=\big(\sum_{n\ge n_{r_0}}V_2(A^\prime_{r}f:r\in[\delta^{n},\delta^{n+1}))^2\big)^{1/2}$. Fix a `dyadic cube' $Q^k_\beta$ and let $=f\mathds{1}_{{Q^k_\beta}^*}+f\mathds{1}_{G\setminus {Q^k_\beta}^*}:=f_1+f_2$. We only state the proof for case $n_3\ge n_4$, namely $\delta^{n_3}>d(Q_\beta^k)$, since the proof for case $n_3<n_4$ is similar. By an argument similar to that in the  proof for operator $S$, it is sufficient to prove that there exists a constant $C_{\epsilon,\delta, K}>0$ such that  for all $x\in Q_\beta^k$,
\begin{equation*}%Hence applying a similar argument of the proof of~\eqref{BMO-type inequalities of square function}
  \bigg(\sum_{n>n_3}V_2(A^\prime_rf_2(x)-A^\prime_rf_2(z_\beta^k):r\in[\delta^{n},\delta^{n+1}))^2\bigg)^{1/2}\le C_{\epsilon,\delta, K}\|f\|_{L^\8}.
   %4K\bigg(\big(\frac{\delta-1}{1-\delta^{-\epsilon}}\big)^{1/2}+2\big(\frac{(K+1)\delta^\epsilon}{1-\delta^{-\epsilon^2}}\big)^{1/2}\bigg)\|f\|_{L^\8}.
\end{equation*}
Fix $n>n_3$ and  $x\in Q_\beta^k$. By the definition of short variation operator, there exists a sequence $\{r_i\}\subseteq [\delta^n,\delta^{n+1})$ such that
\begin{equation}\label{controll-ineq}
\begin{split}
   &V_2(A^\prime_rf_2(x)-A^\prime_rf_2(z_\beta^k):r\in[\delta^{n},\delta^{n+1}))\\
   &\le 2\bigg(\sum_{i}|A^\prime_{r_i}f_2(x)-A^\prime_{r_i}f_2(z_\beta^k)-(A^\prime_{r_{i-1}}f_2(x)-A^\prime_{r_{i-1}}f_2(z_\beta^k))|^2\bigg)^{1/2}.
\end{split}
\end{equation}
 We split such sequence $\{r_i\}$ into two cases: ${J_1}=\{i:r_i-r_{i-1}\le d(Q_\beta^k)^\epsilon/\delta^{(\epsilon-1)n}\}$ and $J_2=\{i:r_i-r_{i-1}>d(Q_\beta^k)^\epsilon/\delta^{(\epsilon-1)n}\}$. Therefore, the right hand side of the above inequality is controlled by the sum of
 %we first split the short variation operator $V_2(A^\prime_rf_2(x)-A^\prime_rf_2(z_\beta^k):r\in[\delta^{n},\delta^{n+1}))$ into two parts $SV_{I_n}$ and $SV_{II_n}$ by comparing $r_i-r_{i-1}$ with $d(Q_\beta^k)^\epsilon/\delta^{(\epsilon-1)n}$,
\begin{equation*}
\begin{split}
  SV&_{I_n}(f_2)(x)=\bigg(\sum_{i:i\in J_1}
 |A^\prime_{r_i}f_2(x)-A^\prime_{r_i}f_2(z_\beta^k)-(A^\prime_{r_{i-1}}f_2(x)-A^\prime_{r_{i-1}}f_2(z_\beta^k))|^2\bigg)^{1/2},
 \end{split}
\end{equation*}
and
\begin{equation*}
\begin{split}
  SV&_{II_n}(f_2)(x)=\bigg(\sum_{i:i\in J_2}|A^\prime_{r_i}f_2(x)-A^\prime_{r_{i}}f_2(z_\beta^k)-
  (A^\prime_{r_{i-1}}f_2(x)-A^\prime_{r_{i-1}}
  f_2(z_\beta^k))|^2\bigg)^{1/2}.
  \end{split}
\end{equation*}
We now focus on $SV_{I_n}(f_2)(x)$. By the triangle inequality, we have
\begin{equation*}
 \begin{split}
  SV_{I_n}(f_2)(x)&\le \bigg(\sum_{i:i\in J_1}|A^\prime_{r_i}f_2(x)-A^\prime_{r_{i-1}}f_2(x)|^2\bigg)^{1/2}\\
  &+\bigg(\sum_{i:i\in J_1}|A^\prime_{r_i}f_2(z_\beta^k)-A^\prime_{r_{i-1}}f_2(z_\beta^k)|^2\bigg)^{1/2}.
 \end{split}
\end{equation*}
 Note that for any $z\in Q_\beta^k$, we have
\begin{align*}
  |&A^\prime_{r_i}f_2(z)-A^\prime_{r_{i-1}}f_2(z)|\le\frac{1}{m(B(z,r_i))}\int_{B(z,r_i)\setminus B(z,r_{i-1})}|f_2(y)|dm(y)\\
  &+\big(\frac{1}{m(B(z,r_{i-1}))}-\frac{1}{m(B(z,r_{i}))}\big)\int_{B(z,r_{i-1})}|f_2(y)|dm(y)\\
  &\le \frac{m(B(z,r_i))-m(B(z,r_{i-1}))}{m(B(z,r_i))}\|f_2\|_{L^\8}+\bigg(\frac{m(B(z,r_{i-1}))}{m(B(z,r_{i-1}))}-\frac{m(B(z,r_{i-1}))}
  {m(B(z,r_{i}))}\bigg)\|f_2\|_{L^\8}\\
  &\le 2\bigg(\frac{m(B(z,r_{i}))-m(B(z,r_{i-1}))}{m(B(z,r_{i}))}\bigg)\|f_2\|_{L^\8}\\
  &\le 2\|f\|_{L^\8}\int_{m(B(z,r_{i-1}))}^{m(B(z,r_{i}))}\frac{1}{u}du,
\end{align*}%and inequality~\eqref{int},
by condition~\eqref{decay property}, and the integral term of the above inequality is controlled by
\begin{align*}
\int_{m(B(z,r_{i-1}))}^{m(B(z,r_{i}))}&\frac{1}{u}du\le \big(m(B(z,r_{i}))-m(B(z,r_{i-1}))\big)^{1/2}
\bigg(\int_{m(B(z,r_{i-1}))}^{m(B(z,r_{i}))}\frac{1}{u^2}du\bigg)^{1/2}\\
&\le \bigg(\frac{m(B(z,r_{i}))-m(B(z,r_{i-1}))}{m(B(z,r_{i-1}))}\bigg)^{1/2}\bigg(
m(B(z,r_{i-1}))\int_{m(B(z,r_{i-1}))}^{m(B(z,r_{i}))}\frac{1}{u^2}du\bigg)^{1/2}\\
&\le K\bigg(\frac{r_i-r_{i-1}}{r_{i-1}}\bigg)^{\epsilon/2}{\bigg(
\int_{m(B(z,{r_{i-1}}))}^{m(B(z,r_{i}))}\frac{m(B(z,\delta^{n+1}))}{u^2}du\bigg)^{1/2}}.
%1-\frac{m(B(z,r_{i-1}))}{m(B(z,r_{i}))}\bigg)^{1/2}\\
%&\le C_{m,\delta}\bigg(\frac{d(Q_\beta^n)}{r_{i-1}}\bigg)^{\eta/2}\bigg(1-C_m\bigg(\frac{r_{i-1}}{r_i}\bigg)^{c_m}\bigg)^{1/2}.
\end{align*}
{Then using~\eqref{int}, the above inequality yields
\begin{align*}
  &\bigg(\sum_{i:i\in J_1}|A^\prime_{r_i}f_2(x)-A^\prime_{r_{i-1}}f_2(x)|^2\bigg)^{1/2}\\
  &\le 2K\|f\|_{L^\8}\bigg(\sum_{i:i\in J_1}\bigg(\frac{r_i-r_{i-1}}{r_{i-1}}\bigg)^{\epsilon}
\int_{m(B(x,{r_{i-1}}))}^{m(B(x,r_{i}))}\frac{m(B(x,\delta^{n+1}))}{u^2}du\bigg)^{1/2}\\
&\le 2K\|f\|_{L^\8}\bigg(\frac{d(Q_\beta^k)}{\delta^n}\bigg)^{\epsilon^2/2}\bigg(
\int_{m(B(x,{\delta^{n}}))}^{m(B(x,\delta^{n+1}))}\frac{m(B(x,\delta^{n+1}))}{u^2}du\bigg)^{1/2}\\
&\le 2K\big((K+1)\delta^\epsilon\big)^{1/2}\|f\|_{L^\8}\bigg(\frac{d(Q_\beta^k)}{\delta^n}\bigg)^{\epsilon^2/2}.
\end{align*}
By a similar argument, we have
\begin{equation*}
  \bigg(\sum_{i:i\in J_1}|A^\prime_{r_i}f_2(z_\beta^k)-A^\prime_{r_{i-1}}f_2(z_\beta^k)|^2\bigg)^{1/2}\le 2K\big((K+1)\delta^\epsilon\big)^{1/2}\|f\|_{L^\8}\bigg(\frac{d(Q_\beta^k)}{\delta^n}\bigg)^{\epsilon^2/2}.
\end{equation*}
Therefore,
\begin{align*}
 SV_{I_n}(f_2)(x)&\le 4K\big((K+1)\delta^\epsilon\big)^{1/2}\|f\|_{L^\8}\bigg(\frac{d(Q_\beta^k)}{\delta^n}\bigg)^{\epsilon^2/2}.
\end{align*}}

For the part $SV_{II_n}(f_2)(x)$, using the triangle inequality, we first have
\begin{equation*}
 \begin{split}
  SV_{II_n}(f_2)(x)&\le\bigg(\sum_{i:i\in J_2}|A^\prime_{r_i}f_2(x)-A^\prime_{r_{i}}f_2(z_\beta^k)|^2\bigg)^{1/2}\\
  &+ \bigg(\sum_{i:i\in J_2}|A^\prime_{r_{i-1}}f_2(x)-A^\prime_{r_{i-1}}f_2(z_\beta^k)|^2\bigg)^{1/2}.
 \end{split}
\end{equation*}
Note that $\#\{J_2\}\le{(\delta-1)\delta^{\epsilon n}}/{d(Q_\beta^k)^\epsilon}$. Similar to~\eqref{estimate of averaging operator}, for any $z\in Q_\beta^k$ and $r\in [\delta^n,\delta^{n+1})$, we have
\begin{equation*}
  |A^\prime_{r}f_2(z)-A^\prime_{r}f_2(z^k_\beta)|\le{2\big(K+2^\epsilon K(K+1)\big)\bigg(\frac{d(Q_\beta^k)}{\delta^n}\bigg)^{\epsilon}}\|f\|_{L^\8}.
\end{equation*}
Combining the above inequalities, we have%\sum_{n\in\mathbb{N}:\delta^n>d(Q_\beta^k)}
\begin{align*}
 SV_{II_n}(f_2)(x) &\le 4\big(K+2^\epsilon K(K+1)\big)\|f\|_{L^\8}\bigg(
(\delta-1)\bigg(\frac{\delta^n}{d(Q_\beta^k)}\bigg)^\epsilon\bigg(\frac{d(Q_\beta^k)}{\delta^n}\bigg)^{2\epsilon}\bigg)^{1/2}\\
&=4\big(K+2^\epsilon K(K+1)\big)(\delta-1)^{1/2}\|f\|_{L^\8}\bigg(\frac{d(Q_\beta^k)}{\delta^n}\bigg)^{\epsilon/2}.
\end{align*}
Finally, together the estimates of $SV_{I_n}(f_2)(x)$ and $SV_{II_n}(f_2)(x)$ with~\eqref{controll-ineq}, we have

\begin{equation*}
  \begin{split}
  \bigg(\sum_{n>n_3}&V_2(A^\prime_rf_2(x)-A^\prime_rf_2(z_\beta^k):r\in[\delta^{n},\delta^{n+1}))^2\bigg)^{1/2}\\
  &\le 8\big(K+2^\epsilon K(K+1)\big)(\delta-1)^{1/2}\|f\|_{L^\8}\bigg(\sum_{n:\delta^n>d(Q_\beta^k)}\bigg(\frac{d(Q_\beta^k)}{\delta^n}\bigg)^{\epsilon}\bigg)^{1/2}\\
  &+ 8K\big((K+1)\delta^\epsilon\big)^{1/2}\|f\|_{L^\8}\bigg(\sum_{n:\delta^n>d(Q_\beta^k)}\bigg(\frac{d(Q_\beta^k)}{\delta^n}\bigg)^{\epsilon^2}\bigg)^{1/2}\\
  &\le 8\big(K+2^\epsilon K(K+1)\big)\bigg(\frac{\delta-1}{1-\delta^{-\epsilon}}\bigg)^{1/2}\|f\|_{L^\8}
  +8K\bigg(\frac{(K+1)\delta^\epsilon}{1-\delta^{-\epsilon^2}}\bigg)^{1/2}\|f\|_{L^\8},
  \end{split}
\end{equation*}
and the proof is complete.
%\begin{equation*}
%\begin{split}
%  &V_2(A^\prime_rf_2(x)-A^\prime_rf_2(z_\beta^k):r\in[\delta^{n},\delta^{n+1}))\\
%  &\le 4K\big((K+1)\delta^\epsilon\big)^{1/2}\|f\|_{L^\8}\bigg(\frac{d(Q_\beta^k)}{\delta^n}\bigg)^{\epsilon^2/2}+
%  2K\|f\|_{L^\8}(\delta-1)^{1/2}\bigg(\sum_{n:\delta^n>d(Q_\beta^k)}\bigg(\frac{d(Q_\beta^k)}{\delta^n}\bigg)^{\epsilon}\bigg)^{1/2}.
%\end{split}
%\end{equation*}
%Furthermore, the above inequality yields
\end{proof}

%%%%%%%%%%%%%%%%%%%%%%%%%%%%%%%%%%%%%%%%%%%%%%%%%%%%%%%%%%%%%%%%%%%%%%%%%%%%%%%%%%%%%%%%%%%%%%%%%%%%%%%%%%%%%%%%%%%%%%%%%%%%%%%
%%%%%%%%%%%%%%%%%%%%%%%%%%%%%%%%%%%%%%%%%%%%%%%%%%%%%%%%%%%%%%%%%%%%%%%%%%%%%%%%%%%%%%%%%%%%%%%%%%%%%%%%%%%%%%%%%%%%%%%%%%%%%%%
\section{Transference principles}\label{ST6}
In this section, we establish the transference principles for the jump operator.  Recall that a sequence of compact sets $\{F_n\}_{n\in\mathbb{N}}$  with positive measures in a locally compact group $G$ is called a F${\o}$lner sequence if for every $g\in G$
\begin{equation*}
 \lim_{i}\frac{m((F_ng)\triangle F_n)}{m(F_n)}=0,
\end{equation*}
or equivalently
for all compact set $K$ in $G$,
\begin{equation}\label{amen}
  \lim_{n}\frac{m(F_nK)}{m(F_n)}=1.
\end{equation}
A group $G$ is called amenable if it admits such a F${\o}$lner sequence. It is well known that %\end{lemma}%Note that condition~\eqref{asymptotically invariant} implies that the balls $\{B_r\}_r$ is a F${\o}$lner sequence.
if $G$ is a group with polynomial volume growth, then it is amenable (cf. \cite{Gui}), and in particular the family of balls $\{B_r\}_{r>0}$ generated by any word metric on $G$ is a F${\o}$lner sequence (cf. \cite{Breuillard14, Nevo06, Tessera07}). For more information about F${\o}$lner sequences and amenable groups we refer the reader to~\cite{Pat88}.

\subsection{Weak type inequalities.}

\begin{proof}[Proof of Theorem~\ref{thm:trans}$\emph(i)$]
{We only give the proof of transference principle for weak type inequalities, since the strong type one can be done verbatim. }Let $p\in[1,\8)$. Given a $f\in L^p(X,\mu)$. Let $T$ be an action induced by a $\mu$-preserving  measurable transformation $\tau$, that is for all $g\in G$, $T_gf(x)=f(\tau_{g^{-1}}x)$. Fix $x\in X$ and a compact set $A$, define
\begin{equation*}
  F_A(g)=\left\{
    \begin{array}{ll}
      T_gf(x), & \hbox{$g\in A$;} \\
      0, & \hbox{$g\notin A$.}
    \end{array}
  \right.
\end{equation*}
Let $N$ be an integer large enough and $K$ a compact set such that for $r\le N$ we have $ B_r\subseteq K$. Clearly, $F_{AK}(g)=T_gf(x)\mathds{1}_{AK}(g)$. Moreover, if $h\in A$ and $k\in K$, we have
$$T_hT_kf(x)=T_{hk}f(x)=F_{AK}(hk).$$%.$$
It follows that for all $h\in A$ and $r\in(0,N]\cap \mathcal{I}$, we have
\begin{equation}\label{equality-1}
  T_hA_rf(x)=\frac{1}{m(B_r)}\int_{B_r}T_{hg}f(x)\mathds{1}_{AK}(hg)dm(g)=\frac{1}{m(B_r)}\int_{B_r}F_{AK}(hg)dm(g).%=T_h\int_{G}T_gf(x)\mathds{1}_{K}(g)dm_r(g)
\end{equation}
Let $\mathbf{A}_Nf=\{A_rf:r\in(0,N]\cap \mathcal{I}\}$ and $\mathbf{A}^\prime_N F=\{A^\prime_rF:r\in(0,N]\cap \mathcal{I}\}$. For all $h\in G$, set $T_h\mathbf{A}_Nf=\{T_hA_rf:r\in(0,N]\cap \mathcal{I}\}$; here and subsequently, for {a sequence of measurable functions} $\a=\{\a_r:r\in\mathcal{I}\}$, $T\a$ stands for $\{T\a_r:r\in\mathcal{I}\}$. From~\eqref{equality-1}, we have
\begin{equation}\label{equality-2}
  \lambda\sqrt{\mathcal{N}_\lambda(T_h\mathbf{A}_Nf)(x)}=\lambda\sqrt{\mathcal{N}_\lambda(\mathbf{A}^\prime_NF_{AK})(h)}.
\end{equation}

Fix $\gamma>0$, define the set%=
\begin{equation*}
  \mathcal{D}(\gamma)=\big\{(h,x)\in A\times X:\lambda\sqrt{\mathcal{N}_\lambda(T_h\mathbf{A}_Nf)(x)}>\gamma\big\}.
\end{equation*}
Fix $h\in A$ and define the set
\begin{equation*}
  \mathcal{D}^h(\gamma)=\big\{x\in X:\lambda\sqrt{\mathcal{N}_\lambda(T_h\mathbf{A}_Nf)(x)}>\gamma\big\}.
\end{equation*}
Since $T_{h}(\mathds{1}_{\mathcal{D}^e(\gamma)})(x)=\mathds{1}_{\mathcal{D}^h(\gamma)}(x)$, it follows that
\begin{equation}\label{equality-3}
\int_{X}\mathds{1}_{\mathcal{D}^h(\gamma)}(x)d\mu(x)=\int_{X}T_{h}(\mathds{1}_{\mathcal{D}^e(\gamma)})(x)d\mu(x)
=\int_{X}\mathds{1}_{\mathcal{D}^e(\gamma)}(x)d\mu(x).
\end{equation}
On the other hand, fix $x\in X$ and define the set
\begin{equation*}
  \mathcal{D}_x(\gamma)=\big\{h\in A:\lambda\sqrt{\mathcal{N}_\lambda(T_h\mathbf{A}_Nf)(x)}>\gamma\big\}.
\end{equation*}
By~\eqref{equality-2}, one can see that
\begin{equation*}
  \mathcal{D}_x(\gamma)=\big\{h\in A:\lambda\sqrt{\mathcal{N}_\lambda(\mathbf{A}^\prime_NF_{AK})(h)}>\gamma\big\}.
\end{equation*}
Moreover, using the assumption that the jump operator is of weak type $(p,p)$, we have

\begin{equation*}
  m(\mathcal{D}_x(\gamma))\le \frac{\|\lambda\sqrt{\mathcal{N}_\lambda(\mathbf{A}^\prime_N)}\|_{L^p\rightarrow L^{p,\8}}^p}{\gamma^p}\|F_{AK}\|^p_{L^p(G,m)}.
\end{equation*}
It follows from the above inequality that
\begin{equation}\label{inequality-1}
\begin{split}
  \int_X\int_{A}\mathds{1}_{\mathcal{D}_x(\gamma)}(h)dm(h)d\mu(x)&\le \frac{\|\lambda\sqrt{\mathcal{N}_\lambda(\mathbf{A}^\prime_N)}\|_{L^p\rightarrow L^{p,\8}}^p}{\gamma^p}\int_{X}\int_G|F_{AK}(h)|^pdm(h)d\mu(x)\\
  &= \frac{\|\lambda\sqrt{\mathcal{N}_\lambda(\mathbf{A}^\prime_N)}\|_{L^p\rightarrow L^{p,\8}}^p}{\gamma^p}m(AK)\int_{X}|f(x)|^pd\mu(x).
  \end{split}
\end{equation}
By the Fubini  theorem, we can see that
\begin{equation}\label{inequality-2}
\begin{split}
  \int_X\int_A\mathds{1}_{\mathcal{D}_x(\gamma)}(h)dm(h)d\mu(x)&=\int_{G\times X}\mathds{1}_{\mathcal{D}(\gamma)}(h,x)dm(h)d\mu(x)\\
  &=\int_{A}\int_X\mathds{1}_{\mathcal{D}^h(\gamma)}(x)d\mu(x)dm(h).
  \end{split}
\end{equation}
Using~\eqref{equality-3}, we have%the integration term of the second line of the above inequality
\begin{equation*}
\begin{split}
    \int_{A}\int_X\mathds{1}_{\mathcal{D}^h(\gamma)}(x)d\mu(x)dm(h)&=m(A)\int_X\mathds{1}_{\mathcal{D}^e(\gamma)}(x)d\mu(x)\\
    &=m(A)\mu\big(\{x\in X:\lambda\sqrt{\mathcal{N}_\lambda(\mathbf{A}_Nf)(x)}>\gamma\}\big).
\end{split}
\end{equation*}%Now by assumption, the action by and Fubini's theorem&\int_{G\times X}\mathds{1}_{\mathcal{D}(\gamma)}(h,x)dm(h)d\mu(x)=\int_X\int_A\mathds{1}_{\mathcal{D}_x(\gamma)}(h)dm(h)d\mu(x)\\
Together~\eqref{inequality-1}, \eqref{inequality-2} with the above inequality, we conclude
\begin{equation*}
  \mu\big(\{x\in X:\lambda\sqrt{\mathcal{N}_\lambda(\mathbf{A}_Nf)(x)}>\gamma\}\big)\le  \frac{\|\lambda\sqrt{\mathcal{N}_\lambda(\mathbf{A}^\prime_N)}\|_{L^p\rightarrow L^{p,\8}}^p}{\gamma^p}\frac{m(AK)}{m(A)}\int_{X}|f(x)|^pd\mu(x).
\end{equation*}
Since $G$ is an amenable group, by~\eqref{amen}, for any $\varepsilon>0$ we can choose the above subset $A$ such that $m(AK)/m(A)\le(1+\varepsilon)$. By the arbitrariness of $\varepsilon$ and the monotone convergence theorem, letting $N\rightarrow\8$, we have
\begin{equation*}
  \mu\big(\{x\in X:\lambda\sqrt{\mathcal{N}_\lambda(\mathbf{A}f)(x)}>\gamma\}\big)\le  \frac{\|\lambda\sqrt{\mathcal{N}_\lambda(\mathbf{A}^\prime)}\|_{L^p\rightarrow L^{p,\8}}^p}{\gamma^p}\int_{X}|f(x)|^pd\mu(x).
\end{equation*}
which is the desired conclusion.
\end{proof}

\subsection{Strong type inequalities.}\label{strong-type}
%%%%%%%%%%%%%%%%%%%%%%%%%%%%%%%%%%%%%%%%%%%%%%%%%%%%%%%%%%%%%%%%%%%%%%%%
%%%%%%%%%%%%%%%%%%%%%%%%%%%%%%%%%%%%%%%%%%%%%%%%%%%%%%%%%%%%%%%%%%%%%%%%
In this subsection, we {assume that} the action $T$ is a strongly continuous regular action of $G$ on $L^p(X,\mu)$. Before stating the result to be proved, we give some notations and lemmas. The following lemma was proved in~\cite{Mirek-Stein-Zor20}.
\begin{lemma}\label{equivalent norm}
Let $\a=\{\a_r(x),r\in \mathcal{I}\}$ be a sequence of measurable functions on measurable space $(X,\mu)$.~For every $p\in (1,\8)$ and $\theta\in(0,1)$, there exists a positive constant $c_{p,\theta}$ such that%and a norm $\|\cdot\|$ such that%let $\a=(\a_r(x):r>0)$ be a sequence of measure function on $L^p$.
\begin{equation*}%\label{eq norm}
  c^{-1}_{p,\theta}\sup_{\lambda>0}\|\lambda\sqrt{\mathcal{N}_\lambda(\alpha)}\|_{L^p}\le [L^{\8}(X;V_{\8}),L^{\theta p}(X;V_{2\theta})]_{\theta,\8}(\a)\le c_{p,\theta}\sup_{\lambda>0}\|\lambda\sqrt{\mathcal{N}_\lambda(\alpha)}\|_{L^p}.
\end{equation*}
Moreover, if $\max\{1/p,1/2\}<\theta<1$, then the vector-valued interpolation space
$[L^{\8}(X;V_{\8}),$ $L^{\theta p}(X;V_{2\theta})]_{\theta,\8}$
admits an equivalent norm; in particular, if $p>1$, $\sup_{\lambda>0}\|\lambda\sqrt{\mathcal{N}_\lambda(\cdot)}\|_{L^p}$ admits an equivalent norm.
\end{lemma}
See \cite{Mirek-Stein-Zor20} for the definition of vector-valued interpolation spaces and more details about jump quasi-norms.

{An} operator $T:L^p(X,\mu)\rightarrow L^p(X,\mu)$ is called regular if there exists a constant $C>0$ such that
\begin{equation}\label{regular}
  \|\sup_{k\ge 1}|T(f_k)|\|_{L^p}\le C\|\sup_{k\ge 1}|f_k|\|_{L^p},
\end{equation}
for any finite sequence $\{f_k:k\ge 1\}$ in $L^p(X,\mu)$. Let us denote by $\|T\|_r$ the smallest $C$ for which this holds. Let $\mathfrak{B}$ be a Banach space. If $T$ is a regular operator on $L^p(X,\mu)$, then the tensor product operator $T\otimes id_{\mathfrak{B}}:L^p(X,\mu)\otimes\mathfrak{B}\rightarrow L^p(X,\mu)\otimes\mathfrak{B}$ extends to a bounded operator  $\widetilde{T\otimes id_{\mathfrak{B}}}$ from the Bochner space $L^p(X;\mathfrak{B})$ to $L^p(X;\mathfrak{B})$, and
\begin{equation}\label{extension}
  \|\widetilde{T\otimes id_{\mathfrak{B}}}\|_{L^{p}(X;\mathfrak{B})\rightarrow L^{p}(X;\mathfrak{B})}\le \|T\|_{r}.
\end{equation}
For more information {on regular operators} we refer the reader to~\cite{Pis94}. {A group action $T$ of $G$ on $L^p(X,\mu)$ is called regular if for any $g\in G$, $T_g$ is regular and $\sup_{g\in G}\|T_g\|_r<\infty$.}

Moreover, together Lemma~\ref{equivalent norm} with~\eqref{extension}, one can obtain the following lemma.
\begin{lemma}\label{extension of Tg}
 Fix $p\in(1,\8)$. Let $T:L^p(X,\mu)\rightarrow L^p(X,\mu)$ be a regular operator. Given a sequence of measurable functions $\a=\{\a_r(x):r\in \mathcal{I}\}$ in $L^p(X,\mu)$,  there exists a constant $c_p>0$ such that
\begin{equation*}
  \sup_{\lambda>0}\|\lambda\sqrt{\mathcal{N}_\lambda(T\a)}\|_{L^p}\le c_p \|T\|_{r} \sup_{\lambda>0}\|\lambda\sqrt{\mathcal{N}_\lambda(\a)}\|_{L^p}.
\end{equation*}
\end{lemma}

In what follows, we state the strong type $(p,p)$ transference principle for strongly continuous regular group actions. %{\color{red}uniformly bounded positive???regular, and modify accordingly the context in the introduction} group actions. %The proof for  actions induced by measure-preserving measurable transforms is similar{\color{red}delete this setence, since this is included in the class of regular action}.

\begin{proof}[Proof of Theorem~\ref{thm:trans}$\emph{(ii)}$]
Let $p\in (1,\8)$ and $f\in L^p(X,\mu)$. Fix $x\in X$ and a compact set $A$, define $F_{A}(g)=T_gf(x)\mathds{1}_{A}(g)$. Let $N$ be an integer large enough and $K$ a compact set such that for every $r\le N$ we have $B_r\subseteq K$. Clearly, $F_{AK}(h)=T_hf(x)\mathds{1}_{AK}(h)$.

Keeping the notations introduced in the proof of weak type inequalities and using~\eqref{equality-2}, we have
\begin{align*}
  \int_A\big|\lambda\sqrt{\mathcal{N}_\lambda(T_h\mathbf{A}_Nf)(x)}\big|^pdm(h)&=\int_A\big| \lambda\sqrt{\mathcal{N}_\lambda(\mathbf{A}^\prime_NF_{AK})(h)}\big|^pdm(h)\\
   &\le \int_G\big|\lambda\sqrt{\mathcal{N}_\lambda(\mathbf{A}^\prime_NF_{Ak})(h)}\big|^pdm(h).
\end{align*}
Using the strong type $(p,p)$ jump inequality which is assumed for the translation action, we obtain
\begin{equation*}
  \begin{split}
    \int_A\big|\lambda\sqrt{\mathcal{N}_\lambda(T_h\mathbf{A}_Nf)(x)}\big|^pdm(h)&\le\|\lambda\sqrt{\mathcal{N}_\lambda(\mathbf{A}^\prime_N)}\|_{L^p\rightarrow L^{p}}^p\int_{G}|F_{Ak}(h)|^pdm(h)\\
    &= \|\lambda\sqrt{\mathcal{N}_\lambda(\mathbf{A}^\prime_N)}\|_{L^p\rightarrow L^{p}}^p\int_{AK}|T_hf(x)|^pdm(h).
   \end{split}
\end{equation*}
Moreover, integrating both sides of the above inequality over $X$ and using the Fubini theorem, we have
\begin{equation}\label{ineq-1}
  \begin{split}
     &\int_{X}\int_A\big|\lambda\sqrt{\mathcal{N}_\lambda(T_h\mathbf{A}_Nf)(x)}\big|^pdm(h)d\mu(x)\\
     &\le  \|\lambda\sqrt{\mathcal{N}_\lambda(\mathbf{A}^\prime_N)}\|_{L^p\rightarrow L^{p}}^p\int_{AK}\int_{X}|T_hf(x)|^pd\mu(x)dm(h)\\
     &= \|\lambda\sqrt{\mathcal{N}_\lambda(\mathbf{A}^\prime_N)}\|_{L^p\rightarrow L^{p}}^p\sup_{h\in G}\|T_h\|_{r}^pm(AK)\int_{X}|f(x)|^pd\mu(x).
   \end{split}
\end{equation}

On the other hand, by the assumption that $T$ is a strongly continuous {regular} action of $G$ on $L^p(X,\mu)$ and Lemma~\ref{extension of Tg}, we have
\begin{equation*}%\label{ineq-2}
\begin{split}
  \sup_{\lambda>0}\|\lambda\sqrt{\mathcal{N}_\lambda(\mathbf{A}_Nf)}\|_{L^p}&= \inf_{h\in A}\sup_{\lambda>0}\|\lambda\sqrt{\mathcal{N}_\lambda(T_{h^{-1}}T_h\mathbf{A}_Nf)}\|_{L^p}\\
&\le c_p \sup_{h\in G}\|T_h\|_{r} \inf_{h\in A}\sup_{\lambda>0}\|\lambda\sqrt{\mathcal{N}_\lambda(T_h\mathbf{A}_Nf)}\|_{L^p}.
\end{split}
\end{equation*}
{By the above inequality, we have
\begin{equation*}
  \begin{split}
  &\sup_{\lambda>0}\int_{X}|\lambda\sqrt{\mathcal{N}_\lambda(\mathbf{A}_Nf)(x)}|^pd\mu(x)= \frac{1}{m(A)}\int_{A}\sup_{\lambda>0}\int_{X}|\lambda\sqrt{\mathcal{N}_\lambda(\mathbf{A}_Nf)(x)}|^pd\mu(x)dm(h)\\
   &\le c_p\sup_{h\in G}\|T_h\|^p_{r} \frac{1}{m(A)}\int_{A}\sup_{\lambda>0}\inf_{h\in A}\int_X|\lambda\sqrt{\mathcal{N}_\lambda(T_h\mathbf{A}_Nf)(x)}|^pd\mu(x)dm(h)\\
   &= c_p \sup_{h\in G}\|T_h\|^p_{r} \sup_{\lambda>0}\frac{1}{m(A)}\int_{A}\inf_{h\in A}\int_X|\lambda\sqrt{\mathcal{N}_\lambda(T_h\mathbf{A}_Nf)(x)}|^pd\mu(x)dm(h)\\
   &\le c_p \sup_{h\in G}\|T_h\|^p_{r} \sup_{\lambda>0}\frac{1}{m(A)}\int_A
   \int_X|\lambda\sqrt{\mathcal{N}_\lambda(T_h\mathbf{A}_Nf)(x)}|^pd\mu(x)dm(h)
  \end{split}
\end{equation*}}
Using the Fubini theorem and~\eqref{ineq-1}, the above inequality yields
\begin{align*}
  &\sup_{\lambda>0}\int_X\big|\lambda\sqrt{\mathcal{N}_\lambda(\mathbf{A}_Nf)(x)}\big|^pd\mu(x)\\
  &\le  c_p\sup_{\lambda>0}\|\lambda\sqrt{\mathcal{N}_\lambda(\mathbf{A}^\prime_N)}\|_{L^p\rightarrow L^{p}}^p\sup_{h\in G}\|T_h\|_{r}^{2p}\frac{m(AK)}{m(A)}\int_{X}|f(x)|^pd\mu(x).
\end{align*}
By~\eqref{amen} and a similar argument as in the proof of weak type inequalities, letting $N\rightarrow\8$, we have
\begin{equation*}
  \sup_{\lambda>0}\|\lambda\sqrt{\mathcal{N}_\lambda(\mathbf{A}f)}\|_{L^p}\le c_p\sup_{\lambda>0} \|\lambda\sqrt{\mathcal{N}_\lambda(\mathbf{A}^\prime)}\|_{L^p\rightarrow L^{p}}\sup_{h\in G}\|T_h\|^2_{r}\|f\|_{L^p},
\end{equation*}
which is the desired conclusion.
\end{proof}
%\begin{remark}
%Note that  the positivity of the action $T$ can be relaxed by assume that there exists a constant $C_p>0$ such that
%\begin{equation*}
%\sup_{\lambda>0}\|\lambda\sqrt{\mathcal{N}_\lambda (T_g\a)}\|_{L^p(X)\rightarrow L^p(X)}< C_p \sup_{\lambda>0}\|\lambda\sqrt{\mathcal{N}_\lambda (\a)}\|_{L^p(X)\rightarrow L^p(X)}.
%\end{equation*}
%\end{remark}
%%%%%%%%%%%%%%%%%%%%%%%%%%%%%%%%%%%%%%%%%%%%%%%%%%%%%%%%%%%%%%%%%%%%%%%%%%%%%%%%%%%%%%%%%%%%%%%%%%%%%%%%%%%%%%%%%%%%%%%%%%%%%%%%%
%%%%%%%%%%%%%%%%%%%%%%%%%%%%%%%%%%%%%%%%%%%%%%%%%%%%%%%%%%%%%%%%%%%%%%%%%%%%%%%%%%%%%%%%%%%%%%%%%%%%%%%%%%%%%%%%%%%%%%%%%%%%%%%%%
\section{{Annular decay property}}\label{ST7}
{In this section, we discuss the annular decay property. We first recall the $(\epsilon,1)$-annular decay property of word metrics and verify that this property is stable under $(1,C)$-quasi isometry (recalled below), and thus obtain the quantitative ergodic theorems, including Theorem \ref{main-thm2}, on the polynomial growth group equipped with a metric that is $(1,C)$-quasi isometric to a word metric. We then check that all the known examples of periodic metric, which was introduced by Breuillard~\cite{Breuillard14}, satisfy some $(\epsilon,r_0)$-annular decay property. At the moment of writing, we do not know how to verify this property for all periodic metrics. }

%we first focus on the polynomial growth group equipped with a word metrics. We verify that such group endowed with word metric satisfies conditions \eqref{decay property}-\eqref{asymptotically invariant}. Furthermore, we discuss the groups of polynomial volume growth endowed with periodic metric. According to Breuillard~\cite{Breuillard14}, we know that the group endowed with periodic metric satisfies the measure doubling condition and condition~\eqref{asymptotically invariant}. Nevertheless, the approach in~\cite{Breuillard14} to verify  condition \eqref{decay property} looks quite difficult. But there are serval examples of periodic metrics which are introduced by Breuillard~\cite{Breuillard14} satisfying the condition \eqref{decay property}. It seems to be true that condition~\eqref{decay property} holds for all periodic metrics.%$(\epsilon,r_0)$-annular decay property.

Let $G$ be a polynomial growth group with a symmetric compact generating set $V$. Recall that the word metric $d$ is defined by
\begin{equation*}
 \forall~x,y\in G,\qquad~d(x,y)=\inf\{n\in\mathbb{N},x^{-1}y\in V^n\}.
\end{equation*}
It is clear that $d$ is a {(left-)} invariant metric on $G$. Let $r>0$ and $B_r$ be the ball generated by word metric $d$ of radius $r$ in $G$. It is well known that there exist two constants $C_{V}>0$ and $D_G>0$ such that for every $r\in(0,\infty)$,
\begin{equation}\label{ball-ineq1}
  C_V^{-1}r^{D_G}\le m(B_r)\le C_Vr^{D_G}.
\end{equation}
Form the above inequality, it is easy to check that $(G,d,m)$ satisfies the measure doubling condition, it follows that such $(G,d,m)$ satisfies condition~\eqref{geo-doubling}. By an argument same as in the proof of~\cite[Theorem 4]{Tessera07}, we have the following proposition. %is due to Tessera~\cite[Theorem 4.14]{Tessera07}.Moreover, Tessera~\cite{Tessera07} and Nevo~\cite{Nevo06} proved that $(G,d,m)$ satisfies the following proposition. {\color{red} Are you sure that Tessera obtained the $\varepsilon$-annular decay property, but not the  $(\varepsilon,1)$-annular decay property?? See Remark 4.1 in Lin-Nakai-Yang's paper.}
\begin{prop}\label{word}
Let $G$ be a polynomial growth group with a symmetric compact generating set $V$ and $\{B_r\}_{r>0}$ be the balls given by the corresponding word metric. Then there exist two constants $\theta=\log_2(1+\frac{1}{C_V^210^{D_G}})$ and  $c_V=(1+\frac{1}{C_V^210^{D_G}})^3$ such that for all $r\in[1,\8)$ and $s\in(0,r]$,
\begin{equation}\label{shell}
  m(B_{r+s}\setminus B_{r})\le c_V\bigg(\frac{s}{r}\bigg)^{\theta}m(B_r).
\end{equation}
\end{prop}

\begin{remark}
\emph{In terms of our terminology, the above ball annular decay property is just the $(\epsilon, 1)$-annular decay property~\eqref{decay property} of $(G,d,m)$. Note that all the groups of polynomial growth are amenable (cf.~\cite{Gui}),
thus Theorem \ref{main-thm2} follows from Theorem \ref{main-thm1} by the transference principles. }
\end{remark}

In fact, in~\cite[Theorem 4]{Tessera07}, {the $(\epsilon,1)$-annular decay property is established for every  metric measure space satisfying the measure doubling condition and Property ($M$).}
%Tessera~\cite{Tessera07} proved that the metric space $(\mathcal{X},d,\mu)$ satisfies~\eqref{shell} for all $r\ge 1$ and $k=1$ if $(\mathcal{X},d,\mu)$ satisfies the  measure doubling condition and Property $(M)$.
%we can establish using the same arguments as in the  proof of~\cite[Theorem 4]{Tessera07}
% Moreover, \eqref{shell} implies that such metric space $(G,d,m)$ satisfies the $(\epsilon,1)$-annular decay property.
% rediscovering a cute argument of Colding and Minicozzi~\cite[Lemma 3.3]{col-Min98},  if $(\mathcal{X},d,\mu)$. satisfies the measure doubling condition and the monotone geodesics property. %As an application,
%$(\epsilon,1)$-annular decay property
% Moreover, we can see that~\eqref{shell} is stronger than condition~\eqref{decay property}, namely $(\epsilon,r_0)$-annular decay property. {\color{red}this annular decay property is defined for a metric measure space, that is, for the balls $(B_r)_{r>0}$, not just $(B_r)_{r\in\mathbb N}$!!!}
%In the following, it will cause no confusion if we use the same letters to designate the group and the metric space.
%In particular, when $(\mathcal{X},d,\mu)$ is a polynomial growth group endowed with a word metric and haar measure, such $(\mathcal{X},d,\mu)$ satisfies the measure doubling condition and the monotone geodesics property.  %Tessera [35] who rediscovered a cute argument let $(G,d)$ be
Recall that a metric space  $(\mathcal{X},d,\mu)$ is said to satisfy Property $(M)$ if there exists a constant $C>0$ such that the Hausdorff distance between any pair of balls with same center and any radii between $r$ and $r+1$ is less than $C$. In other words, for all $x\in\mathcal{X}$, $r>0$ and $y\in B(x,r+1)$, we have $d(y,B(x,r))\le C$. Property $(M)$ is equivalent to the property: there exists a constant $C<\8$ such that for all $r>0$, $s\ge1$, $y\in B(x,r+s)$

%for all There are two equivalent properties for
%the monotone geodesic property if there exists a  constant $C>0$, such that for all $x,y\in\mathcal{X}$, $d(x,y)\ge 1$, there exists a finite chain $x_0=y,x_1,\cdots,x_m=x$ such that for $0\le i<m$, $d(x_i,x_{i+1})\le C$ and $d(x_{i+1},x)\le d(x_i,x)-1$.
\begin{equation}\label{monotone}
  d(y,B(x,r))\le Cs;
\end{equation}
and is also equivalent to that  the metric space $(\mathcal{X},d,\mu)$ admits monotone geodesics, {see e.g. ~\cite[Proposition 2]{Tessera07} for the relevant definitions and proofs}. Moreover, Property (${M}$) is invariant under Hausdorff equivalence but unstable under quasi-isometry in the sense of \cite[Page 50]{Tessera07}, where one can find the relevant counterexamples.

We prove that Property ($M$) is invariant under the ($1,C$)-quasi-isometry.
Two metrics $d_1$ and $d_2$ on $\mathcal{X}$ are called $(1,C)$-quasi-isometric if there exists a constant $C>0$ such that for any $x,y\in\mathcal{X}$, $|d_{1}(x,y)-d_{2}(x,y)|\le C$.%$\mathcal{X}=\mathcal{Y}$ and $\varphi=\mathds{1}_{id}$.
%Recall that two metric spaces $(\mathcal{X},d_{\mathcal{X}})$,  and $(\mathcal{Y},d_{\mathcal{Y}})$  are called $(1,C)$-quasi-isometric if there exists a map $\varphi:\mathcal{X}\rightarrow\mathcal{Y}$ such that for any $y\in\mathcal{Y}$, we can find some points $x\in\mathcal{X}$ such that $d_{\mathcal{Y}}(y,\varphi(x))\le C$ and for all $x_1,x_2\in\mathcal{X}$, $|d_{\mathcal{Y}}(\varphi(x_1),\varphi(x_2))-d_{\mathcal{X}}(x_1,x_2)|\le C$.
%
\begin{prop}\label{quasi-iso}
Let $d_{1}$ and $d_2$ be two metrics on $\mathcal{X}$. Assume that $d_1$ is $(1,C_{\mathcal{X}})$-quasi-isometric to $d_2$ and $(\mathcal{X}, d_{1})$ satisfies property ($M$), then $(\mathcal{X}, d_{2})$ satisfies property ($M$).
\end{prop}
\begin{proof}
Fix a point $x\in \mathcal{X}$, $s\ge 1$ and $r>0$. We denote by $B_1(x,r)$ and $B_2(x,r)$ the balls generated by $d_1$ and $d_2$, respectively. Since  $d_{1}$ is $(1,C_{\mathcal{X}})$-quasi-isometric to $d_{2}$, it follows that $B_2(x,r)\subseteq B_1(x,r+C_\mathcal{X})$. Let $z\in B_2(x,r+s)$. We split $r$ into two cases: $r>C_{\mathcal{X}}$ and $r\le C_{\mathcal{X}}$. For $r>C_{\mathcal{X}}$, we have $B_1(x,r-C_\mathcal{X})\subseteq B_2(x,r)$ and it follows that
\begin{equation*}
  d_2(z,B_2(x,r))\le d_1(z,B_1(x,r-C_{\mathcal{X}}))+C_{\mathcal{X}}.
\end{equation*}
Combining \eqref{monotone} for $d_1$ with the fact that $s\ge 1$ and $z\in B_1(x,r+s+C_\mathcal{X})$, the above inequality yields
\begin{equation*}
   d_2(z,B_2(x,r))\le C(s+2C_{\mathcal{X}})+C_{\mathcal{X}}\le (C+2CC_{\mathcal{X}}+C_{\mathcal{X}})s.
\end{equation*}
For $r\le C_{\mathcal{X}}$, it is easy to check that $d_2(z,B_2(x,r))\le r+s\le (C_{\mathcal{X}}+1)s$. Combining this case with the case $r>C_{\mathcal{X}}$, we prove that $(\mathcal{X}, d_{2})$ satisfies~\eqref{monotone}, which completes the proof.
\end{proof}%left-invariant

\begin{remark}
\emph{By the above observations, we can see that any left-invariant metric, defined on a polynomial volume growth group $G$ which is $(1,C)$-quasi-isometric to a word metric, satisfies the $(\epsilon,1)$-annular decay property. Thus Theorem \ref{main-thm2} holds for any metric that is  ($1,C$)-quasi-isometric to a word metric.}%let $G$ be of polynomial volume growth group, $d_1$ is a metric on $G$ and (1,C)-quasi-isometric to a word metric $d$
\end{remark}

{Motivated by the notion---Property $(M)$, we can introduce Property $(M_{r_0})$, namely there exists a positive constant $C<\8$ such that the Hausdorff distance between any pair of balls with same center and any radii belonging to $[r,r+1]$ with $r>r_0$ is less than $C$.
%for all $x\in\mathcal{X}$, $r>0$ and $y\in B(x,r+r_0)$, we have $d(y,B(x,r))\le C$.
Similar to \cite[Theorem 4]{Tessera07}, one can show that for a doubling metric measure space $(X,d,\mu)$ with Property $(M_{r_0})$, then there exists $\theta>0$ and a constant $C>0$ such that for all $x\in X$, $r\in[r_0,\8)$ and $s\in (0,r]$, we have $\mu(B_{r+s}\setminus B_{r})\le C\big({s}/{r}\big)^{\theta}\mu(B_r)$.} As shown in Proposition \ref{quasi-iso}, Property $(M_{r_0})$ is also stable under $(1,C)$-quasi-isometry.

\bigskip

{If a metric measure space satisfies $(\epsilon,r_0)$-annular decay property for all $r_0>0$, then we call it satisfy the $\epsilon$-annular decay property.}
%A stronger property is  the $\epsilon$-annular decay property, that is,
%A slightly stranger property than $(\epsilon,r_0)$-annular decay property in metric space is the $\epsilon$-annular decay property. The metric space $(\mathcal{X},d,\mu)$ is said to satisfy the $\epsilon$-annular decay property with $\epsilon\in(0,1]$ if there exists a constant $C\ge 1$ such that for all $x\in\mathcal{X}$, $r>0$, $\varepsilon\in(0,1)$, we have
%\begin{equation*}
  %\mu(B(x,r))-\mu(B(x,(1-\varepsilon)r))\le C\varepsilon^{\epsilon}\mu(B(x,r)).
%\end{equation*}
%{\color{red} it is not small r on the right hand side???this defnition is not consistent with the ones appeared in the references}
This terminology, to our best knowledge, was introduced by  Buckley~\cite[(1.1)]{Buck99} on metric space. A slight variant  on manifold was introduced by Colding and Minicozzi~\cite{col-Min98}, which they called the $\epsilon$-volume regularity property. In recent years the $\epsilon$-annular decay property has been widely exploited in harmonic analysis, see~\cite{Arr-Llo19, Aus-Rou13, Kin-Shu, Kin-Shu14, Lin-Nakai-Yang11, Zorin-Kranich20}  for more details.

{The following example tells us  that there are numerous metric measure spaces only satisfying the $(\epsilon,r_0)$-annular decay property for some $r_0$ but not all $r_0>0$.}
%Riemannian geometry.
%In the following,  we give an example to verify that the $(\epsilon,r_0)$-annular decay property is weaker than the $\epsilon$-annular decay property. {\color{red} this conclusion is trivially right from the definition!!! the point is that there are numerous examples which satisfy only the weaker property.}
\begin{example}
Fix a positive integer $r_0$. Let $\mathcal{X}=\mathbb{Z}$ endowed with the counting  measure $\mu$. The metric $d$ is given by
\begin{equation*}
  d(x,y)=\left\{
           \begin{array}{ll}
             0, & \hbox{$x=y$,} \\
              r_0/2,&\hbox{$0<|x-y|\le r_0/2$,}\\
             \max\{|x-y|,r_0\}, & \hbox{$|x-y|>r_0/2$.}
           \end{array}
         \right.
\end{equation*}
One can check that for any $k\in (-\8,-1)\cap\mathbb{Z}$, ${\mu(B(0,r_0)\setminus B(0,r_0-2^k))}=r_0$, it follows that this metric space does not satisfy $\epsilon$-annular decay property but satisfies the $(\epsilon,r_0)$-annular decay property.
\end{example}
%But for $\epsilon$-annular decay property, there are some Nevertheless,
In \cite{Buck99}, Buckley proved that the metric space $(\mathcal{X},d,\mu)$ has the $\epsilon$-annular decay property when it satisfies the measure doubling condition and the $(\a,\b)$-chain ball property. {Lin, Nakai and Yang~\cite{Lin-Nakai-Yang11} established the $\epsilon$-annular decay property for the metric space $(\mathcal{X},d,\mu)$ if $(\mathcal{X},d,\mu)$ satisfies the measure doubling condition and the weak geodesic property (i.e., \eqref{monotone} holds for all $s>0$)}. Actually, Lin {\it et al} proved that the weak geodesic property is equivalent to the $(\a,\b)$-chain ball property, and is also equivalent to the monotone geodesic property. A typical class of metric spaces having the $(\a,\b)$-chain ball property (or the weak geodesic property) is the length spaces. The following proposition was established in \cite[Corollary 2.2]{Buck99}.  %class of metric spaces%which is defined \cite{Lin-Nakai-Yang11}. %with $s>0$, and is also equivalent to $(\a,\b)$-chain ball property
%If~\eqref{monotone} holds for all $s>0$,  Lin, Nakai and Yang called it the weak geodesic property~\cite{Lin-Nakai-Yang11}. They proved that the weak geodesic property is equivalent to the monotone geodesic property with $s>0$, and is also equivalent to $(\a,\b)$-chain ball property which was introduced by Buckley~\cite{Buck99}. Moreover, by~\cite{Buck99} and \cite{Lin-Nakai-Yang11}, we have the following result.{\color{red}The following result is due to Lin et al, no???}metric spaces in which the distance between anypair of points equalsthe infimum of the lengths of rectifiable paths joining them
%Recall that a metric space $\mathcal{X}$ is called a length space  if the distance between any two points is the infimum of the lengths of all curves joining the two points.
\begin{prop}\label{length space}
If $(\mathcal{X},d,\mu)$ is a length space (a metric space in which the distance between any two points is the infimum of the lengths of all curves joining the two points) and satisfies the measure doubling condition. Then $(\mathcal{X},d,\mu)$ satisfies the $\epsilon$-annular decay property.
\end{prop}

Basic examples of length spaces include the graph which is defined in~\cite[Page 51]{Tessera07}, the homogeneous groups endowed with homogeneous metrics (such as $\real^n$ and the Heisenberg group $\mathbb{H}^n$), the Riemannian metric spaces,  the Finsler metric spaces, the subRiemannian metric and the subFinsler metric spaces (also called the Carnot-Carath\'{e}odory metric spaces). For more examples we refer the reader to the book~\cite{Gromov99}.

In addition to length spaces, the RD-spaces with condition ($H_\a$), $\a\in(0,1]$ satisfy also the $\epsilon$-annular decay property, see e.g. \cite[Example 4.1(iii)] {Lin-Nakai-Yang11} for the assertion and the related notions.% It follows from \cite[Exercise 9.14.]{Hein01} and \cite[Example(4.5)]{Ma-Tor88} that there exists an RD-space with condition ($H_\a$) which is not a length space.

\bigskip

Finally, we  focus on the periodic metric spaces.  Recall that a pseudodistance $d$ on locally compact group $G$ is called a periodic metric if it satisfies the following properties:
 \begin{enumerate}[\noindent]
  \item~{(i)}~$d$ is invariant under left translations by a closed co-compact subgroup $H$, meaning that for all $x,y\in G$ and all $h\in H$, $d(hx,hy)=d(x,y)$;
  \item~{(ii)}~$d$ is locally bounded and proper;
  \item~{(iii)}~$d$ is asymptotically geodesic.
\end{enumerate}
For more information about the periodic metrics we refer the reader to~\cite[Section 4]{Breuillard14}. Moreover, in \cite{Breuillard14} the following result is proved. There exist two constants $d_G>0$ and  $c_d>0$ such that
\begin{equation*}
  \lim_{r\rightarrow\8}\frac{m(B_r)}{r^{d_G}}=c_d,
\end{equation*}
where $B_r$ is a ball given by the periodic metric of radius $r$ in a polynomial growth group $G$. From the above identity, we can see  that equipped with a left-invariant periodic metric $d$, $(G,d,m)$ satisfies the doubling measure condition and~\eqref{asymptotically invariant}. However  condition~\eqref{decay property} seems quite inaccessible from the above estimate. {But all the periodic metrics provided in~\cite[Section 4]{Breuillard14} satisfy condition~\eqref{decay property}.}
%In the following, we give several examples of periodic metric spaces having the $(\epsilon,r_0)$-annular decay property. has exact polynomial volume growth.
%And for the left-invariant periodic metric $d$, $(G,d,m)$
\begin{example}
\begin{enumerate}[\noindent]
  \item~{(i)}~Let $G$ be a polynomial volume growth group  with a compact symmetric generating set. {The corresponding word metric is a periodic metric.} %AND formulating the examples below in this way }% By~\eqref{shell}, we known that $(G,d,m)$ hashas the $\epsilon$-annular decay property.
  \item~{(ii)}~Let $G$ be a simply connected nilpotent Lie group. Let $\Gamma$ be a finitely generated torsion free nilpotent group which is embedded as a co-compact discrete subgroup of $G$. We denote by $V$ the generating set of $\Gamma$ and $d_V$ the word metric on $\Gamma$. {The metric $d$ on $G$ which is defined by $d(x,y)=d_V(h_x,h_y)$ is a periodic metric, where $x\in h_xF$ and $y\in h_yF$ and $F$ is some fixed fundamental domain for $\Gamma$ in $G$.} %$d_V$ is a word metric on polynomial volume growth $\Gamma$ equipped  with a finite symmetric generating set $V$.

  \item~{(iii)}~Let $G/\Gamma$ be a nilmanifold with universal cover $G$ and fundamental group $\Gamma$. Let $d$ be a Riemannian metric on $G/\Gamma$. {The Riemannian metric $\tilde{d}$ on $G$ which is extended by Riemannian metric $d$ is a periodic metric.}
   \item~{(iv)}~Let $G$ be a connected Lie group. {The left invariant Carnot-Carath\'{e}odory  metric or left invariant Riemannian metric on $G$ is a periodic metric.}
\end{enumerate}
\end{example}
Note that Examples (i) and (ii) are the spaces endowed with word metrics, then by Proposition~\ref{word}, such spaces satisfy the $(\epsilon,1)$-annular decay property; Examples (iii) and (iv) are actually length spaces, and so by Proposition~\ref{length space}, such spaces satisfy the $\epsilon$-annular decay property.%Under the additional assumption that $d$ is a left-invariant metric

Based on the above examples, one can not expect the $\epsilon$-annular decay property for all periodic metrics, but the following conjecture is still expected; if this were the case, we would be able to obtain the quantitative ergodic theorems for all periodic metrics on polynomial growth groups.%Nevertheless, we don't know how to prove that the simply connected and solvable Lie group holds the $\epsilon$-annular decay property holds for.
\begin{que}
Let $G$ be a polynomial growth group endowed with a left-invariant periodic metric $d$ and a Haar measure $m$. Then there exists one $r_0>0$ such that $(G,d,m)$ satisfies the  $(\epsilon,r_0)$-annular decay property.%Does  there exists a $r_0>0$ such that $(G,m,d)$ satisfies the $(\epsilon,r_0)$-annular decay property?
\end{que}

Combining Proposition~\ref{quasi-iso} with  the proof of Corollary 1.6 in~\cite{Breuillard14}, we can see that in order to prove the  $(\epsilon,r_0)$-annular decay property for polynomial growth group $G$ endowed with left-invariant periodic metric, the question is left to prove the case when $G$ is a simply connected solvable Lie group of polynomial growth.

%%%%%%%%%%%%%%%%%%%%%%%%%%%%%%%%%%%%%%%%%%%%%%%%%%%%%%%%%%%%%%%%%%%%%%%%%%%%%%%%%%%%%%%%%%%%%%%%%%%%%%%%%%%%%%%%%%%%%%%%%%%%%%%%%%
%%%%%%%%%%%%%%%%%%%%%%%%%%%%%%%%%%%%%%%%%%%%%%%%%%%%%%%%%%%%%%%%%%%%%%%%%%%%%%%%%%%%%%%%%%%%%%%%%%%%%%%%%%%%%%%%%%%%%%%%%%%%%%%%%%
\section{Exponential decay estimates}\label{ST8}

In this section, we establish the jump estimates with exponential decay, namely Theorem~\ref{exponential-estimate}. Let $G$ be a group of polynomial growth with a symmetric finite generating set $V$ in this section, and  $\mathbf{A}^\prime=\{A^\prime_r:r\in\mathbb{N}\}$ be the sequence of averaging operators given by~\eqref{averaging operator1}. {It is certainly interesting to establish Theorem~\ref{exponential-estimate} in the more general setting. However, we do not know how to prove it.}

We start with several lemmas.
The first one is a transference principle, which can be established verbatim using the arguments as in the proof of Theorem~\ref{thm:trans}(i). We omit the details.
\begin{lemma}\label{flu}
Let $p\in[1,\8)$. Let $T$ be an action induced by a $\mu$-preserving measurable transformation $\tau$ on $X$. Let $\mathbf{A}=\{A_r:r\in\mathbb{N}\}$ %and $\mathbf{A}^\prime=\{{A}^\prime_r:r\in\mathbb{N}\}$
be the sequence of averaging operators given  by~\eqref{averaging operator}. If there exist two constants $C_\lambda>0$ and $c_\lambda\in (0,1)$ such that for every $n\in\mathbb N$ and $F\in L^p(G,m)$%with $\|F\|_{L^{\8}(G,m)}\le 1$
\begin{equation*}%\label{jump-ineq}
   m\big(\{g\in G:\mathcal{N}_{\lambda}(\mathbf{A}^\prime F)(g)>n\}\big)\le C_{\lambda}c_{\lambda}^n\|F\|^p_{L^p(G,m)},
\end{equation*}%with $\|f\|_{L^{\8}(X,\mu)}\le 1$
then for every $n\in\mathbb N$ and $f\in L^p(X,\mu)$,
\begin{equation*}%
   \mu\big(\{x\in X:\mathcal{N}_{\lambda}(\mathbf{A}f)(x)>n\}\big)\le C_{\lambda}c_{\lambda}^n\|f\|^p_{L^p(X,\mu)}.
\end{equation*}
\end{lemma}

With the above transference principle, it suffices to show Theorem~\ref{exponential-estimate} when $X=G$.
The second lemma that we need is a trivial jump estimate following from Theorem~\ref{main-thm2}.
%Combining \eqref{deal with variational operator} with Theorem~\ref{the estimate of square function}, Theorem~\ref{the estimate of short variation} and Lemma~\ref{lem:jump}, we have the following lemma.% (taking $\gamma=\lambda\sqrt{n}$).
\begin{lem}\label{averaging-oper}
 For every $p\in[1,\8)$, there exists a constant $c_p>0$ such that for all $F\in L^p(G,m)$,
\begin{equation*}%\label{upcrossings estimate}
  m\big(\{g\in G:\mathcal{N}_{\lambda}(\mathbf{A}^\prime F)(g)>n\}\big)\le \frac{c_p}{\big(\lambda\sqrt{n}\big)^p}\|F\|^p_{L^p(G,m)}.
\end{equation*}
\end{lem}

To present the next two lemmas, we fix $\lambda>0$ and $F\in L^p(G,m)$ with $\|F\|_{L^\8}\le 1$. For $q\in\mathbb{N}$, define
\begin{equation*}
  \mathcal{F}_q= \mathcal{F}_q(\lambda,F)=\{x\in G:\mathcal{N}_{\lambda}(\mathbf{A}^\prime F)(x)>q\},
\end{equation*}
and
\begin{equation*}%\mathcal{F}^\prime_q
 \mathcal{F}^\prime_q(F,\lambda)=\{(x,r_0):x\in G,\exists~1\le r_0<r_1<\cdots<r_q~\textit{such~that}~\min_{1\le i\le q}|A^\prime_{r_i}F(x)-A^\prime_{r_{i-1}}F(x)|>\lambda\}.
\end{equation*}
Let $\mathcal{G}^\prime_q=\mathcal{G}^\prime_q(F,\lambda)=\{B(x,r_0):(x,r_0)\in \mathcal{F}^\prime_q(F,\lambda)\}$ and $\mathcal{G}_q=\mathcal{G}_q(F,\lambda)=\cup_{B\in \mathcal{G}^\prime_q}B$. It is clear that $\mathcal{F}_q\subseteq \mathcal{G}_q$.

%Let $\mathcal{G}^\prime_q=\{B(x,r_0):(x,r_0)\in \mathcal{F}^\prime_q\}$ and $\mathcal{G}_q=\cup_{B\in \mathcal{G}^\prime_q}B$$\mathcal{G}_q=\cup_{(x,r_0)\in \mathcal{F}^\prime_q}B(x,r_0)$
%Applying Lemma~\ref{averaging-oper}, we will prove the following result.
%Let the constants $C_m$ and $\kappa$ are defined in \eqref{shell}.
%be the set that contains all balls $B$ which satisfy the property that for some point $x\in G$ there exists a sequence $1\le r_0<r_1<\cdots<r_q$ such that $B=B(x,r_0)$ and
%\begin{equation*}
%  |A_{r_i}F(x)-A_{r_{i-1}}F(x)|>\lambda,~\forall~1\le i\le q.
%\end{equation*}
% So there exists some balls in with the same center but radius are different.

 %The following lemma says that for the set $\mathcal{G}_q$ also satisfies the analogous conditions (i) and (ii) of Theorem~\ref{transfer principle for upcroosing}. for $\mathcal{G}_q$
Set
 $$C_{V,\lambda}=\min\{1/(8c_V),1/\lambda\},$$

$$\Phi(q)=2^p\cdot3^{D_G}c_pC^4_VC_{V,\lambda}^{-\frac{D_G}{\theta}}\lambda^{-\frac{D_G}{\theta}-p}q^{-\frac{p}{2}},$$
 where the constants $D_G$, $C_V$, $c_V$ and $\theta$ were given in Section~\ref{ST7}.
\begin{lemma}\label{measure estimate of Gq}
For any $q\in\mathbb{N}$, one has %let the function $F\in L^p(G,m)$ with $\|F\|_{L^\8(G,m)}\le 1$, then%there exists a constant $C>0$ such that%for every compact supported function $g:G\rightarrow\mathbb{R}$ with $\|g\|_{L^{\8}}\le 1$, we have
  \begin{equation*}
    m(\mathcal{G}_q(F,\lambda))\le \Phi(q)\|F\|^p_{L^p(G,m)}.
    \end{equation*}
\end{lemma}

%\begin{remark}
%{\color{red} The proof of this lemma is based on the Vitali covering lemma and the assumption that $m$ is a counting measure.   At the moment of writing, we do not know how to prove this lemma for general measure.}
%\end{remark}

\begin{proof}
%Without loss of generality, we assume that $\alpha>0$.
We first prove the following inequality %for every $r\in\mathbb{N}$,
\begin{equation}\label{con}
m(B(x,r))\le  C^2_V(C_{V,\lambda}\lambda)^{{D_G}/\theta} m(B(x,(C_{V,\lambda}\lambda)^{1/\theta}r)),~\forall~r\in\mathbb{N}.
\end{equation}
Indeed, if $(C_{V,\lambda}\lambda)^{1/\theta}r<1$, since $m$ is a counting measure, then $ m(B(x,(C_{V,\lambda}\lambda)^{1/\theta}r))=1$. By~\eqref{ball-ineq1}, we have $m(B(x,r))\le C_V(C_{V,\lambda}\lambda)^{-\frac{D_G}{\theta}}m(B(x,(C_{V,\lambda}\lambda)^{1/\theta}r))$. If $(C_{V,\lambda}\lambda)^{1/\theta}r\ge1$, using~\eqref{ball-ineq1} again, we have $m(B(x,r))\le  C^2_V(C_{V,\lambda}\lambda)^{-\frac{D_G}\theta} m(B(x,(C_{V,\lambda}\lambda)^{1/\theta}r))$, and so~\eqref{con} is proved.

 Applying the Vitali covering lemma, we can select a subset
$$\{B_{j_1},B_{j_2}, \cdots, B_{j_n},\cdots\}\subseteq\mathcal{G}^\prime_q$$
of pairwise disjoint balls satisfying
\begin{equation}\label{Vitali covering}
  \mathcal{G}_q\subseteq\cup_{i} 3B_{j_i}.
\end{equation}
For each ball $B_{j_i}=B(x_{j_i},r_{j_i})$ selected, by the definition of $\mathcal{G}^\prime_q$,  there exists a sequence $1\le r_{j_{i}}<r_{j_{i}+1}<\cdots<r_{j_{i}+q}$ such that
\begin{equation*}
  |A^\prime_{r_{j_{i}+k}}F(x_{j_i})-A^\prime_{r_{j_{i}+k-1}}F(x_{j_i})|>\lambda,~\forall~1\le k\le q.
\end{equation*}
Now we fix such ball $B(x_{j_i},r_{j_{i}})$ and sequence $1\le r_{j_{i}}<r_{j_{i}+1}<\cdots<r_{j_{i}+q}$.  We claim that for all $y\in  B(x_{j_i},(C_{V,\lambda}\lambda)^{1/\theta}r_{j_{i}})$ and $1\le k\le q$,
\begin{equation*}
  |A^\prime_{r_{j_{i}+k}}F(y)-A^\prime_{r_{j_{i}+k-1}}F(y)|>\lambda/2,
\end{equation*}
namely $B(x_{j_i},(C_{V,\lambda}\lambda)^{1/\theta}r_{j_{i}})\in \mathcal{F}_q\big(\frac{\lambda}{2},F\big)$.

Assume this claim momentarily. We have
\begin{align*}
  m\big(\cup_{i} B_{j_i}\big)&=\sum_im(B_{j_i})%\le C^2_V\bigg(\frac{8c_V}{\lambda}\bigg)^{{D_G}/\theta}\sum_im(B_{j_i})\\
\le C^2_V\big(C_{V,\lambda}\lambda\big)^{-\frac{D_G}{\theta}}\sum_i m\bigg(B\big(x_{j_i},(C_{V,\lambda}\lambda)^{1/\theta}r_{j_{i}}\big)\bigg)\\
&\le C^2_V\big(C_{V,\lambda}\lambda\big)^{-\frac{D_G}{\theta}} m\bigg(\cup_iB\big(x_{j_i},(C_{V,\lambda}\lambda)^{1/\theta}r_{j_{i}}\big)\bigg)\\
&\leq C^2_V\big(C_{V,\lambda}\lambda\big)^{-\frac{D_G}{\theta}}m\bigg(\mathcal{F}_q\big(\frac{\lambda}{2},F\big)\bigg)\le C^2_V\big(C_{V,\lambda}\lambda\big)^{-\frac{D_G}{\theta}}\frac{2^pc_p}{\big(\lambda\sqrt{q}\big)^p}\|F\|^p_{L^p},
\end{align*}
where {the equality} follows from the disjointness of the balls $B_{j_i}$, the first inequality follows from~\eqref{con}, the second inequality follows from the fact $C_{V,\lambda}\lambda\le 1$ and the disjointness of the balls $B_{j_i}$, the third inequality follows from the claim and the last inequality follows from Lemma~\ref{averaging-oper}.

Moreover, combining the above inequality with~\eqref{Vitali covering}, we have
\begin{equation*}
  m(\mathcal{G}_q)\le m(\cup_{i} 3B_{j_i})\le 3^{D_G}C^2_V \sum_im(B_{j_i})\le 2^p\cdot3^{D_G}c_pC^4_VC_{V,\lambda}^{-\frac{D_G}{\theta}}\lambda^{-\frac{D_G}{\theta}-p}q^{-\frac{p}{2}}\|F\|^p_{L^p},
\end{equation*}
and the conclusion is proved.
%From the above inequality and $\|F\|_{L^\8}\le 1$,
%Note that for each $s_i<r_i$, $B(x_{j_i},r_{j_i})\subseteq B(x_{j_i},s_i)\subseteq B(x_{j_i},r_i)$.

We now prove the claim. Fix $y\in  B(x_{j_i},(C_{V,\lambda}\lambda)^{1/\theta}r_{j_{i}})$.
%Note that if $(\frac{\lambda}{8c_V})^{1/\theta}r_{0}<1$, then the claim is trivial. So we only need to prove the claim for the case $(\frac{\lambda}{8c_V})^{1/\theta}r_{0}\ge1$.
By a simple geometric argument, we can check at once that for every $0\le k\le q$, $B(x_{j_i},r_{j_i+k})\triangle B(y,r_{j_i+k})$ is contained in
\begin{equation*}
 \bigg(B\big(x_{j_i},r_{j_i+k}+(C_{V,\lambda}\lambda)^{1/\theta}r_{j_i}\big)\setminus B(x_{j_i},r_{j_i+k})\bigg)\cup\bigg( B(y,r_{j_i+k}+(C_{V,\lambda}\lambda)^{1/\theta}r_{0})\setminus B(y,r_{j_i+k})\bigg).
\end{equation*}%If $[(\tilde{c}(\beta-\alpha))^{1/\theta_0}r_{j_i}]<1$, then $m(B(x_{j_i},r)\triangle B(y,r))=0$.
%Let $[\cdot]$ be the integer function. Without loss of generality, we assume that $[(\tilde{c}(\beta-\alpha))^{1/\theta_0}r_{j_i}]\ge 1$, then
Using the inequality~\eqref{shell} and the fact that the measure $m$ is invariant under the translation, the above inequality implies
\begin{equation*}
\begin{split}
   m(B(x_{j_i},r_{j_i+k})\triangle B(y,r_{j_i+k}))&\le m\bigg(B(x_{j_i},r_{j_i+k}+(C_{V,\lambda}\lambda)^{1/\theta}r_{j_i})\setminus B(x_{j_i},r_{j_i+k})\bigg)\\
&+m\bigg(B(y,r_{j_i+k}+(C_{V,\lambda}\lambda)^{1/\theta}r_{j_i})\setminus B(y,r_{j_i+k})\bigg)\\
&\le 2c_V\bigg(\frac{(C_{V,\lambda}\lambda)^{1/\theta}r_{j_i}}{r_{j_i+k}}\bigg)^{\theta}m(B_{r_{j_i+k}})\\
&\le \frac{\lambda}{4}m(B_{r_{j_i+k}}).
\end{split}
\end{equation*}
Combining the above inequality with $\|F\|_{L^{\8}}\le 1$, one has
\begin{equation}\label{controlled by averaging operator}
 \begin{split}
    \bigg|\int_{B(y,r_{j_i+k})}F(z)dm(z)-\int_{B(x_{j_i},r_{j_i+k})}F(z)dm(z)\bigg|&=\bigg|\int_{B(x_{j_i},r_{j_i+k})\triangle B(y,r_{j_i+k})}F(z)dm(z)\bigg|\\
&\le \frac{\lambda}{4}m(B_{r_{j_i+k}}).
 \end{split}
\end{equation}
It follows that
\begin{equation*}
  |A^\prime_{r_{j_i+k}}F(x_{j_i})-A^\prime_{r_{j_i+k}}F(y)|\le \frac{\lambda}{4},~\forall~0\le k\le q.
\end{equation*}
Using the triangle inequality, for every $1\le k\le q$, we have
\begin{align*}
 |A^\prime_{r_{j_i+k}}F(x_{j_i})-A^\prime_{r_{j_i+k-1}}F(x_{j_i})|&\le |A^\prime_{r_{j_i+k}}F(y)-A^\prime_{r_{j_i+k-1}}F(y)|+|A^\prime_{r_{j_i+k}}F(x_{j_i})-A^\prime_{r_{j_i+k}}F(y)|\\
 &+|A^\prime_{r_{j_i+k-1}}F(x_{j_i})-
A^\prime_{r_{j_i+k-1}}F(y)|\\
&\le |A^\prime_{r_{j_i+k}}F(y)-A^\prime_{r_{j_i+k-1}}F(y)|+\lambda/2.%-\frac{\lambda}{2}=\lambda/2,
\end{align*}
By the above inequality, we have $|A^\prime_{r_{j_i+k}}F(y)-A^\prime_{r_{j_i+k-1}}F(y)|>\lambda/2$, and the claim is proved.

\end{proof}%Following the arguments of~\cite{JKRW98},

%The proof of the following lemma follows from the argument in \cite[inequality (5.7)]{JKRW98} if we replace rectangles by balls in the argument. We omit the details.
%We also need the following lemma.%which is similar to \cite[inequality (5.7)]{JKRW98}.
\begin{lemma}\label{comparing lemma}
For positive integers $p$ and $q$, one has
\begin{equation*}
  m(\mathcal{G}_{(p+1)q}(F,\lambda))\le  3^{D_G}C^2_V\Phi(q)m(\mathcal{G}_{pq}(F,\lambda)).
\end{equation*}
\end{lemma}
The proof of this lemma is inspired by the proof of \cite[inequality (5.7)]{JKRW98}.
\begin{proof}
For every ball $B=B(x,r)\in \mathcal{G}^\prime_{(p+1)q}$, by the definition of $\mathcal{G}^\prime_{(p+1)q}$, there exists a sequence $r=r_0<r_1<\cdots<r_{(p+1)q}$ such that
\begin{equation*}
  |A^\prime_{r_k}F(x)-A^\prime_{r_{k-1}}F(x)|>\lambda,~\forall~1\le k\le (p+1)q.
\end{equation*}
So we have $B(x,r_{q})\in\mathcal{G}^\prime_{pq}$. Write $\widetilde{B}=B(x,r_{q})$. Set $\mathcal{B}=\mathcal{G}^\prime_{(p+1)q}$ and
\begin{equation*}
  \mathcal{B}^\prime=\{B^\prime:B^\prime\in\mathcal{G}^\prime_{pq}~satisfies~B^\prime=\widetilde{B}~for~some~B\in\mathcal{B}\}.
\end{equation*}
Note that $\mathcal{B}^\prime\subseteq \mathcal{G}^\prime_{pq}$, then by the definition of $\mathcal{G}_{(p+1)q}$,  the proof is finished if we show
\begin{equation}\label{measure control}
  m(\cup_{B\in\mathcal{B}}B)\le  3^{D_G}C^2_V\Phi(q)m(\cup_{B^\prime\in\mathcal{B}^\prime}B^\prime).
\end{equation}

We now focus on the above inequality.  Before proving this estimate, we introduce some new notations. We assume that $B_1$ is the maximal size ball of $\mathcal{B}^\prime$. We set $\mathcal{B}_1=\mathcal{B}$, $\mathcal{B}^\prime_1=\mathcal{B}^\prime$ and define the sets
\begin{align*}
  &\mathcal{I}_1=\{B|B\in\mathcal{B}_1, \widetilde{B}\cap B_1\neq \emptyset\},\\
&\mathcal{I}^\prime_1=\{B^\prime|B^\prime\in\mathcal{B}^\prime_1,B^\prime\cap B_1\neq \emptyset\}.
\end{align*}
We first prove the following estimate
\begin{equation}\label{measure control-1}
  m(\cup_{B\in\mathcal{I}_1}B)\le \Phi(q)m(3B_1),
\end{equation}
where $3B_1$ denotes the ball with the same center as $B_1$ and its radius is $3$ times that of $B_1$.

Fix $B=B(x,r)\in\mathcal{I}_1$. Since $B(x,r)\in\mathcal{G}_{(p+1)q}^\prime$, then there exists $r=r_0<r_1<\cdots<r_{(p+1)q}$ such that for all $1\le k\le (p+1)q$,
 \begin{equation*}
  |A^\prime_{r_k}F(x)-A^\prime_{r_{k-1}}F(x)|>\lambda.
\end{equation*}
 Since $B_1$ is maximal size ball of $\mathcal{B}^\prime$ and $B_1\cap B(x,r_{q})\neq\emptyset$, so for all $0\le k\le q$, $ B(x,r_k)\subseteq 3B_1$. It follows that%Note that $\widetilde{B}=B(x,r_{q})$ and
\begin{equation*}
  |A^\prime_{r_k}(F\mathds{1}_{3B_1})(x)-A^\prime_{r_{k-1}}(F\mathds{1}_{3B_1})(x)|>\lambda,~\forall~1\le k\le q.
\end{equation*}
Hence $B\in\mathcal{G}^\prime_q(F\mathds{1}_{3B_1},\lambda)$. Combining this observation with Lemma~\ref{measure estimate of Gq} and~\eqref{ball-ineq1}, we have
\begin{align*}
 m(\cup_{B\in\mathcal{I}_1}B)&\le m\big(\cup_{B\in\mathcal{G}^\prime_q(F\mathds{1}_{3B_1},\lambda)} B\big)= m({\mathcal{G}_q(F\mathds{1}_{3B_1},\lambda)})\le \Phi(q)\|F\mathds{1}_{3B_1}\|^p_{L^p(G,m)}\\
&\le \Phi(q)m(3B_1)\le 3^{D_G}C^2_V\Phi(q)m(B_1).
\end{align*}
So~\eqref{measure control-1} is proved.

Let $\mathcal{B}_2=\mathcal{B}_1\setminus\mathcal{I}_1$ and $\mathcal{B}^\prime_2=\mathcal{B}^\prime_1\setminus\mathcal{I}^\prime_1$. Select $B_2\in\mathcal{B}^\prime_2$ is a maximal size ball of $\mathcal{B}^\prime_2$. Note that $B_1\in \mathcal{I}^\prime_1$, then $B_2$ is disjoint from $B_1$. We define the sets
\begin{align*}
  &\mathcal{I}_2=\{B|B\in\mathcal{B}_2, \widetilde{B}\cap B_2\neq \emptyset\},\\
&\mathcal{I}^\prime_2=\{B^\prime|B^\prime\in\mathcal{B}^\prime_2,B^\prime\cap B_2\neq \emptyset\}.
\end{align*}
By similar discussions of~\eqref{measure control-1}, we have
\begin{equation*}
 m(\cup_{B\in\mathcal{I}_2}B)\le \Phi(q)m(3B_2)\le  3^{D_G}C^2_V\Phi(q)m(B_2).
\end{equation*}
Repeating the above process, then we can select the pairwise disjoint balls $B_1,~B_2,~\cdots$ which belongs to $\mathcal{B}^\prime$ and the sets $\mathcal{I}_1,~\mathcal{I}_2,\cdots$, $\mathcal{I}^\prime_1,~\mathcal{I}^\prime_2,\cdots$ with the properties
\begin{equation*}
  \cup_i\mathcal{I}_i=\mathcal{B},~\cup_i\mathcal{I}^\prime_i=\mathcal{B}^\prime,~m(\cup_{B\in\mathcal{I}_i}B)\le 3^{D_G}C^2_V\Phi(q)m(B_i).
\end{equation*}
Summing $i$ for the latter inequality and using property that the balls $\{B_i\}_i$ are pairwise disjoint, we have
\begin{align*}
m(\cup_{B\in\mathcal{B}}B)&=m(\cup_i\cup_{B\in\mathcal{I}_i}B)\le \sum_im(\cup_{B\in\mathcal{I}_i}B)\le 3^{D_G}C^2_V\Phi(q)\sum_im(B_i)\\
&= 3^{D_G}C^2_V\Phi(q)m(\cup_iB_i)\le 3^{D_G}C^2_Vm(\cup_{B^\prime\in\mathcal{B}^\prime}B^\prime).%3^{D_G}C^2_V\Phi(q)m(\cup_i\mathcal{I}^\prime_i)=
\end{align*}
We obtained~\eqref{measure control}, and the lemma follows.
\end{proof}

%We now prove Theorem~\ref{upcrossings estimate}.
\begin{proof}[Proof of Theorem~\ref{exponential-estimate}]
By Lemma~\ref{flu}, it suffices for our purpose to show that for any $\lambda>0$, there exist two constants $\tilde{c}_{1}$ and $\tilde{c}_{2}\in(0,1)$ such that for any $n\in\mathbb N$ and $F\in L^p(G,m)$ with $\|F\|_{L^\8}\le 1$,
 \begin{equation*}%\label{jump-ineq}
   m\big(\{g\in G:\mathcal{N}_{\lambda}(\mathbf{A}^\prime F)(g)>n\}\big)\le \tilde{c}_{1}\tilde{c}_{2}^n\|F\|^p_{L^p(G,m)}.
\end{equation*}
From now on, fix one $\lambda>0$ and one $F\in L^p(G,m)$ with $\|F\|_{L^\8}\le 1$. First, by Lemma~\ref{measure estimate of Gq}, we set $q_0=\min\{q\in\mathbb{N}:\Phi(q)\le \frac{1}{2}\}$.
For each $n>0$, compared with  $q_0$, we divide the $n$ into two cases: $n\ge q_0$ and $1\le n<q_0$.

We first consider the case $1\le n<q_0$. By Lemma~\ref{averaging-oper}, we can set $\tilde{c}_1={2c_p}/{\lambda^p}$, $\tilde{c}_2=\big({1}/{2}\big)^{{1}/{q_0}}$.
It remains to consider the case $n\ge q_0$. Write $n=sq_0+r$ with $0\le r<q_0$. Using Lemma~\ref{comparing lemma}, we have
\begin{equation*}
  m(\mathcal{F}_n)\le m(\mathcal{F}_{sq_0})\le m(\mathcal{G}_{sq_0})\le \bigg(\frac{1}{2}\bigg)^{s-1}m(\mathcal{G}_{q_0}).
\end{equation*}
On the other hand, using Lemma~\ref{measure estimate of Gq} again, we have
\begin{equation*}
  m(\mathcal{G}_{q_0})\le\Phi(q_0)\|F\|^p_{L^p}\le \frac{1}{2}\|F\|^p_{L^p}.
\end{equation*}
Note that $s=(n-r)/q_0$, by the above discussions we have
\begin{equation*}
  m(\mathcal{F}_n)\le\bigg(\frac{1}{2}\bigg)^{-r/q_0}\bigg(\frac{1}{2}\bigg)^{n/q_0}\|F\|^p_{L^p},
\end{equation*}
 and so we can set $\tilde{c}_1=2$ and $\tilde{c}_{2}=({1}/{2})^{1/q_0}$ in this case.

Finally, we can take
$$\tilde{c}_1=\max\{2, 2c_p/{\lambda}^p\},~\tilde{c}_2=({1}/{2})^{{1}/{q_0}},$$
and the proof is complete. %and the theorem is proved.
%Combing the above two lemmas with the arguments by~\cite{JKRW98}, we can complete the proof. For more details we refer the reader to~\cite[Theorem 5.3]{JKRW98}.
%The proof of Theorem~\ref{transfer principle for upcroosing} is based on the following lemma.
\end{proof}
\bigskip

%\begin{que}
%{\color{red}Let $G$ be a polynomial growth group endowed with a left-invariant metric $d$ and a Haar measure $m$. Assume that $(G,d,m)$ satisfies the  $(\epsilon,r_0)$-annular decay property, then Theorem~\ref{exponential-estimate} holds for $(G,d,m)$.}

%Then there exists one $r_0>0$ such that $(G,d,m)$ satisfies the  $(\epsilon,r_0)$-annular decay property.%Does  there exists a $r_0>0$ such that $(G,m,d)$ satisfies the $(\epsilon,r_0)$-annular decay property?
%\end{que}

%%%%%%%%%%%%%%%%%%%%%%%%%%%%%%%%%%%%%%%%%%%%%%%%%%%%%%%%%%%%%%%%%%%%%%%%%%%%%%%%%%%%%%%%%%%%%%%%%%%%%%%%%%%%%%%%%%%%%%%%%%%%%%%
%%%%%%%%%%%%%%%%%%%%%%%%%%%%%%%%%%%%%%%%%%%%%%%%%%%%%%%%%%%%%%%%%%%%%%%%%%%%%%%%%%%%%%%%%%%%%%%%%%%%%%%%%%%%%%%%%%%%%%%%%%%%%%%

\end{document}